\documentclass[aos, preprint]{imsart}

\RequirePackage{amsthm,amsmath,amsfonts,amssymb}
\RequirePackage[numbers,sort&compress]{natbib}
\RequirePackage[colorlinks,citecolor=blue,urlcolor=blue]{hyperref}
\RequirePackage{graphicx}
\usepackage{color}
 
\usepackage[colorlinks]{hyperref}

\usepackage{amsthm, amsmath, mathtools} 
\usepackage{amsfonts, amssymb}
\usepackage{graphicx, color, latexsym}
\usepackage{enumitem}
\usepackage{booktabs}

\usepackage{eufrak}

\usepackage{tikz}
\usepackage{tikz-qtree}

\startlocaldefs
\theoremstyle{plain}

\newtheorem{theorem}{Theorem}[section]
\newtheorem{lemma}[theorem]{Lemma}
\newtheorem{thm}{Theorem}

\theoremstyle{definition}

{}

\def\1{1\!{\rm l}}
\newcommand{\mbf}{\mathbf}
\newcommand{\ind}[1]{\mbf{1}\{ #1 \}}

\newcommand{\leqa}{\lesssim}
\newcommand{\geqa}{\gtrsim}

\newcommand{\EM}{\ensuremath}

\newcommand{\al}{\alpha}

\newcommand{\ga}{\gamma}

\newcommand{\la}{\lambda}
\newcommand{\La}{\Lambda}

\newcommand{\te}{\theta}
\newcommand{\ta}{\tau}
\newcommand{\veps}{\varepsilon}
\newcommand{\vphi}{\varphi}

\newcommand{\cA}{\EM{\mathcal{A}}}
\newcommand{\cB}{\EM{\mathcal{B}}}
\newcommand{\cC}{\EM{\mathcal{C}}}

\newcommand{\cF}{\EM{\mathcal{F}}}

\newcommand{\cI}{\EM{\mathcal{I}}}
\newcommand{\cJ}{\EM{\mathcal{J}}}

\newcommand{\cL}{\EM{\mathcal{L}}}
\newcommand{\cM}{\EM{\mathcal{M}}}
\newcommand{\cN}{\EM{\mathcal{N}}}
\newcommand{\cP}{\EM{\mathcal{P}}}

\newcommand{\cR}{\EM{\mathcal{R}}}
\newcommand{\cS}{\EM{\mathcal{S}}}
\newcommand{\cT}{\EM{\mathcal{T}}}

\definecolor{blendedblue}{rgb}{0.2,0.2,0.7}

\DeclareMathAlphabet{\mathpzc}{OT1}{pzc}{m}{it}

\newcommand{\RR}{\mathbb{R}}

\newcommand{\given}{\,|\,}

\newtheorem{corollary}{Corollary}

\newtheorem{proposition}{Proposition}

\DeclareMathOperator{\FDR}{FDR}
\DeclareMathOperator{\FNR}{FNR}
\newcommand{\fr}{\cR}

\newcommand{\mockalph}[1]{}

\newcommand{\bi}{\begin{enumerate}[label=\roman*)]}
	\newcommand{\ei}{\end{enumerate}}
\newcommand{\ba}{\begin{array}{rcl}}
	\newcommand{\ea}{\end{array}}

\usepackage{bm} 
\newcommand{\bb}{\bm{b}}

\newcommand{\ol}{\overline}

\graphicspath{{./Figures/}}

\endlocaldefs

\begin{document}
	
	\begin{frontmatter}
		\title{Multiple testing with the Horseshoe}
		\runtitle{Multiple testing with the Horseshoe}
		
		\begin{aug}
			\author[A]{\fnms{Sayantan}~\snm{Banerjee} \ead[label=e1]{sayantanb@iimidr.ac.in}}
			\author[B]{\fnms{Isma\"el}~\snm{Castillo}\ead[label=e2]{ismael.castillo@upmc.fr}}
			\author[B]{\fnms{Fanny}~\snm{Villers}\ead[label=e3]{fanny.villers@upmc.fr}}
			\address[A]{Operations Management \& Quantitative Techniques Area, Indian Institute of Management Indore, Madhya Pradesh 453556, India\printead[presep={ ,\ }]{e1}}
			
			\address[B]{Laboratoire de Probabilité, Statistique et Mod\'elisation (LPSM), Sorbonne University, 4, place Jussieu, 75005 Paris, France \printead[presep={,\ }]{e2,e3}}
		\end{aug}
		
		\begin{abstract}
			We study multiple testing under continuous global--local shrinkage priors, with a focus on the horseshoe prior in high-dimensional sparse settings. While such priors provide adaptive shrinkage and computational scalability, they do not induce exact zeros and hence do not directly yield posterior inclusion probabilities, making principled false discovery control nontrivial. We propose posterior--based decision rules for signal detection that are applicable across a broad class of continuous shrinkage priors and are calibrated to control the false discovery rate (FDR) while retaining high power. In the sparse normal means model, we show that the proposed procedures attain the optimal detection boundary and achieve frequentist asymptotic control of both FDR and false negative rate (FNR). The method is readily implementable via standard posterior sampling, and empirical studies indicate that the realised FDR and FNR closely track their theoretical targets. Applications to high-dimensional regression and Gaussian graphical models further illustrate the scope and practical effectiveness of the approach.
		\end{abstract}
		
		
		
	\end{frontmatter}


\tableofcontents 
\addcontentsline{toc}{section}{Table of contents}

\section{Introduction} \label{sec:intro}

Multiple testing methods are commonplace in many scientific areas such as genomics, astrophysics, network science, among many others. A main aim is to build a testing procedure that takes into account the multiplicity, while controlling in a best possible way the type I and type II errors (the errors under the null and under the alternative, respectively). Following a natural approach in statistical testing, a main research direction consists in finding procedures that at least already guarantee a form of type I error control. The False Discovery Rate (FDR), introduced in the seminal paper \cite{bh95}, is probably the most broadly used error of the first type; the Benjamini-Hochberg procedure famously controls it under independence, and many works since then have focused on extending this type of guarantees in more general settings, including ones with dependence (two popular options are empirical Bayes, as described in more details below, and knock--off methods,  e.g. \cite{bc15,candesetal18}). 

Taking a Bayesian point of view, one may place a prior distribution on the parameter of interest that, for instance, favours sparsity as a way to handle multiplicity. One may then use decision theory to suggest testing procedures: assuming that the data has truly been generated  from the Bayesian model, Bayes estimators (or tests, classifiers etc.) are then by definition optimal in terms of the loss function with respect to which they have been defined. Of course, in practice it is quite uncommon for one to be certain about the validity of the exact Bayesian model. But the latter reasoning suggests that, even if one does not know the model or prior, it could possibly be estimated from the data. This idea is at the heart of {\em empirical Bayes} (e.g. Robbins \cite{robbins56}), an approach made popular in the context of multiple testing in particular by Bradley Efron and co-authors \cite{EfronEtAl2001} and that we briefly describe now.

To illustrate this approach of multiple testing, suppose we are in the celebrated two-groups model of Efron:  that is, the observations are of the form `signal plus noise', and the signal is itself drawn randomly from one of two groups: either it equals $0$ (first group, the null) with some probability, or otherwise it is drawn from some absolutely continuous distribution (second group, the alternative). One can view this as drawing the signal from a spike--and--slab prior distribution; both the prior probability assigned to the null and the slab distribution are unknown, and can then, for instance, be estimated from the data. Natural multiple testing procedures are then the ones suggested by decision theory, with parameters of the prior estimated by empirical Bayes. Obtaining theory in these settings may require some effort (as one needs to understand how empirical Bayes estimates behave), but when available, such approach often comes with fairly strong guarantees. For instance, this path was developed for dependent (Hidden Markov-type) data by Wenguang Sun and Tony Cai in a series of papers \cite{SunCai2007, SC09} (see also \cite{acg22} for nonparametric HMMs and \cite{sunspatial15} for spatial data), where near-optimality of such Bayesian multiple testing (e.g. near-optimal power given type I error control) is derived. However, for such guarantees to hold, the {\em a priori} `model' on the signal (HMM, in this case) needs to hold, and theory for a fixed, non-random, signal is, to the best of our knowledge, not available to date.  

In order to investigate the behaviour of Bayesian procedures in a prior-independent way, a common approach is to undertake a frequentist analysis of the posterior distribution, using the Bayesian model to form the posterior, but then analysing the latter under the usual frequentist assumption that the data $X$ has been generated from some distribution $P_{\te_0}$, for a true fixed (non-random) parameter $\te_0$ (see e.g. \cite{gv17, csf24} for overviews on this approach). In recent years, numerous works have studied such behaviour for Bayesian procedures in high-dimensional models (see for instance \cite{banerjee2021bayesian} for a review), mostly from the {\em estimation} point of view. Perhaps the most natural way to model sparsity, especially in view of {\em testing}, is to use a spike--and--slab prior (e.g. \cite{gm93, cv12}) on the parameter $\te\in\RR^p$, with the spike ensuring the presence of exact $0$'s in the posterior; another advantage is that one may then use the posterior inclusion probabilities $\Pi[\te_i=0\given X]$ (which we call $\ell$--values) to test for the null hypothesis $\te_i=0$. Very recently, frequentist guarantees on several Bayesian multiple testing procedures based on spike and slab priors have been obtained: \cite{cr20} showed that $\ell$--value based procedures with these priors have a vanishing FDR over sparse vectors, and that corresponding $q$--values (\cite{Storey2003}) have an FDR close to the target level for strong signals (see also \cite{acr22} for related results). In \cite{acr24}, sharp minimax rates of testing for the FDR+FNR risk  (where the False Negative Rate, or FNR, is used as type II error) are investigated and again the (empirical Bayes) $\ell$--value procedure can be shown to be sharp minimax.    
  
While spike--and--slab priors, and associated $\ell$--value procedures, can be seen as theoretical ideals, the computations of corresponding posterior distributions can be challenging -- there has been steady progress over recent years though, including \cite{ves21, montanari}, but  high dimensionality (above one thousand predictors) remains a challenge. Among more scalable alternatives, we discuss two options: using continuous shrinkage priors (the approach considered thereafter in the paper) and using approximations such as Variational Bayes (VB). The appeal of VB methods in sparse settings is that they scale comparably to other optimisation procedures such as the LASSO: the package {\tt sparsevb} \cite{sparsevb} provides samples from VB posteriors in high-dimensional regression using spike--and--slab for the prior and variational class. Posterior convergence rates of such methods are now available \cite{rayszabo22}, but very little is known so far on their variable-selection properties, even assuming  strong signal conditions.   

The approach that we follow here is based on {\em continuous shrinkage priors} -- which we now review briefly.  The general idea is to define a density over the real line that has both heavy tails (to detect signals) and puts a lot of mass near zero (to enforce approximate sparsity and avoid overfitting). Global--local Gaussian scale mixtures set
\begin{equation}
	\label{eq:gl}
	\theta_i \mid \lambda_i,\tau \sim \cN(0,\tau^2\lambda_i^2),\qquad 
	\lambda_i \sim \pi_\lambda,
\end{equation}
where $\lambda_i$ is a local scale parameter and $\tau$ is a global scale controlling overall sparsity. 
The horseshoe prior \citep{CarvalhoPolsonScott2010} is obtained by taking half--Cauchy distributions for both local and global scales $\lambda_i\sim\mathrm{C}^+(0,1)$ and $\tau\sim\mathrm{C}^+(0,1)$. For sparse sequences, posterior contraction rates and near-minimax $\ell_2$ risk properties were established in \cite{vdp2014hs}, while adaptive behavior under empirical or hierarchical Bayes choices of $\tau$ was analyzed in \cite{vsv17}. Uncertainty quantification and marginal credible sets under the horseshoe were further studied in \cite{vsvuq17}, clarifying regimes in which posterior summaries are well calibrated from a frequentist perspective. Related extensions and alternatives include the horseshoe+ \citep{BhadraEtAl2017HSplus}, Dirichlet--Laplace priors \citep{BhattacharyaEtAl2015DL}, and other global--local constructions with tunable tail behavior \citep{GriffinBrown2010NG,ArmaganDunsonLee2013GDP}. 

In terms of algorithmical properties, one main advantage of shrinkage-type priors over spike--and--slab priors is that one avoids the need to sample from the posterior over all possible models (that is, indicators of nonzero variables), which is often costly, due to the combinatorial nature of the discrete prior employed. Since with continuous priors there is no discrete posterior to sample from, this results in significantly faster sampling using MCMC. As an example, in high-dimensional linear regression, the recent R package {\tt Mhorseshoe} \cite{Mhorseshoepackage, Mhorseshoe} based on the algorithm introduced in \cite{johndrow2020}, produces posterior samples for a dimension of the vector $\te$ in the $10000$'s (see Section \ref{sec:simu}). Many recent contributions explore the use of such priors in a variety of high-dimensional settings, e.g. \cite{GHSCM, Gan19, li2019graphical, sagar2024precision, Chandra26}.

Because continuous shrinkage priors do not produce exact zeros however, variable selection and multiple testing require an additional decision rule. Common empirical approaches include thresholding posterior shrinkage factors or posterior means, declaring significance when marginal credible intervals exclude zero, or applying posterior tail probability rules. However, to the best of our knowledge, there is no theory yet backing up one method or the other. Also, from the methodological point of view, it is not so clear how to use the previous empirical rules to produce a testing procedure that would, say, control the FDR at a level close to a target level $t$. The main purpose of the present work is to fill this gap.

Let us now briefly describe the paper's main contributions
\begin{enumerate}
\item we provide a coherent and 
theoretically grounded way to do multiple testing with continuous shrinkage priors. To do so, 
we introduce three procedures called respectively the $s$--value, $S$--value and C$s$--value procedures. We focus on the horseshoe prior mainly but the method can be applied to any sparse prior;
\item the $s$--value procedure is shown, in the sparse sequence model, to reach the sharp minimax multiple testing risk FDR+FNR over (relevant classes of) sparse signals. We prove under a similar setting that the $S$--value procedures controls the FDR asymptotically at given target level $t$, while maintaining highest possible power;
\item we illustrate empirically the obtained theoretical results; we verify through an extended simulation study that the proposed procedures control the combined multiple testing risk and the FDR much beyond the setting considered for theory, both in terms of models and choice of sparse prior, suggesting its broad applicability and opening the door to future investigations (in particular, theory) in more complex models. 
\end{enumerate}
From the theoretical perspective, this work can be seen as a continuation of a line of work on the horseshoe \cite{CarvalhoPolsonScott2010, vdp2014hs, vsv17} and, in particular, \cite{vsvuq17}, combined with recent insights from Bayesian multiple testing \cite{cr20, acr22, acr24}. We leverage existing fast sampling algorithms with the horseshoe prior to achieve practical performance. 
The results also shed light on empirical Bayes, the method considered herein to have adaptation to sparsity; the empirical Bayes estimate of the horseshoe parameter is shown to concentrate very precisely (with a sharp constant) around the value required to have exact FDR control asymptotically. This goes beyond existing results in the literature, where typically  one-sided concentration only up to a constant was obtained; here we derive concentration with the precise constant, which is necessary to obtain the sharp optimal multiple testing risk with a FDR at a given target level.

{\em Outline of the paper.} Section \ref{sec:proc} introduces the setting, the horseshoe prior and the considered procedures. Theoretical results in the sequence model for the $s$--value and $S$--value procedures are presented in Section \ref{sec:main}. A simulation study covering three different high-dimensional models is presented in Section \ref{sec:simu}, while the brief discussion Section \ref{sec:disc}  puts the results into perspective. Proofs can be found in Section \ref{sec:proofs} and in the Supplementary material. The latter also contains additional simulation results.

\section{Multiple testing procedures for shrinkage priors} \label{sec:proc}

In most of this section for simplicity of presentation we assume observations are from the canonical sparse sequence model (Eq. \eqref{model} below) and that we equip $\te_i$'s with independent horseshoe priors $\Pi_\ta$  with parameter $\tau$ (Eq. \eqref{hs}). Our approach can be used more generally in other sparse high-dimensional models, as well as with any continuous shrinkage prior: this is discussed in Section \ref{subsec:gen}, and reflected in the simulation study in Section \ref{sec:simu}. The choice of the horseshoe hyperparameter $\ta$, which is typically data-driven or given itself a prior, is discussed of the end of this Section.

\subsection{Setting and notation}

{\em Sparse sequence model.} In the sparse normal means model, one observes a sequence $X=(X_1,\ldots,X_n)$
\begin{equation} \label{model}
X_i = \theta_i + \veps_i,
\end{equation}
for  $i=1,\ldots,n$, with $\theta=(\theta_1,\ldots,\theta_n) \in\RR^n$ and $\veps_1,\ldots,\veps_n$ independent $\cN(0,1)$ variables. Given $\te$, the distribution of $X$ is a product of Gaussians and is  denoted by $P_\te$. \\

\noindent {\em Prior and posterior distributions.} We take a Bayesian approach and put a prior distribution $\Pi$ on the unknown parameter $\theta$. The distribution $P_\te=\otimes_{i=1}^n \cN(\te_i,1)$ is then viewed as the conditional distribution of $X$ given $\theta$. The posterior distribution, denoted $\Pi[\cdot\given X]$ is the other conditional, that is the law $\theta\given X$.

In the special case of a prior distribution $\Pi=\otimes_{i=1}^n \Pi_i$ that makes the coordinates of $\te$ independent,  
the posterior distribution is also an independent product, and the marginal distribution on the $i$th coordinate only depends on $X_i$ and is denoted $\Pi_i[\cdot\given X_i]$. \\

\noindent {\em Frequentist analysis of posterior distributions.} We study the posterior distribution $\Pi[\cdot\given X]$ under the frequentist assumption that there exists a true deterministic parameter $\theta_0$, that is, we study this distribution in probability under $X\sim P_{\theta_0}$.\\

\noindent {\em Sparse vectors.} We assume that the true $\te_0$ is sparse, in that it belongs to
\[ \ell_0[s_n] = \left\{\, \te\in \RR^n,\ |\{i:\ \te_i\neq 0\}|\le s_n\, \right\}, \]
where $s_n\le n$ is called sparsity parameter and $|B|$ denotes the cardinality of the finite set $B$. For a vector $\te\in\RR^n$, we denote by $S_\te=\{i:\ \te_i \neq 0\}$ its support and set $S_0:=S_{\te_0}$. We work in the (classical) asymptotic setting $n,s_n\to\infty$ and assume, without restating it explicitly henceforth, that $s_n/n\to 0$.\\

\noindent {\em The multiple testing problem.} Suppose one wants to simultaneously test 
\[ H_{0i}:\ \te_i=0\qquad \text{against}\qquad H_{1i}:\ \te_i\neq 0,\qquad \text{for all }
1\leq i \leq n. \]
To do so, one considers a multiple testing procedure $\vphi(X)=(\vphi_i(X),1\le i\le n)$ taking values in $\{0,1\}^n$. 
The False Discovery Rate (FDR) of $\vphi$ at vector $\te$ is defined as
\[ \FDR(\te,\vphi) = E_{\te}\left[ \frac{\sum_{i=1}^n \ind{\te_{i}=0,\vphi_i=1}}{1\vee  \sum_{i=1}^n \ind{\vphi_i= 1}} \right].
\]
The $\FDR$ is the expected ratio of false positives (that is, false `discoveries') over the total number of discoveries. 
The False Negative Rate (FNR) of $\vphi$ at $\te$ is defined as
\[ 
\FNR(\te,\vphi) 
= E_{\te}\left[\frac{\sum_{i=1}^n \ind{\te_{i}\neq 0} (1-\vphi_i(X))}{1\vee \sum_{i=1}^n \ind{\te_{i}\neq 0}}\right].
\]
The associated \ {\em multiple testing risk}, denoted by $\fr$, is then
\[ \fr(\te,\vphi) :=  \FDR(\te,\vphi)+  \FNR(\te,\vphi). \]

{\em Notation.} We denote by $\phi$ the density of a standard normal variable and by $\ol{\Phi}(x)=\int_{-\infty}^x \phi(u)du$. Also $\phi_{\mu,\sigma^2}(\cdot)$ denotes the Gaussian density  with mean $\mu$ and variance $\sigma^2$.\\

\subsection{Sparse priors and existing approaches}

{\em Spike--and--slab (SAS) prior.} A celebrated choice of sparse prior consists in drawing
\[ \te_1 \,\sim\, (1-w)\delta_0 + w G,\]
for $w\in[0,1]$ a weight parameter and $G$ a {\em slab} distribution, often a Laplace or Cauchy distribution. This prior draws exact zeros with large probability $1-w$ if $w$ is small, which is typically assumed in sparse settings. When using a prior distribution, such as the spike--and--slab prior, that draws exact zeros, a natural way to perform model selection is based on the (non-)inclusion probabilities $\ell_i(X_i)=\Pi[\te_i=0\given X_i]$. The multiple testing rule that rejects the null for the $i$--th coordinate of $\te$ if $\ell_i(X_i)<t$ is particularly appealing: it is the Bayes classifier for a weighted classification risk. Its {\em Bayesian} FDR is bounded from above by $t$ (e.g. \cite{mpr04,cr20}). Its optimality properties for the $\fr$--risk and spike--and--slab priors are studied in \cite{acr24}.

{\em The Horseshoe prior.} The horseshoe prior  $\Pi_\ta$  \cite{CarvalhoPolsonScott2010}  with parameter $\ta>0$  is the prior distribution  on $\te_1$ induced by 
\begin{equation} \label{hs}
\begin{aligned} 
        \te_1 \given \lambda & \,\sim\, \cN(0,\lambda^2\tau^2), \\
    \lambda & \,\sim\, C^+(0,1),
\end{aligned}
\end{equation}
where $C^+(0,1)$ denotes the half--Cauchy distribution. It is a {\em continuous} (as opposed to discrete) scale mixture of normal distributions, in the sense that the mixing parameter $\la$ takes values in $\RR^+$. To gain intuition on the proposed procedures below, it will be helpful to consider also a discrete mixture.
 
{\em Spike--and--Slab LASSO (SSL) prior.} This prior \cite{rockovageorge17} replaces the Dirac mass at $0$ of the SAS prior by a peaked Laplace distribution, for some $\la_0, \la_1>0, w\in[0,1]$,
\begin{equation} \label{ssl}
 \te_1 \,\sim\, (1-w)G_0 + w G_1, 
\end{equation} 
where $G_0=\text{Lap}(\la_0)$ and $G_1=\text{Lap}(\la_1)$ (or another more heavy--tailed distribution \cite{cm18}) for $\la_0$ a suitably large inverse scale parameter and $\la_1$ a constant.

{\em Existing approaches.} 
For continuous shrinkage priors, inclusion probabilities are constant ($\Pi[\te_i=0\given X]=0$, since the prior is absolutely continuous), so one needs to proceed differently. Suppose that, working in the sequence model \eqref{model}, we endow each coordinate $\te_i$ with a continuous shrinkage prior, say the horseshoe $\Pi_\ta$ to fix ideas. A number of choices have been proposed in the literature to perform variable selection; two popular approaches are
\begin{enumerate}
\item {\em Comparison to the posterior mean}. One rejects the null 
$``\te_i=0"$ if the posterior mean $\bar\te_i(X_i)=E[\te_i\given X_i]=\int \te_i d\Pi_\ta(\te_i\given X_i)$ exceeds $X_i/2$, or more generally $\kappa\cdot X_i$, for some $\kappa\in(0,1)$, as considered in \cite{CarvalhoPolsonScott2010};
\item {\em Methods based on credible intervals}. Based on the posterior distribution, one constructs a credible interval $I(X)$ of level $1-\alpha$ for the coordinate $\te_i$, for some $\alpha>0$, that is $\Pi_\ta[\, \te_i\in I(X)\given X_i\,]=1-\al$. One rejects   the null 
$``\te_i=0"$ if $0$ does {\em not} belong to $I(X)$. One may choose credible intervals in different ways: e.g. using the $\al/2$-- and $(1-\al/2)$--quantiles of the posterior on $\te_i$ (two-sided case), as suggested in \cite{vsvuq17} or either the $\al$-- or $(1-\al)$--quantile (one-sided case) depending on the sign of $X_i$ (choosing the latter if $X_i>0$), as suggested for spike--and--slab priors in \cite{cr20}.
\end{enumerate}
However, for continuous shrinkage priors these choices are mostly empirical in that, to the best of our knowledge, they are not backed up by  guarantees that the multiple testing risk $\cR$ and/or the FDR either vanish or are controlled at a certain level for any of these procedures. So far, works in the literature have been mostly limited to the Bayes risk, such as \cite{datta2013}. Two exceptions are the work  \cite{salomond17}, who obtained preliminary results for large signal sizes, and the recent preprint  \cite{pauletal25_preprint}, that mostly assumes known sparsity $s_n$, with a simple empirical estimate of $\ta$ otherwise, but with limited adaptation--to--$s_n$ guarantees. 

One main aim of the present work is thus to find a quantity that would be a natural analogue of the $\ell$--value for spike--and--slab priors. If this can be achieved, then there is a natural roadmap to build procedures that achieve a target FDR level and is fully adaptive with respect to the unknown sparsity $s_n$,  following what has been explored in this direction for $\ell$--values.

\subsection{The $s$--value procedure}

We advocate the use of the following quantity. Recall that for simplicity in this section we work with the sequence model and independent horseshoe $\Pi_\ta$ priors on coordinates. Define the $s$--values $s_i(X)$ as
\begin{align} \label{sval}
 s_i(X) & = s(X_i;\ta), 
\end{align}
where the quantity $s(X_i;\ta)$ is defined from, for $x\wedge y=\min(x,y)$,
\begin{align}  \label{svalx}
 s(x;\ta) & = 2\left( \Pi_\ta(\te_1<0 \given X_1=x) \wedge  \Pi_\ta(\te_1>0 \given X_1=x)\right).
\end{align}
Note the presence of a factor $2$ in front of the minimum in the last display.  

{\em Multiple testing procedure based on $s$--values.} For any given level $t\in(0,1)$, let us define
\begin{equation} \label{svpro}
\vphi^s_t(X) = \1\left\{ s_i(X) < t \right\},\quad 1\le i\le n.
\end{equation}

This procedure is called $s$--value procedure; we prove in Section \ref{sec:main} that $\vphi^s$ is optimal in a number of senses in terms of the $\cR$--risk, provided the parameter $\tau$ is chosen in a well-suited data-driven way, as made precise in Section \ref{sec:ebhb} below. Beyond minimax optimality, its use in practice can have  the drawback that it is somewhat conservative; to circumvent this possible issue, we consider below two other procedures, the $S$-- and $Cs$--value procedures, that rather target a small FDR level $t$ and reach optimal boundaries for such FDR--targetting  procedures.

\subsection{Intuition behind the $s$--values: comparing to the SSL prior}

At this point it is not completely clear why the $s$--values for the horseshoe prior can act as a natural proxy for inclusion probabilities. In order to gain intuition on this, we make a comparison with the spike--and--slab LASSO  prior \eqref{ssl} of \cite{rockovageorge17}, with a sparsity parameter $w$ chosen equal to $\ta$ (in can indeed be seen that both parameters play a very similar role for such priors, see \cite{vsv17}). This prior, like $\Pi_\ta$,  puts zero mass on the event $\{\te_i=0\}$, but unlike $\Pi_\ta$ it is a discrete two point mixture, so the posterior probability to belong to the `spike'--part of the mixture is well-defined and is the natural proxy for the inclusion probability of SAS priors. 

More precisely, with a SSL prior with parameter $w$ on the vector $\theta$  in model \eqref{model}, it is easy to check that 
that the posterior distribution of $\theta$ given $X=x$ is a two point mixture
\begin{equation} \label{sslmix}
	\Pi_w[\cdot\given X] 
	\sim \bigotimes_{i=1}^n\, (1-w(X_i))G_{0,X_i}(\cdot) + w(X_i) G_{1,X_i}(\cdot),
\end{equation}
for some data--dependent continuous distributions $G_{k,X_i}$ and weights $w(X_i)$ (see  Appendix \ref{app_ssl} for details). In particular, $1-w(x)$, which has a simple explicit expression,  is the probability to belong to the `slab' part of the mixture and is the analogue of the $\ell$--value for a spike--and--slab prior. In Appendix \ref{app_ssl}, we give more intuition and compare quantitatively $1-w(x)$ for the choice of sparsity parameter $w=\ta$ and the $s$--value \eqref{svalx}, showing a very close agreement. 

On the left column of Figure \ref{fig:illus_ssl_hs}, we represent the two components  $G_{0,X}$ and $G_{1,X}$ in the SSL posterior distribution, along with their mixture \eqref{sslmix}.
For two values of $x \ge 0$, we represent in blue the posterior probability $\Pi(\te<0\given X=x)$. When $x=0$, $\Pi(\te<0\given X=0)=1/2$. The  $s$--value is defined as {\em twice} this probability: the reason for this is that the probability $\Pi(\te<0\given X=x)$ only catches (around) {\em half} of the mass corresponding to the `spike'  around 0. 
In particular, we see that the procedure for shrinkage priors has to be different compared to spike--and--slab priors: indeed  $\Pi(\te<0\given X=x)$ does not capture {\em any} mass of the spike part of the posterior, while $\Pi(\te\le 0\given X=x)$ already captures the {\em whole} mass of the spike part (which stands {\em exactly} at zero due to the Dirac mass). 
As it turns out, it can be seen that for continuous shrinkage priors the presence of the factor $2$ implies that the $s$--value procedure is equivalent to a simple {\em two--sided} procedure based on credible intervals, where one rejects if $0$ does not belong to the centered two-sided quantile credible interval of level $1-t$ (see Appendix \ref{app:linkmci} for a proof and more discussion). 
\begin{figure}[h]
\includegraphics[scale=0.7]{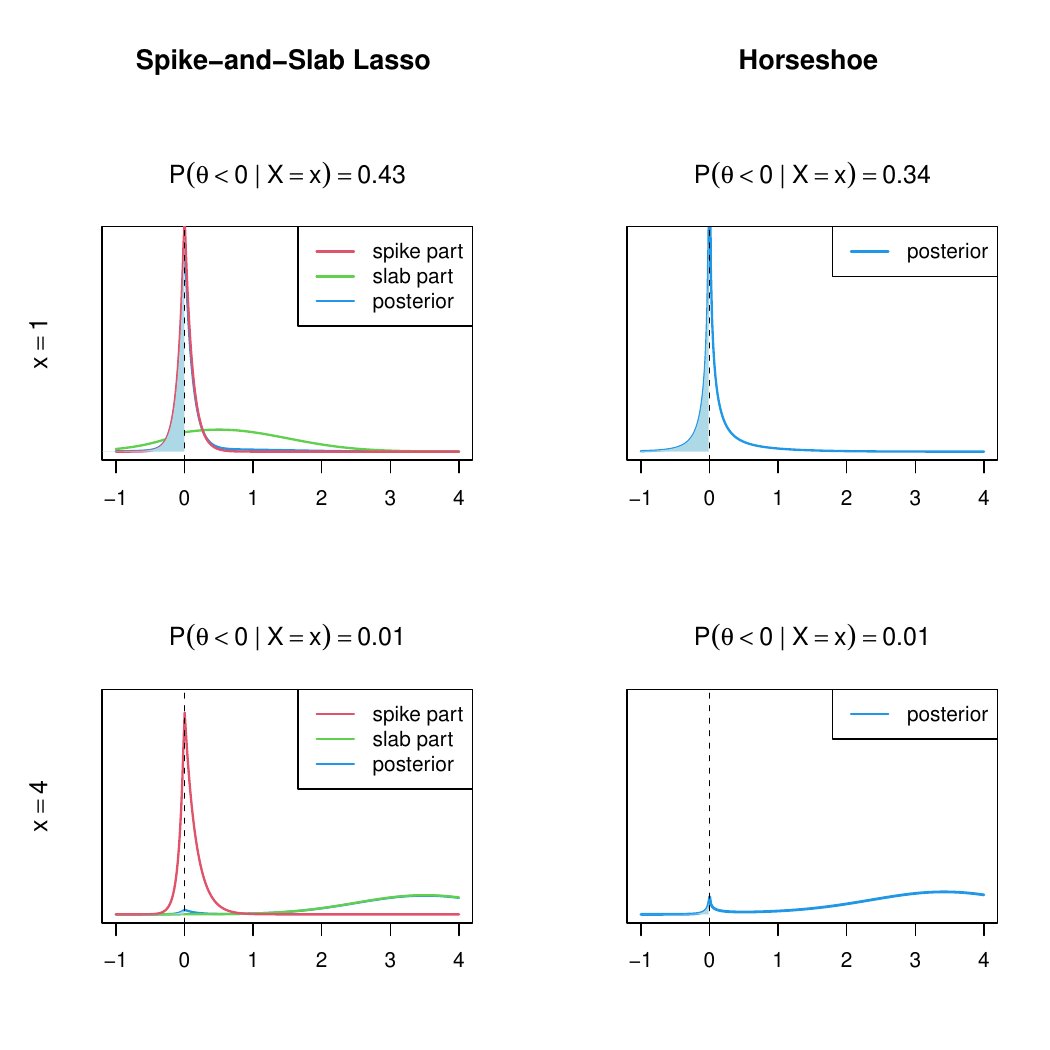}
\centering
	\caption{Comparison of posterior densities for the SSL prior (left column) and Horseshoe prior (right), for $x=1$ (top row) and $x=4$ (bottom) and $\la_0=10, \la_1=1/2, w=0.15$; the shaded area is half of the $s$--value.}  
	\label{fig:illus_ssl_hs}
\end{figure}

\subsection{The $S$--value and the C$s$--value}

While controlling both type I and type II errors simultaneously is certainly a desirable goal in general, given the asymmetrical meaning of both errors in practice, it is often of interest to seek for procedures that keep the type I error, here the FDR, close to a given target level $t\in(0,1)$. As we shall see below, it turns out that  the $s$--value procedure is not exactly adapted to this task, as its FDR vanishes.

{\em Case of SAS priors: $q$--value and C$\ell$--value procedure}. In the case of `model selection' priors with exact zeroes, decision theory can suggest procedures that control the Bayesian FDR. One prominent example is the $q$--value introduced by Storey \cite{Storey2003}, see also  \cite{efron08}: in the setting of the sequence model, with $q_i=q_i(X_i)$ and $q_i(u)=\Pi[\te_i=0\given |X_i|\ge |u|]$, this procedure rejects indices $i$'s for which $q_i$ is smaller than a target level $t$. This procedure has Bayesian FDR very close to the target level $t$ (see e.g. \cite{cr20}, Proposition 1).

Another procedure is based on ranking inclusion probabilities; for SAS priors, using the $\ell$--values $\ell_i=\ell_i(X_i)=\Pi[\te_i=0\given X_i]$, one can form a procedure $\vphi^{C\ell}$ \cite{mpr04, SunCai2007}, called C$\ell$--value procedure (as cumulative $\ell$--values) in \cite{acr22}, that rejects the first $\hat{k}$ smallest $\ell$--values, where $\hat{k}$ is the largest integer $k$ such that the average $(1/k)\sum_{i=1}^k \ell_{(i)}$ still stays below a target level $t\in(0,1)$.  Such procedure controls the Bayesian FDR at a level very close to the target $t$. 
 
{\em $S$--value.} We define, by analogy with the way the $q$--value is constructed from the inclusion probabilities, a $S$--value by replacing the conditioning on $X_1=x$ in \eqref{svalx} by $X_1\ge x$ and $X_1\le x$ depending on the sign of $x$ \begin{align}  \label{Svalx}
 S(x;\ta) & = 2\left( \Pi_\ta(\te_1<0 \given X_1\ge x) \wedge  \Pi_\ta(\te_1>0 \given X_1\le x)\right).
\end{align}
Note again the presence of a factor $2$ in front of the minimum in the last display.  

For any given level $t\in(0,1)$, let us define $S_i(X)=S(X_i;\ta)$ and 
\begin{equation} \label{Svpro}
\vphi^S_t(X) = \1\left\{ S_i(X) < t \right\},\quad 1\le i\le n.
\end{equation}

{\em C$s$--value.} Again by analogy with the procedure called C$\ell$--values in \cite{acr22} (that can be intuitively introduced as an $\ell$--value type procedure that maximises the posterior FDR, see \cite{acr22}), one defines the  $Cs$--value procedure, that rejects the $\hat{k}=\hat{k}_{C_s}$ smallest $s$--values, where $\hat{k}_{C_s}$ is defined as
	\begin{equation} \label{kchap}
	\hat{k} = \max\left\{1 \le k \le n:\ \ \frac{1}{k} \sum_{i=1}^k s_{(i)}(X) \le t \right\},
	\end{equation}
where $s_{(1)}(x) \le \ldots \le s_{(n)}(x)$ are the ordered elements of $\{s(x_i;\tau), 1\le i \le n\}$. That is,
\begin{equation}\label{Csval}	
	\vphi^{Cs}_t(X) = \1\left\{ s_i(X) \le s_{\large(\hat{k}\large)}(X) \right\},\quad 1\le i\le n.
\end{equation}

\subsection{More general models and shrinkage priors} \label{subsec:gen}

Now leaving the setting of the sequence model \eqref{model}, suppose we are given a model $\cP=\{P_\te^{(n)},\ \te\in\Theta\}$ and observations $Z$, where here $\Theta\subset\RR^p$ for some $p\ge 1$ (typically large in high-dimensional settings). Two examples considered in the simulations Section \ref{sec:simu} are high-dimensional linear regression and Gaussian graphical models. 
 We see below that both the $s$--value and $Cs$--value procedures defined above in the sequence model have natural generalisations.

Following a Bayesian approach, we put a prior $\Pi$ on the unknown parameter $\theta$ and  interpret the conditional distribution $Z\given \theta$ as $P_\te^{(n)}$. One can then form the posterior distribution, that is the conditional distribution $\te\given Z$, denoted $\Pi[\cdot\given Z]$,  and study it as above under a frequentist perspective by assuming that in reality $Z\sim P_{\te_0}^{(n)}$ for some fixed $\te_0\in \Theta$. Since we have sparse models as well as scalability of posterior sampling algorithms  in mind (and we will see indeed in Section \ref{sec:simu} that we can deal with genuinely high-dimensional models), one takes $\Pi$ to be a continuous shrinkage prior; beyond the horseshoe, one can think for instance of the spike--and--slab LASSO \cite{rockovageorge17}, other continuous or discrete scale mixtures of normal distributions \cite{cm18, salometal} or even choices that make the coordinates of $\te$ only conditionally independent, such as the Dirichlet--Laplace prior \cite{BhattacharyaEtAl2015DL}.

{\em The s--value and Cs--value procedures in general settings.} 
For each coordinate $i\in\{1,\ldots,p\}$ of the vector $\te$, one defines an $s$--value simply from the marginal distribution of $\te_i$ given the data $Z$
\begin{align}  \label{svalxgen}
 s_i(z) & = 2\left( \Pi(\te_i<0 \given Z=z) \wedge  \Pi(\te_i>0 \given Z=z)\right).
\end{align}
and one sets, again for any $i\in\{1,\ldots,p\}$,
\begin{align} \label{svalgen}
 s_i(Z) & = s_i(Z;\ta),\qquad \quad \vphi^s_t(Z)_i = \1\left\{ s_i(Z) < t \right\},
\end{align}
where we call $\vphi^s_t(Z)=(\vphi^s_t(Z)_1,\ldots,\vphi^s_t(Z)_p)$ the $s$--value multiple testing procedure in this more general framework. 
This time, coordinates of $\te$ are not necessarily independent, and  conditioning is with respect to the whole data $Z$; also, $s_i(z)$ may vary across indices $i$. 
Similar to the case of $s$--values in the sequence model, one may view the last display as a proxy for the $\ell$--value procedure, which here would reject the null, for a prior that puts mass to exact zeros (such as the spike--and--slab prior), whenever $\ell_i(z)=\Pi[\te_i=0\given Z]$ is smaller than $t$. The latter $\ell$--value procedure has similar decision-theoretical support as the $\ell$--value procedure in the sequence model: in particular, it controls the Bayesian FDR at level $t$.  
In turn, the $Cs$--value procedure is defined exactly in the same way as in \eqref{kchap}--\eqref{Csval}, replacing $X$ therein by $Z$ and using the general definition of $s_i(Z)$. It is inspired by the $C\ell$--value procedure, which in this more general context as well can be shown to have a Bayesian FDR of at most $t$ (and in fact very close to $t$). 

The behaviour of these procedures will be studied in the simulations Section \ref{sec:simu}.  

\subsection{Choice of hyperparameters and marginal maximum likelihood} \label{sec:ebhb}

Still for a moment in the general setting of the previous subsection, let us note that the prior distribution often crucially depends on one or more hyperparameter(s). Considering the case of the horseshoe with hyperparameter $\ta$ to fix ideas, a fully Bayesian way to choose $\ta$ is via hierarchical Bayes, by drawing it itself at random: that is, the prior scheme $\eqref{hs}$ is now viewed as the distribution $\te_1\given \ta$, and $\ta$ itself is drawn according to a prior distribution; the half--Cauchy prior on $\ta$ is a standard choice for the horseshoe.

Another popular way of choosing $\ta$ is via  empirical Bayes: one replaces $\ta$ by a data-driven choice $\hat\ta(Z)$. The choice we consider is the marginal maximum likelihood estimator (MMLE), which is obtained by forming the marginal likelihood of $\ta$, that is the marginal density of the data $Z$ (given $\ta$) in the Bayesian model, and maximising it with respect to $\ta$. Since this is the choice we consider below for the sequence model, let us focus on this setting.

{\em Empirical $s$-- and $S$--value procedures in the sequence model.}  Following \cite{vsv17}, we set, for $h_\ta$ the horseshoe density with parameter $\ta$, and in the setting of the sequence model \eqref{model},
\begin{equation} \label{defmmle}
\hat\ta = \hat\ta(X) = \,\underset{\ta\in[1/n,1]}{\text{argmax}}\ \prod_{i=1}^n \int \phi(X_i-\te)
h_\ta(\te)d\te.
\end{equation}
For such $\hat\ta$  data-driven estimate of $\ta$, and $t\in(0,1)$ a given level, we denote 
\begin{align}
\hat\vphi^{s}_t(X) & = \1\left\{ s(X;\hat\ta) < t \right\},  \label{svhat}\\
\hat\vphi^{S}_t(X) & = \1\left\{ S(X;\hat\ta) < t \right\}, \label{Svhat}
\end{align}
respectively the $s$-- and $S$--value multiple testing procedures based on `empirical' $s$--values. In the sequence model, given the relatively simple expression \eqref{defmmle}, one can argue that empirical Bayes is faster to implement, as there is only a one--dimensional function to optimise with respect to $\ta$. Note, however, that the integral in \eqref{defmmle} does not have an explicit form, so it needs to be approximated;  the R {\tt horseshoe} package \cite{horseshoe_package} does this by numerical integration. In order to carry out simulations in the high-dimensional settings considered in Section \ref{sec:simu}, we propose a new way to implement this approximation  that scales faster with the dimension $n$ (see Appendix \ref{app:seqmod}).

For simplicity in this paper for the theory part we shall restrict to the empirical Bayes approach to choose $\ta$. However, given the arguments in \cite{vsv17, vsvuq17} when studying the hierarchical Bayes (HB) version of the horseshoe prior, one can expect  the results in Section \ref{sec:main} to go through as well for the HB approach. Indeed, the proof for HB consists in showing that the posterior distribution of $\ta$ concentrates around the MMLE $\hat\ta$ in \eqref{defmmle}, when a half-Cauchy prior is chosen on $\ta$. More generally, often EB and HB can be expected to have similar behaviours, although this may not always be the case: in particular, for spike--and--slab priors with Laplace slabs, it is in fact shown in \cite{cm18} that EB achieves a slightly sub-optimal posterior convergence rate (and that this is not the case for Cauchy slabs, for which EB and HB behave similarly, similar to the horseshoe case). In complex models, the marginal likelihood may be more difficult to compute and optimise, and sampling from the hierarchical Bayes posterior is often more convenient: we follow this approach for linear regression and graphical models in Section \ref{sec:simu}.

\section{Main results}  \label{sec:main}

We now derive theoretical results for two of the multiple testing procedures introduced above, the $s$--value and $S$--value procedures, in the sequence model. We show that these procedures respectively achieve two information-theoretic boundaries: the $s$--value procedure achieves the optimal minimax combined multiple testing risk $\cR=$FDR+FNR; the $S$--value procedure is optimal among procedures that achieve a target FDR level asymptotically. We comment briefly on the $Cs$--value procedure 
in Section \ref{sec:mainSval}. 

\subsection{Frequentist optimality of multiple testing risks}

Optimality of arbitrary multiple testing procedures $\vphi$ with respect to the $\fr$--risk has recently been investigated in \cite{acr24}, with previous related results including \cite{fromontetal16} (family-wise error rates), \cite{rabinovich20} (for specific signal classes and procedures) and \cite{butucea18} (classification risk). It is shown in \cite{acr24} that the hardness of the multiple testing problem can be expressed in terms of the survival function of the noise, namely $\bar\Phi(u):=P[\cN(0,1)>u]$ for the Gaussian noise considered here. 

Let us introduce the set, for any given real number $b$,
\begin{equation} \label{defclb}
 \cL_0[s_n;b]= \bigg\{ \theta\in \ell_0[s_n] \::\: |\theta_{i}| \geq \: \sqrt{2\log \frac{n}{s_n}}  + b\ \text{ for all } i\in S_\te,\ |S_\te|=s_n \bigg\}. 
\end{equation}
For this class, the minimax multiple testing risk is
\begin{equation} \label{mmr}
 \inf_{\vphi} \sup_{\te\in\cL_0[s_n;b]} \fr(\te,\vphi) = \bar\Phi(b)+o(1), 
\end{equation}
where the infimum is over all possible procedures $\vphi$. 
To put this into perspective, and also for use later in the paper, it is also useful to consider the class of `boundary signals' in $\cL_0[s_n;b]$
\begin{equation} \label{defclbe}
 \cL_{0,=}[s_n;b]= \bigg\{ \theta\in \ell_0[s_n] \::\: |\theta_{i}| = \sqrt{2\log \frac{n}{s_n}}  + b\ \text{ for all } i\in S_\te,\ |S_\te|=s_n \bigg\}. 
\end{equation}
It can be checked that vectors in $\cL_{0,=}[s_n;b]$ are `least-favourable' in the sense that the supremum in \eqref{mmr} can be restricted to $\te\in\cL_{0,=}[s_n;b]$ without changing the bound. Sometimes for simplicity of exposition this class is considered in the literature instead of \eqref{defclb}, see e.g. \cite{rabinovich20}. 

Optimality can also be investigated in a more precise minimax--local sense, for {\em arbitrary} signals. Define, for $\bb=(b_1,\ldots, b_{s_n})$ a vector of real numbers such that $\sqrt{2 \log(n/s_n)} +b_j>0$ for all $j$, the class
\begin{align} \label{thetab} 
 \Theta_{\bb}  =
\Big\{ \theta\in \ell_0[s_n] \::\: 
 \exists\,  & 
 i_1, \ldots,  i_{s_n} \mbox{ all distinct, } \ 
 |\theta_{i_j}| \ge \sqrt{2 \log(n/s_n)} +b_j>0
 \Big\}.
\end{align}
Note that the union of all possible classes $ \Theta_{\bb} $ when $\bb$ varies gives the set of sparse vectors with exactly $s_n$ non--zero coordinates. Further denote
\begin{equation} \label{lambdab}
\Lambda_n (\bb)=s_n^{-1} \sum_{j =1 }^{s_n} \bar\Phi\left(b_j \right).
\end{equation}
Then it can be shown (\cite{acr24}) that 
\[  \inf_\vphi \sup_{\te\in\Theta_{\bb}}  \fr (\theta,\vphi)= \La_n(\bb) + o(1). \]
In this paper for simplicity of exposition we refrain from stating minimax--local results and focus on the simpler class \eqref{defclb}. We expect that minimax--local results can be obtained (albeit with somewhat more lengthy and technical proofs) in a similar way as in \cite{acr24}: we comment more precisely on this below each Theorem in the following two subsections.

\subsection{The $s$--value is optimal for the $\cR$--risk} \label{sec:mainsval}

Our first main result shows that the $s$--value multiple testing procedure is optimal in terms of the multiple testing risk $\fr=\FDR+\FNR$.
\begin{thm} \label{thm:hsrisk}
Let $b$ be a fixed, arbitrary, real number. Let 
$\hat\vphi_t^{s}$ denote the $s$--value procedure with data-driven choice $\hat{\ta}$ of the horseshoe parameter $\ta$ and fixed level $t\in(0,1)$ as in \eqref{defmmle}--\eqref{svhat}. Then, as $n, s_n\to \infty$,
\[ 
\sup_{\te_0 \in  \cL_0[s_n;b]} \ \left\{ 
\FDR(\te_0,\hat\vphi_t^{s}) + \FNR(\te_0,\hat\vphi_t^{s}) \right\}= \bar{\Phi}(b) + o(1).
\]
\end{thm}

\noindent {\em Remark (optimality).} Theorem \ref{thm:hsrisk} shows that the $s$--value procedure \eqref{svpro} reaches the optimal minimax testing risk over the class 
$\cL_0[s_n;b]$. In particular, it matches the behaviour of $\ell$--values for spike--and--slab priors in this setting (\cite{acr24}, Theorem 1). \\

\noindent {\em Remark (other signal shapes).} The above result is of minimax nature: it can be checked that the optimal constant $\bar\Phi(b)$ is attained for signals on the boundary (those in \eqref{defclbe}), and the behaviour of the risk can be better for certain signals within the class. A more general statement (albeit more technical to state and prove) of Theorem \ref{thm:hsrisk} would consist in replacing $\cL_0[s_n;b]$ in the supremum in the last display by $\Theta_{\bb}$ in \eqref{thetab} for some given (possibly $n$--dependent) vector $\bb$. Then $\bar{\Phi}(b)$ in the last statement would be replaced by $\Lambda_n (\bb)$ in \eqref{lambdab}.\\

This result is illustrated through simulations in Section \ref{sec:simu}, see Figure \ref{fig:all.Horseshoe.s10} (see also Appendix \ref{Appendix:sequencemodel} for an explicit comparison of procedures with the combined risk $\FDR + \FNR$).

The proof of Theorem \ref{thm:hsrisk} relies on several steps. As a first step, one proves that the $s$--value, which by definition is a ratio of two integrals with respect to a half-Cauchy density (numerator and denominator are both {\em continuous} mixtures), can be well approximated for most signals by a ratio of two simpler quantities, where the idea is to relate these to an expression similar to the one obtained for a {\em discrete} mixture prior, such as spike--and--slab or spike--and--slab LASSO. Since those are inequalities only, one cannot rely on a precise `threshold' of signal from which the method starts to reject (as was done for a spike--and--slab prior in \cite{cr20}). Rather, sufficiently precise inequalities on integrals with respect to the prior suffice; to do so, one partly relies on (and needs to refine) a number of bounds obtained in \cite{vsv17, vsvuq17}. A second major step consists in the study of the marginal maximal likelihood estimate $\hat\ta$, which is shown to concentrate {\em up to sufficiently large constants} around $s_n/n$ times a logarithmic factor. The corresponding upper--bound was obtained in  \cite{vsv17}; here we prove a lower--bound counterpart -- which is possible in our context given that the statement of the minimax testing rates guarantees by definition a certain amount of `separation' from the null hypotheses --.  This result will itself need to be much refined for Theorem \ref{thm:fdrcontrol} below (therein we need a concentration of $\hat\ta$ with a {\em matching} constant for both upper and lower-bound asymptotically). As a last step, one shows that both FDR and FNR (which are ratios) are controlled as desired, using the mentioned inequalities and concentration of both numerator and denominators in the ratios involved for the two types of errors.

\subsection{The $S$--value and target FDR control} \label{sec:mainSval}

 The next theorem shows that  the $S$--value procedure achieves a desired target FDR level while keeping the FNR as small as possible; its proof is more delicate and in order  to have a more transparent argument, we have focused on a class of signals of specific form, which corresponds to signals at the `boundary' of the class considered in Theorem \ref{thm:hsrisk}. As remarked below and similar to the comments made below Theorem  \ref{thm:hsrisk}, the results could be extended to cover almost arbitrary shapes of signals.

\begin{thm} \label{thm:fdrcontrol}
Let $b$ be a fixed, arbitrary, real number and fixed level $t\in(0,1/2]$. Let 
$\hat\vphi_t^{S}$ denote the $S$--value procedure  \eqref{Svhat} with data-driven choice $\hat{\ta}$ of $\ta$ as in \eqref{defmmle}. 
Then 
\[ 
\FDR(\te_0,\hat\vphi_t^S) \to t,
\qquad\quad
\FNR(\te_0,\hat\vphi_t^S) \to \bar\Phi(b),
\] 
as $n,s_n\to\infty$, where the convergence is uniform over any $\te_0 \in  \cL_{0,=}[s_n;b]$.
\end{thm}

\noindent {\em Remark (minimax optimality).} In view of the optimality results obtained in  \cite{acr24}, 
Theorem \ref{thm:fdrcontrol} shows that the asymptotic values of both FDR and FNR of the S--value procedure are optimal (at least when comparing to most `relevant' multiple-testing procedures that control the FDR at level $t$, see below).  This statement may sound surprising at first, since the $\cR=$FDR+FNR risk of $ \hat\vphi_t^S$ equals asymptotically $t+\bar\Phi(b)$ and so is strictly larger than $\bar\Phi(b)$. One would think  that, since the $S$--value procedure has a FDR of $t$ (and indeed, it can be shown to be less conservative than the $s$--value procedure) asymptotically, it could perhaps be able to compensate by a lower FNR in order to still reach a $\cR$ risk of $\bar\Phi(b)$. This is however impossible: Theorem 3 in \cite{acr24} indeed proves a lower bound of $\bar\Phi(b)$ on the FNR {\em only}, for all `sparsity preserving' procedures (that is, procedures whose number of discoveries does not overshoot the true sparsity $s_n$ by a certain multiplicative factor, a very reasonable requirement satisfied for all standard sparsity tests or procedures, including the LASSO estimator).   \\

\noindent{\em Practical choice of $t$.} From a practical point of view, if $t$ is a small number (such as $5\%$),  it may be acceptable to lose slightly in terms of overall $\cR$ risk and compensate in terms of finite-sample performance by having a less conservative procedure. Indeed, while the $s$--value procedure at level $T$ (we deliberately chose another notation to distinguish from $t$) is shown to reach the optimal $\cR$ risk asymptotically for any fixed level $T\in(0,1)$, this is true for any $T$, suggesting it is not in the `boundary' range of procedures that start to be less conservative. Also, its FDR has to go to zero (since the FNR has to be at least $\bar\Phi(b)$ for sparsity-preserving procedures, which can be shown to be the case for $\hat\vphi^s_t$). The simulations in Section \ref{sec:simu} show that choosing a small target level $t>0$ and the S--value procedure strikes a good balance between finite-sample performance and near-optimality of the $\cR$--risk. 
\\
 
\noindent {\em Remark (other signal shapes).} Similar to the case of Theorem \ref{thm:hsrisk}, although we have focused here on the class $\cL_{0,=}[s_n;b]$ for clarity of exposition, it is possible to derive results for other signals, including ones of different amplitudes. Inspection of the proof of Theorem \ref{thm:fdrcontrol} shows that, for a more general class $\Theta_{\bb, =}$ in \eqref{thetab} (that is, the same as $\Theta_{\bb}$, except with an equality instead of $\ge$ in the definition), the arguments go through as soon as the optimality quantity $\La_n(\bb)$ in \eqref{lambdab} has a limit as $n,s_n\to\infty$, a mild requirement satisfied for instance for all the signal examples considered in \cite{acr24}.\\
 
\noindent {\em Proof ingredients.} A main difficulty for the proof of Theorem \ref{thm:fdrcontrol} compared to that of  Theorem \ref{thm:hsrisk} is that it requires sharp concentration, including the exact asymptotic constant, of both the number of false positives (denoted FP) and true positives (TP below) of the $S$--value procedure. Let us focus for instance on the control of the FDR. Using concentration arguments, one can show
\[  \FDR(\te_0,\hat\vphi_t^S)  \, \sim \, \frac{E (FP)}{E (FP) + E(TP)}.
 \]
A main step is then to show that the MMLE $\hat\ta$ concentrates around a value $\ta^*$ that verifies
\[    \ta^* \sim \Phi(b)s_n\zeta_{s_n/n}/\{C^* (n-s_n)\},\] 
where $\zeta_{s_n/n}=\sqrt{2\log(n/s_n)}$ and $C^*=(2/\pi)^{3/2}$. 
Then further concentration arguments show
\begin{align}
 E (TP) & \sim  s_n\Phi(b) \sim C^* (n-s_n)\ta^*/\zeta_{s_n/n}, \label{Std}\\
 E (FP) & \sim C^* (n-s_n)r(\ta^*,t)/\zeta_{s_n/n}, \label{Sfd}
\end{align} 
with $r(\ta,t)=\ta t (1-\ta)^{-1} (1-t)^{-1}$. 
This in turn yields that 
$\FDR(\te_0,\hat\vphi_t^S)\sim t$ as desired. 
It is remarkable that the MMLE $\hat\ta$ strikes the precise balance that enables FDR control, even if the horseshoe prior, a continuous-shrinkage prior, does not model exact zeroes. 

Earlier work on frequentist FDR control of Bayesian multiple testing procedures has mostly focused on the $\cR$--risk, for which it is enough to have a rough bound on the MMLE up to constants, as is the case for Theorem \ref{thm:hsrisk}. In the few existing results with non--vanishing FDR,  the considered regime was one of `large signals' \cite{cr20, acr22}: namely, all non--zero $\te_i$ being well above $\sqrt{2\log{n/s_n}}$, which made it simpler to derive concentration of $\hat\ta$, as in this regime one has exact model selection, namely the procedure can recover all non-zero coordinates with high probability.

In this work we do not consider theory for the $Cs$--value procedure \eqref{Csval}. However, based on the results derived in \cite{acr22} for spike--and--slab priors, it can be expected that Theorem \ref{thm:fdrcontrol} carries over to the $Cs$--value procedure as well. Indeed, proofs in  \cite{acr22}, who obtain theory for $C\ell$--values for spike--and--slab priors (as well as results in the quite different setting of Hidden Markov Models in \cite{SC09, acg22}, where $C\ell$--values are also considered), suggest that the $C\ell$--value procedure can be viewed as an `empirical' version of the $q$--value procedure; since the $C\ell$--value procedure in \cite{acr22}, or the $Cs$--values procedures considered below in Section \ref{sec:simu}, also feature empirical Bayes choices of hyperparameters, they are in a sense `doubly empirical' and somewhat more technical to study theoretically, so we refrain from investigating theory for this procedure in this work; we do recommend it in practice though when targeting a specific FDR level, given excellent empirical performance illustrated in Section \ref{sec:simu}.\\

Finally, we expect the results above to carry over to other shrinkage priors such as the spike--and--slab LASSO prior (SSL) (see Appendix \ref{app_ssl} for more details), and at least in part to more complex sparse models as well. We refer to the Discussion Section \ref{sec:disc} for more on this.

\section{Simulations} \label{sec:simu}

We first present simulations in the sparse sequence model, showing a strong agreement with the behaviour expected from our theory, with both empirical Bayes (EB) and hierarchical Bayes (HB) choices of the horseshoe parameter $\ta$ (for EB, we propose new approximations to handle high dimensional settings). We then examin two more complex models: high-dimensional linear regression and Gaussian graphical models. We show that $s$--value and $Cs$--value procedures continue to perform remarkably well even beyond the theoretically studied sequence setting.

All the codes used in the simulations will be made available on the authors' webpage.

\subsection{Sparse sequence model}

We consider model  \eqref{model} with $n=10^4, s_n=10$ and, for simplicity, constant alternatives: $\theta_{0,i} = \mu$ if $1 \le i \le s_n$ and $0$ otherwise (the nonzero signals could be placed at other arbitrary places). Across different simulations, we vary the signal strength $\mu$ in the grid $\{0, 0.5, 1, 1.5, \ldots, 8\}$. Simulations for $n=10^4$ and $s_n=100$ are presented in Appendix \ref{Appendix:sequencemodel}.
We consider a horseshoe prior distribution on the unknown parameter $\theta$ with parameter $\tau$ chosen in an Empirical Bayes way as defined in \eqref{defmmle}.  We refer to the Appendix Section \ref{Appendix:tauMMLE} for more details on how to compute $\hat\ta$. The case of hierarchical Bayes is discussed below.

 We compare the behaviour of 
 the $s$--value  \eqref{svpro},  the $S$--value    \eqref{Svpro} and $Cs$--value \eqref{Csval} procedures, with `target levels' $t \in  \{0.05, 0.1, 0.2\}$. We also compare these to  
 the thresholding method (referred to as `Thresh' in Figure \ref{fig:all.Horseshoe.s10}) introduced by \cite{CarvalhoPolsonScott2010}. Writing the posterior mean for $\te_i$ as $c_iX_i$, with $c_i\in[0,1]$, the procedure `Thresh' selects variable $i$ if $c_i \geq \kappa$ for some user-selected cut-off $\kappa$ (with the default choice $\kappa = 0.5$, and we also consider  $\kappa = 0.8$).  

To implement the $s$-- and $Cs$--value procedures, one needs to access the $s$--values $s(X_i;\tau)$ or equivalently $\Pi_\ta[\te_i<0 \given X_i]$; one can estimate these from a  sample of the posterior distribution using the R {\tt horseshoe} package  \cite{horseshoe_package}. An important algorithmic point here is that since the mass $\Pi_\ta[\te_i<0 \given X_i]$ depends only on $X_i$ (but not on $i$ itself), we do not need to evaluate these for all $i\in\{1,\ldots,n\}$, but only on `sufficiently many $X_i$'s' so as to locate at which signal strength the procedure starts to reject: this idea is explained in more details in  Appendix \ref{Appendix:comparison.VdP.dichotomie} and is used in the simulations presented below. 
To implement the $S$--value procedure, we use an approximation of the $S$--value provided by the  theoretical bounds from Section \ref{sec-bSval}, see Appendix \ref{Appendix:SvalueEquivalent} for more details.  

The FDR and the True Discovery Rate (TDR, equal to 1-FNR) of each procedure are evaluated empirically with $100$ replications.  
Figure \ref{fig:all.Horseshoe.s10}  displays the results. The FDR of the $s$--value procedure remains very small and always controlled at the target level $t$ whatever the value of $\mu$, with a large TDR (equivalently, small FNR) for sufficiently large $\mu$, which is in close agreement with Theorem \ref{thm:hsrisk}. This is also in line with the empirical observations in \cite{vsvuq17}, Section 5: one of the procedures with credible intervals is in fact an $s$--value procedure if one sets $t=.05$ (they investigated dimensions up to $n=1600$ though, while here we cover higher-dimensional cases with $n=10^4$). As \cite{vsvuq17} noted, the procedure is somewhat conservative, and does not tend to catch moderate signals. 

The $S$--value procedure, on the other hand, has an FDR very close to the target level when the signal is sufficiently large, and otherwise has even smaller FDR under small signal; as expected the procedure is less conservative than the $s$--value procedure, with a higher TDR. 
 The behaviour of the C$s$--value procedure is similar to that of the $S$--value, albeit with an estimated FDR slightly above (but close to) the target level. This is in line with empirical observations in \cite{cr20} for the C$\ell$--value procedure with spike--and--slab priors.  
Finally, the  `Thresh' method based on the posterior mean has an FDR that is too large in all situations. It could be lowered by taking a different cut-off than 0.5, but this method gives no guarantee on the control of the FDR.  

Together with the theory derived in Section \ref{sec:main}, our study thus answers the questions raised in \cite{vsvuq17} for multiple testing with the horseshoe: in particular, we provide both a procedure optimal in terms of $\cR$--risk (the $s$--value); and a less conservative procedure, the $S$--value, that has provably optimal TDR given a certain target FDR level (and similarly for the C$s$--value procedure, empirically at least).  

Our empirical results can be extended in several directions, some of which we explore in the Appendix. The first is hierarchical Bayes: with the popular half--Cauchy prior on $\ta$, we obtain quite similar results to the empirical Bayes setting above. For this we use the off--the--shelf R package {\tt Mhorseshoe} \citep{Mhorseshoe, Mhorseshoepackage} (also used for linear regression below) with identity design matrix. It is likely that the sampler could be adapted to handle the specific case of the identity matrix faster; but here we are already able to cover sample sizes that are essentially comparable to those in \cite{ves21} (see their Section 5), who carried out a specific empirical study of approximations of spike--and--slab posteriors. Another direction consists in having different continuous shrinkage priors. The case of the spike--and--slab LASSO \cite{rockovageorge17} is considered in the Appendix, but we expect the results to hold for others priors as well (and in fact investigate such a case in Section \ref{sec:gm}). One could also put a prior on the noise variance, see Section \ref{sec:linreg}.

\begin{figure}[!t]
	\includegraphics[scale=0.6]{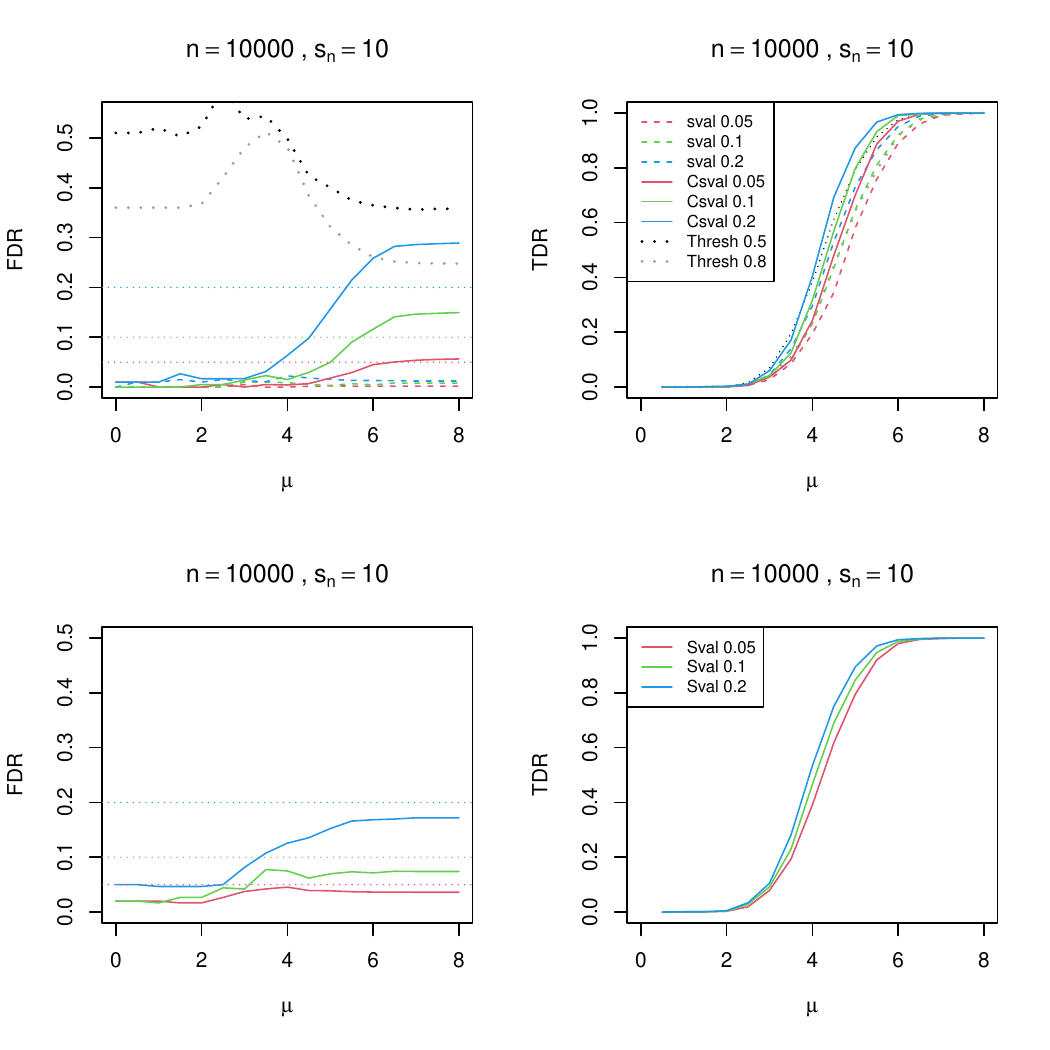}
	\centering
	\caption{In the sequence model, empirical FDR and TDR of the $s$--value and $Cs$-- value procedures with EB choice $\hat{\ta}$ of $\ta$, target level $t \in  \{0.05, 0.1, 0.2\}$, and those of the thresholding method with $\kappa=0.5$ or $0.8$ (top graph), as a function of signal $\mu$ (taken equal on non-zero coordinates); FDR and TDR of the $S$--value procedure with target level $t \in  \{0.05, 0.1, 0.2\}$ (bottom graph); $n = 10000; s_n=10$; averages are over $100$ replications. 	}
	\label{fig:all.Horseshoe.s10}
\end{figure}

\subsection{High-dimensional linear regression} \label{sec:linreg}

Consider observations $Z, X$ from the model 
\begin{equation} \label{regression}
	Z = X \theta +  \sigma \veps,  \hspace{1cm} \veps \sim \cN(0, I_n),
\end{equation}
with response $Z \in \RR^n $, design matrix $X \in \RR^{n \times p}$, parameter vector $\theta \in \RR^p$ and noise level $\sigma >0$. We
are interested in the high-dimensional sparse setting, where $p$ is possibly much larger than $n$  and many of the coefficients $\theta_i$ are zero. Following \cite{johndrow2020}, we consider a hierarchical  horseshoe prior  distribution on $\theta, \sigma$ as follows, independently over $i=1, \ldots,p$, 
\begin{equation} \label{hs_regression}
	\begin{aligned} 
		\theta_i \given \sigma^2, \ta, \lambda_i & \,\sim\, \cN(0,\sigma^2 \tau^2  \lambda_i^2 ) \\
		\lambda_i & \,\sim\, C^+(0,1), \hspace{0.5cm}    \\
		\tau & \,\sim\, C^+(0,1), \\
		\sigma^2 &  \,\sim\, \mbox{InvGamma}(w/2,w/2). \\
	\end{aligned}
\end{equation}
We use the R package {\tt Mhorseshoe} \citep{Mhorseshoe, Mhorseshoepackage}
to generate corresponding posterior samples, using the approximate algorithm proposed in \cite{johndrow2020}. This enables one to investigate settings with $p$ in the 10000's. 
Once the posterior sample of $\theta$ is generated, one can compute the corresponding Monte Carlo approximation of the $s$--values $s_i(Z)$  in \eqref{svalgen}.

We evaluate the behaviour of the  $s$--value and $Cs$--value procedures  in the regression model \eqref{regression} with true $\sigma_0= 1$ and $\theta_0$ having $s$ nonzero coefficients, each equal to $\mu$, for some fixed $\mu$  in $[0,0.2]$ and for different values of $(n, p, s)$. We consider a random design $X$ with $n$ i.i.d. rows drawn from a multivariate $\cN_p (0, \Sigma)$  law with the following choices of  $\Sigma$
\begin{enumerate}
	\item[(i)] $\, $  $\Sigma = I_p$ the identity $p\times p$ matrix, corresponding to uncorrelated features;
	\item[(ii)] $\, $   $\Sigma =\Sigma_{\rho}^{AR}$ the $p\times p$ autoregressive matrix  $(\Sigma_{\rho}^{AR})_{jk} = \rho^{|j-k|}$, with $\rho\in(0,1)$;
	\item[(iii)]  $\, $    $\Sigma =\Sigma_{\rho}$ the $p\times p$ matrix with $(\Sigma_{\rho})_{jk} = \rho$ if $j \ne k$ and $1$ otherwise, and $\rho\in(0,1)$.
	\end{enumerate}
Higher values of $\rho$ and the case $\Sigma = \Sigma_{\rho}$ both correspond to  stronger correlations, the later with $\rho$ close to $1$ thus representing the scenario that deviates most from the independent design.  Note that the typical `interesting' range of $\mu$  is lower than in the sequence model, since under Gaussian random design the signal--to--noise ratio is much higher, with $\mu$ of the order $\sqrt{\log{n}/n}$ (instead of order $\sqrt{\log{n}}$) representing the expected boundary of detection.   

We report empirical values of FDR and TDR for  $(n,p)=(3000,6000)$ and $(5000,10000)$, averaged over 20 simulations for design (i) and for  $(n,p)=(2000,4000)$ averaged over 50 simulations for designs (ii) and (iii).
The design matrix $X$ in \eqref{regression} can either be generated once for all replications or generated separately for each replication. We performed simulations in both cases. Generating $X$ for each simulation introduces a slight additional variability across replications but the results are similar.  
In the simulations presented below, we show results for $X$ generated at each replication for the uncorrelated design, and for $X$ generated once for all replications for the correlated design (see Appendix \ref{Appendix:regressionmodel} for various other settings that confirm these findings). 

Figure \ref{fig:regression.Horseshoe.uncorrelated} displays the results for the uncorrelated design (i). 
The $s$--value and $Cs$--value procedures perform similarly in the uncorrelated regression model, and the findings are qualitatively similar to those obtained in the sequence model, confirming an excellent behaviour in this case as well. The $s$--value procedure maintains an FDR close to $0$ regardless of the value of $\mu$, whereas the FDR of the $Cs$--value procedure is close to the target level, just slightly exceeding it for large signals. As a consequence, the $Cs$--value procedure is more powerful for moderate signals.

Figure \ref{fig:regression.Horseshoe.correlated_n2000_p4000} displays the results for the two correlated designs (ii) and (iii),  for different values of $\rho$. When a small correlation is added to the features (design (ii) with small or moderate values of $\rho$), the procedures perform similarly to the uncorrelated design. When the correlations between the features are stronger, the behaviour somewhat differs especially for small signal, although the FDR remain nearly controlled (design (iii), $\rho=.2$ and $.5$). Note that the later setting corresponds to rather strong correlations, where the posterior could be `biased' -- understanding this bias and the settings where it may appear for Bayesian sparse procedures is an interesting question that goes much beyond the scope of this work --.  Nevertheless, we observe that even in fairly strongly correlated cases (e.g. all settings of design (ii)), both $s$-- and C$s$--value procedures still perform remarkably well in terms of control of the FDR and power (see also Section \ref{sec:disc} for a discussion of open problems on this).

\begin{figure}[!h]
	\includegraphics[scale=0.55]{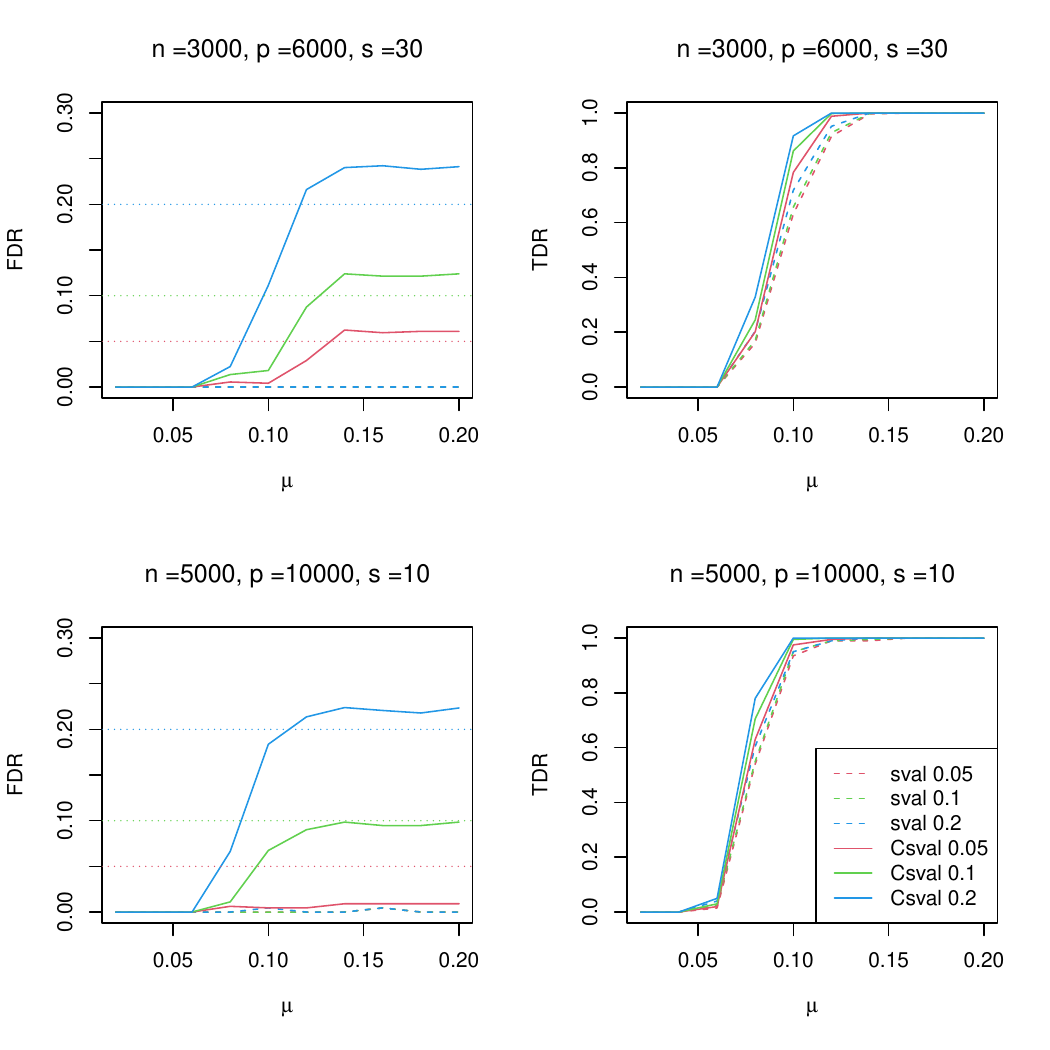}
	\centering 
		\caption{Linear Regression with uncorrelated features (i), performance of $s$--value and C$s$--value procedures. }
	\label{fig:regression.Horseshoe.uncorrelated}
\end{figure}

\begin{figure}[!t]
	\includegraphics[scale=0.55]{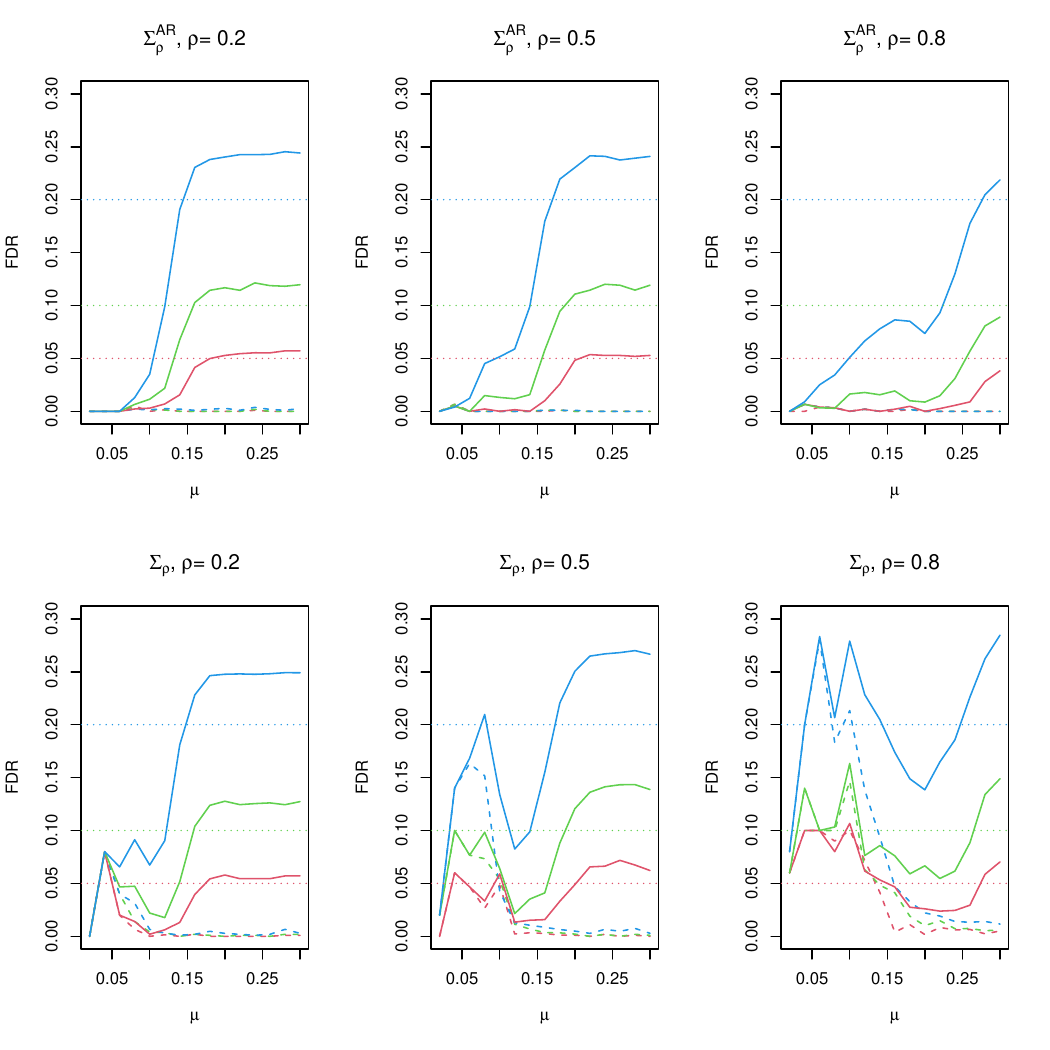}
	\centering
	\caption{Regression with  correlated designs (ii) and (iii), with $n=2000, p=4000, s=20$; the colour coding is as in Figure  \ref{fig:regression.Horseshoe.uncorrelated}.}
	\label{fig:regression.Horseshoe.correlated_n2000_p4000}
\end{figure}

\subsection{Gaussian Graphical Models} \label{sec:gm}

Let $X = (X_1,\dots,X_p)^\top \sim \mathcal{N}_p(0,\Sigma_0)$ with covariance $\Sigma_0$  in the set $\mathbb{S}_{++}^p$  of $p\times p$ symmetric positive definite matrices and let $\Omega_0 = \Sigma_0^{-1}$ denote the corresponding precision matrix. A Gaussian graphical model (GGM) represents the conditional independence structure of $X$ through an undirected graph $G_0 = (V,E_0)$, where $V = \{1,\dots,p\}$ and $E_0 \subseteq \{(i,j): 1 \leq i < j \leq p\}$ is the set of edges. In particular,
\[
(i,j) \notin E_0 
\quad \Longleftrightarrow \quad 
X_i \perp X_j \mid X_{-ij}
\quad \Longleftrightarrow \quad 
\Omega_{0,ij} = 0, \qquad i \neq j,
\]
so that the graph is fully encoded by the sparsity pattern of $\Omega_0$.

Given independent observations $X^{(1)},\dots,X^{(n)} \sim \mathcal{N}_p(0,\Sigma_0)$, the Gaussian log-likelihood for $\Omega \in \mathbb{S}_{++}^p$ is
\[
\ell_n(\Omega) 
= \frac{n}{2} \left( \log \det \Omega - \operatorname{tr}(S_n \Omega) \right),
\quad 
S_n = \frac{1}{n} \sum_{k=1}^n X^{(k)} X^{(k)\top}.
\]
In high-dimensional settings, estimation of $\Omega_0$ requires structural regularization, typically via sparsity assumptions. While penalized likelihood approaches are widely used, Bayesian formulations offer a coherent framework for both estimation and uncertainty quantification.

To this end, we adopt a continuous global--local shrinkage prior on the precision matrix, namely the graphical horseshoe (GHS) prior \cite{li2019graphical}. For $1 \le i < j \le p$,
\[
\Omega_{ij} \mid \lambda_{ij}, \tau \sim \mathcal{N}(0, \tau^2 \lambda_{ij}^2), 
\qquad 
\lambda_{ij} \sim \mathrm{C}^+(0,1), 
\qquad 
\tau \sim \mathrm{C}^+(0,1),
\]
with $\Omega_{ji} = \Omega_{ij}$ and support restricted to $\mathbb{S}_{++}^p$. The local scales $\lambda_{ij}$ enable edge-specific adaptivity, while the global parameter $\tau$ controls overall sparsity. The heavy-tailed half-Cauchy prior ensures significant shrinkage of small coefficients and minimal bias for large signals, yielding a continuous relaxation of graph selection.

The resulting posterior distribution is
\[
\pi(\Omega \mid X^{(1:n)}) \propto 
\exp\!\left\{ \frac{n}{2} \big( \log \det \Omega - \operatorname{tr}(S_n \Omega) \big) \right\} \cdot \pi(\Omega),
\]
which combines likelihood-driven inference with adaptive shrinkage. Under standard sparsity conditions, this formulation achieves optimal contraction behavior while avoiding combinatorial model selection over the graph space. We refer to the review paper \cite{banerjee2021bayesian}, Section 5, for more context and discussion on Bayesian methods for graphical models.

Recovering the conditional independence structure of $X$ amounts to identifying the zero and non-zero entries of the precision matrix. In the recent paper \cite{GHSCM}, the authors compare four frequentist methods with several Bayesian methods based on the GHS prior described above. We consider the same simulation settings to compare the performance of our $s$--value and $Cs$--value procedures for testing whether an entry of the precision matrix is zero with the results reported in \cite{GHSCM}.
Four network structures are considered: a sparse random network and a scale--free network generated using the R package \texttt{BDgraph}, as well as a scale--free network and a hub network generated using the R package \texttt{huge}. We refer to \cite{GHSCM} for details on the default arguments used and provide only the network structure illustrations as presented there, to facilitate visualization (note that the authors employ two data generators because the results may vary depending on the generator). As in \cite{GHSCM}, we perform simulations for $n=120$ and $p=100, 200$. FDR and TDR 
are evaluated empirically with $100$ replications. In the context of graph inference, the FDR corresponds to the average proportion of errors among the detected edges, whereas the TDR measures the average proportion of true edges that are successfully recovered. Results are given in Tables 
\ref{tab:GGM.p100} and 	\ref{tab:GGM.p200}.

\begin{figure}[t]
	\centering
	\begin{minipage}{0.22\textwidth}
			\centering
		\includegraphics[width=\textwidth]{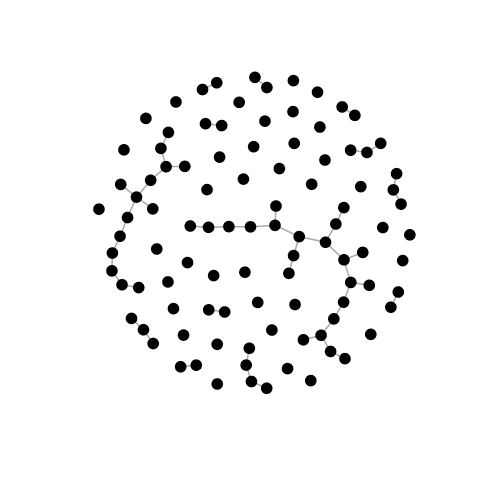}
		\end{minipage}
	\hfill
	\begin{minipage}{0.22\textwidth}
		\centering
		\includegraphics[width=\textwidth]{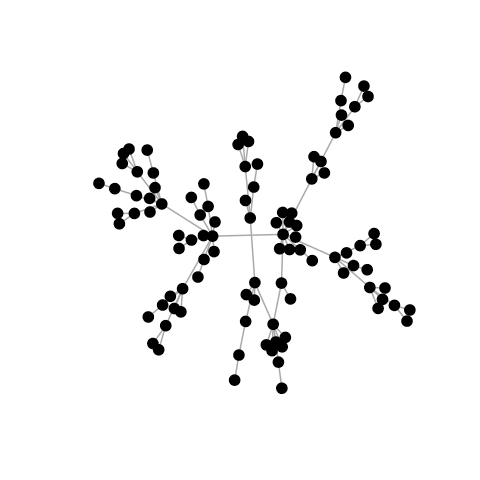}	
\end{minipage}
	\hfill
\begin{minipage}{0.22\textwidth}
		\centering
		\includegraphics[width=\textwidth]{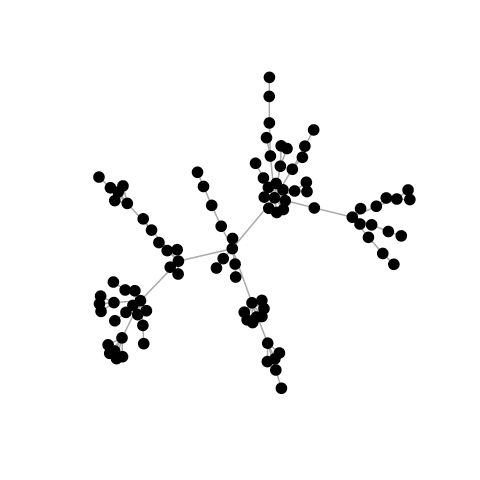}
	\end{minipage}
	\hfill
\begin{minipage}{0.22\textwidth}
		\centering
		\includegraphics[width=\textwidth]{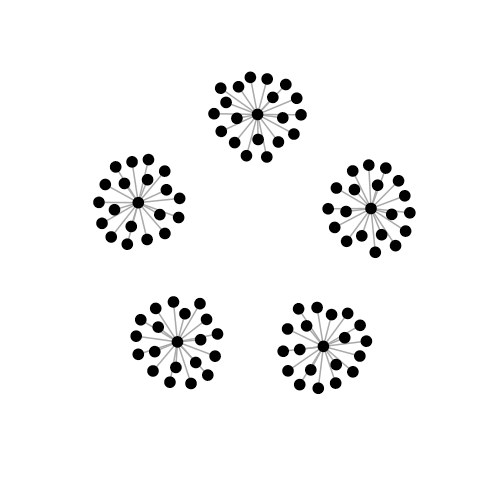}
	\end{minipage}
	\caption{Illustration of the graph structures considered in the simulation study. From left to right, random and scale--free networks were generated using the R package {\tt BDgraph}, and a scale--free network and a hub network generated using the R package {\tt huge}.}
	\label{fig:graph_topology}
\end{figure}

\begin{table}[ht]
	\centering
	\caption{FDR and TDR for the $s$--value and $Cs$--value procedures  with target level $t \in  \{0.05, 0.1\}$ under the four network structures; $n=120$ and $p=100$}
	\label{tab:GGM.p100}
	\begin{tabular}{cccc|cccc}
		\toprule
		\multicolumn{4}{c|}{Random} & \multicolumn{4}{c}{Scale-free (BDgraph)} \\
		\midrule
		t & method & FDR & TDR & $\alpha$ & method & FDR & TDR \\
		\midrule
		0.05 & s--value    & 0.004 & 0.498 & 0.05 & s--value    & 0.005 & 0.457 \\
		0.05 & Cs--value & 0.059 & 0.585 & 0.05 & Cs--value & 0.072 & 0.557 \\
		0.10 & s--value    & 0.007 & 0.523 & 0.10 & s--value    & 0.013 & 0.487 \\
		0.10 & Cs--value & 0.133 & 0.610 & 0.10 & Cs--value & 0.149 & 0.592 \\
		\midrule[\heavyrulewidth]
		\multicolumn{4}{c|}{Scale-free (huge)} & \multicolumn{4}{c}{Hub} \\
		\midrule
		t & method & FDR & TDR & $\alpha$ & method & FDR & TDR \\
		\midrule
		0.05 & s--value    & 0.015 & 0.088 & 0.05 & s--value    & 0.015 & 0.357 \\
		0.05 & Cs--value & 0.042 & 0.141 & 0.05 & Cs--value & 0.054 & 0.556 \\
		0.10 & s--value    & 0.038 & 0.120 & 0.10 & s--value    & 0.032 & 0.459 \\
		0.10 & Cs--value & 0.091 & 0.196 & 0.10 & Cs--value & 0.115 & 0.697 \\
		\bottomrule
	\end{tabular}
\end{table}

\begin{table}[ht]
	\centering
	\caption{FDR and TDR for the $s$--value and $Cs$--value procedures  with target level $t \in  \{0.05, 0.1\}$ under the four network structures; $n=120$ and $p=200$.} 
	\label{tab:GGM.p200}
		\begin{tabular}{cccc|cccc}
		\toprule
		\multicolumn{4}{c|}{Random} & \multicolumn{4}{c}{Scale-free (BDgraph)} \\
		\midrule
		t & method & FDR & TDR & $\alpha$ & method & FDR & TDR \\
		\midrule
		0.05 & s--value    & 0.006 & 0.504 & 0.05 & s--value    & 0.005 & 0.460 \\
		0.05 & Cs--value & 0.074 & 0.588 & 0.05 & Cs--value & 0.070 & 0.559 \\
		0.10 & s--value    & 0.011 & 0.528 & 0.10 & s--value    & 0.011 & 0.488 \\
		0.10 & Cs--value & 0.146 & 0.613 & 0.10 & Cs--value & 0.144 & 0.588 \\
		\midrule[\heavyrulewidth]
		\multicolumn{4}{c|}{Scale-free (huge)} & \multicolumn{4}{c}{Hub} \\
		\midrule
		t & method & FDR & TDR & $\alpha$ & method & FDR & TDR \\
		\midrule
		0.05 & s--value    & 0.039 & 0.030 & 0.05 & s--value    & 0.012 & 0.347 \\
		0.05 & Cs--value & 0.081 & 0.046 & 0.05 & Cs--value & 0.044 & 0.535 \\
		0.10 & s--value    & 0.070 & 0.040 & 0.10 & s--value    & 0.022 & 0.439 \\
		0.10 & Cs--value & 0.116 & 0.067 & 0.10 & Cs--value & 0.096 & 0.660 \\
		\bottomrule
	\end{tabular}
\end{table}

If we compare with the methods studied in \cite{GHSCM} (whose FDR and TDR results are reported in Table 2 of their paper), we observe that these competing methods do not control the FDR in most cases. While some of them achieve high TDR, this often comes at the cost of a very large FDR. In contrast, the $s$--value and $Cs$--value procedures provide control of the FDR at a desired level. Besides maintaining a  control of the FDR, the $s$--value and $Cs$--value procedures are still able to achieve high TDR values, thereby offering an attractive compromise between power and reliability in edge selection, which is essential for the interpretability of inferred graphs in practice.

 In the GGM context, the $s$-- and $Cs$--value procedures can be viewed as an additional layer applied to posterior samples, enabling graph inference with FDR control.

\section{Discussion} \label{sec:disc}

In this work we propose a principled approach to multiple testing with the horseshoe prior, which applies more generally to other continuous shrinkage priors. We demonstrate that the proposed procedures, that include an Empirical Bayes step to calibrate the prior's hyperparameter, achieve sharp minimax behaviour in the sparse sequence model, adaptive to the unknown sparsity parameter $s_n$, in close agreement with simulation results. We briefly mention below several possible extensions of our results; these also open the door to investigating a number of open questions that we also briefly discuss.

Here we have focused mainly on the horseshoe prior with an Empirical Bayes choice of the sparsity parameter $\tau$. Using similar arguments, one could endow $\ta$ itself with a prior and obtain analogous results for the resulting Hierarchical Bayes method. Another natural extension would be to obtain theory for other shrinkage priors: we conjecture that all the results of the paper remain valid for example for the Spike and Slab LASSO prior \cite{rockovageorge17}. Another natural shrinkage prior to investigate, somewhat farther from the horseshoe, would be the Dirichlet--Laplace prior \cite{BhattacharyaEtAl2015DL}, in particular in light of the excellent empirical results obtained in \cite{Chandra26} and Appendix \ref{app_dl}. 
 Finally, although for simplicity for theory we focused mostly on the $s$-- and $S$--value procedures, we believe that all results for the latter also hold in exactly the same form for the (arguably more practical for complex models) C$s$--value procedure, which could be checked following the arguments of \cite{acr22}.  

Our simulations study suggests that the theoretically proved properties continue to hold much beyond the sequence model. A very natural setting to derive further theory would be  high-dimensional linear regression, where we see empirically that the FDR control remains equally good compared to the sparse sequence model, at least under weakly or moderately correlated designs. This goes beyond the scope of the present work -- for instance, one may note that a minimax multiple testing framework for this setting  (analog to the results of \cite{acr24} for sparse sequences) has not yet been investigated --. Similarly, while we have seen strong empirical evidence in favour of the procedures for GGMs, theory for these is also not available yet. We very much hope that such settings will be addressed in future contributions. 

\section{Proof of the first main result} \label{sec:proofs}
In this Section we provide (part of) the proof of Theorem \ref{thm:hsrisk}, whose outline is given below. The proof of Theorem \ref{thm:fdrcontrol} is given in Appendix \ref{app:proofSval}.

\subsection{Expressions for $s$-- and $S$--values}
 Before giving an outline of the proof of Theorem \ref{thm:hsrisk}, we state and prove a simple lemma. 
\begin{lemma} \label{lem:Sval}
Let $x\ge 0$ and $\ta\in(0,1)$. Then, for the horseshoe prior with parameter $\ta$, the $s$--value is given by
\[ s(x;\ta) = 2
\frac{\int_0^{\infty} \Phi \left( - \la \tau x/\sqrt{1 + (\la\tau)^2}\right) \phi_{0,1 + (\la \tau)^2}(x)  d\cC^+(\la)}{\int_0^{\infty}  \phi_{0,1 + (\la\tau)^2}(x)  d\cC^+(\la)}=: 2\frac{N(x)}{D(x)},
\]
where $\cC^+$ denotes a standard half-Cauchy distribution. 
The $S$--value is given by 
\[ S(x;\ta) = 2 \frac{\int_x^{\infty} N(y)dy}{ \int_x^{\infty} D(y)dy}. \]

\end{lemma}
\begin{proof}[Proof of Lemma \ref{lem:Sval}]
By definition of the horseshoe prior as a mixture, the distribution of $(X_1,\te_1)$ given $\la_1$ is 
\[ \cN\left(\, 0 \,,\, 
\begin{pmatrix}
1+(\ta\la_1)^2 & (\ta\la_1)^2\\
(\ta\la_1)^2 & (\ta\la_1)^2
\end{pmatrix}\,
\right).
 \]
It follows that the conditional distribution $\te_1\given  X_1,\la_1$ is $\cN(X_1/(1+(\ta\la_1)^{-2}),1/(1+(\ta\la_1)^{-2}))$ and the marginal distribution $X_1\given \la_1$ is $\cN(0,1+(\ta\la_1)^2)$. 
One deduces that the joint conditional density of $(\te_1,X_1)\given \la_1$ at point $(t,x)$ is
\[  \phi_{x/(1+(\ta\la_1)^{-2}),1/(1+(\ta\la_1)^{-2})}(t)\phi_{0,1+ (\la_1 \tau)^2}(x). \]
From this one deduces the expressions of the $s$--value and $S$--value. For the first, one notes
\[
\Pi(\te_1<0 \given X_1, \lambda_1) = \Phi \left( - \dfrac{\lambda_1 \tau X_1}{\sqrt{1 + (\lambda_1\tau)^2}}\right),
\]
so that $\Pi(\te_1<0 \given X_1)=N(X_1)/D(X_1)$ with 
\[
N(X_1) = \int_0^{\infty} \Phi \left( - \dfrac{\la \tau X_1}{\sqrt{1 + (\la\tau)^2}}\right) \phi_{0,1 + (\la \tau)^2}(X_1) \frac{2}{\pi} \dfrac{d\la}{1 + \la^2}, 
\]
and
\[
D(X_i) = \int_0^{\infty}  \phi_{0,1 + (\la\tau)^2}(X_1) \frac{2}{\pi}\dfrac{d\la}{1 + \la^2}.
\]
For the $S$--value and $X\ge x$, it suffices to integrate each expression from $x$ to $+\infty$, so that
\begin{equation} \label{condexp}
 \Pi(\te_1<0 \given X_1\ge x) = \frac{\int_x^{\infty} N(y)dy}{ \int_x^{\infty} D(y)dy}. \end{equation}  
\end{proof}

\subsection{Outline of the proof of Theorem \ref{thm:hsrisk}} The proof has three main steps  
\begin{enumerate}
\item deriving precise estimates on the $s$--value for the horseshoe prior and fixed value of $\tau$;
\item handling a random $\ta$ and showing that the MMLE $\hat\ta$ concentrates not too far from an `oracle' value $\ta=\ta_n(s_n)=(s_n/n)\sqrt{\log(n/s_n)}$ for signals in $\cL_0[s_n;b]$;
\item showing an upper bound on (FDR+FNR)--risk that matches the minimax rate. 
\end{enumerate}
The required estimates of $s$--values (Step 1) as well as the concentration property of $\hat\ta$ (Step 2) are provided in Section \ref{subsec:lemmas}. The proofs of the corresponding properties can be found in Appendix \ref{app_svalbounds}. Section \ref{sec:proofthm1} proves that the FNR is asymptotically upper bounded by the target minimax rate. To conclude the proof, it then suffices to show that the FDR vanishes asymptotically: this is slightly more complex than for the FNR and verified in Appendix \ref{app_endproofthm1}.

\subsection{Proof of Theorem \ref{thm:hsrisk}}   \label{sec:proofthm1}

\noindent {\em Step 1 (bounds on the $s$--value)}. Lemma \ref{lem:boundsv} below  provides estimates for the $s$--value $s(x;\ta)$ that we use freely in the proof below. \\

\noindent {\em Step 2.}  The second main technical ingredient is to prove that with  high--probability, the marginal maximum likelihood estimate $\hat\ta$ of $\ta$ concentrates appropriately.  
Define, for any $1\le s\le n$,
\begin{equation} \label{taun}
\ta_n(s) = \frac{s}{n}  \sqrt{\log(n/s)}.
\end{equation}  
Let us introduce the following event on the MMLE estimator $\hat\ta$ 
\begin{equation} \label{evento}
\cA:= \left\{\hat\ta(X) \in \left[d_1 \ta_n(s_n) , d_2 \ta_n(s_n)\right] \right\},
\end{equation}
where $d_1, d_2$ are as in Proposition \ref{propmle} below: the latter result shows that the event $\cA$ holds with high probability. 
In the sequel, we also use as shorthand notation, for $\ta\in(0,1)$,
\begin{equation} \label{defshort}
 \ta_+:= d_2 \ta_n(s_n),\  \ \ta_-:= d_1 \ta_n(s_n), \ \ M_\ta:=\ta^{-1/4}. 
\end{equation}

\noindent {\em Step 3.} We start with a control of the FNR. Define the function, for $u>0$, 
\begin{equation} \label{defh}
h(u)= \left( \frac{u^2}{2} \phi(u) \right)^{-1} = \sqrt{2\pi} \frac{e^{u^2/2}}{u^2/2}.
\end{equation}
Since $v\to e^{v}/v$ is strictly increasing on $[1,+\infty)$, the map $h:[\sqrt{2},+\infty)\to [\sqrt{2\pi}e,+\infty)$ admits an inverse $\la(\cdot)$
\begin{equation} \label{defla}
\la(\cdot):=h^{-1}(\cdot),
\end{equation}  
that is $\la(h(x))=x$ and  $h(\la(y))=y$ for $x\in [\sqrt{2},+\infty)$ and $y\in [\sqrt{2\pi}e,+\infty)$.\\

For $\te_0\in\cL_0[s_n;b]$, the number of false negatives of the procedure $\hat\vphi_t^{s}$ can be written as
\begin{align*} 
\sum_{i\in S_{\theta_0}} \1\{ s(X_i;\hat\ta)>t \}
& = \sum_{i\in S_{\theta_0}} \1\{ s(X_i;\hat\ta)>t \} \1_{|X_i|\in[\sqrt{2},M_{\ta_+}]}
+\sum_{i\in S_{\theta_0}} \1\{ s(X_i;\hat\ta)>t \}\1_{|X_i|\notin[\sqrt{2},M_{\ta_+}]} \\
& = \quad \qquad \qquad\qquad (I)  \qquad\qquad \quad + \quad\qquad \qquad\qquad (II). 
\end{align*}
Since $P_{\te_0}[\cA^c]=o(1)$ by Proposition \ref{propmle} for $\cA$ as in \eqref{evento}, one can restrict to the event $\cA$ for the analysis of these terms.
 Since $\ta\to M_\ta$ in \eqref{defshort} is decreasing, one has $M_{\hat\tau}>M_{\ta_+}$ on $\cA$.

Dealing first with term (I), using the first part of Lemma \ref{lem:boundsv} (since $M_{\hat\tau}>M_{\ta_+}$ on $\cA$),
\begin{align*}
 (I)\1_{\cA} & \le \sum_{i\in  S_{\theta_0}} \1\left\{ \frac{1+O(\ta_+^{1/4})}{1 + \ta_- h(X_i)/4}>t \right\} \1_{|X_i|\in[\sqrt{2},M_{\ta_+}]} \\
 & = \sum_{i\in  S_{\theta_0}} \1\left\{ h(X_i)< \frac{4}{\ta_-}\left(\frac{1-t}{t}+O\left(\ta_+^{1/4}\right)\right) \right\}\1_{|X_i|\in[\sqrt{2},M_{\ta_+}]}\\
& \le \sum_{i\in  S_{\theta_0}} \1\left\{ |X_i| < \la\left(\frac{4}{\ta_-}\left(\frac{1-t}{t}+O\left(\ta_+^{1/4}\right)\right)\right) \right\},
 \end{align*}
 where we use the inverse $\la$ of $h$ from \eqref{defla}, after noting that the term under brackets in the last display diverges, so is larger than $\sqrt{2}$ for large enough $n$.   
Let us further denote 
\begin{equation} \label{tautp}
 \ta_{t,+}:=\frac{4}{\ta_-}\left(\frac{1-t}{t}+O(\ta_+^{1/4})\right). 
\end{equation}  
For $i\in S_{\te_0}$ and  if $\te_{0,i}>0$, it holds $\1\left\{ |X_i| < \la(\ta_{t,+}) \right\} \le 
\1\left\{ \veps_i < \la(\ta_{t,+})-\te_{0,i} \right\}$; similarly,  if $\te_{0,i}<0$, one writes $\1\left\{ |X_i| < \la(\ta_{t,+}) \right\} \le 
\1\left\{ \veps_i > -\la(\ta_{t,+})-\te_{0,i} \right\}$. Assuming without loss of generality that $\te_{0,i}>0$ and using $\te_0\in\cL_0[s_n;b]$,
\[ \1\left\{ |X_i| < \la(\ta_{t,+}) \right\} 
\le  \1\left\{ \veps_i < -b + \la(\ta_{t,+}) - \sqrt{2\log(n/s_n)} \right\}.
\]
The upper--bound on $\la$ provided in Lemma \ref{lemla} gives, for $\ta_{t,+}$ as in \eqref{tautp},
\begin{align*} \la(\ta_{t,+})  - & \sqrt{2\log(n/s_n)} 
 = \frac{\la(\ta_{t,+})^2 - 2\log(n/s_n)}{ \la(\ta_{t,+}) + \sqrt{2\log(n/s_n) }}\\
& \le \frac{2\log(s_n \ta_{t,+}/n)+2\log\{2\log(\ta_{t,+})\}}{ \la(\ta_{t,+}) +\sqrt{2\log(n/s_n) }}.
\end{align*}
The numerator in the last expression of the above display is bounded from above by 
\[ 2\log\{(4/d_1)(1-t+o(1))\}-\log\log(n/s_n)+2\log(2\log\{(4/d_1)(n/s_n)(1-t+o(1))\}),\]
which is less than  $C+\log\log(n/s_n)$ for large enough $n/s_n$ for some $C>0$, so that
\[ \1\left\{ |X_i| < \la(\ta_{t,+}) \right\} 
\le  \1\left\{ \veps_i < -b + \frac{C+\log\log(n/s_n)}{\sqrt{2\log(n/s_n)}}\right\}.
\]
From this one deduces that $E_{\te_0}[ (I)\1_{\cA}] \le s_n \Phi(-b+o(1))=s_n\{\bar\Phi(b)+o(1)\}$. 

To bound from above the term (II), one uses 
$\1\{|X_i|\notin [\sqrt{2},M_{\hat\ta}]\}\le \1\{|X_i|\notin [\sqrt{2},M_{\ta_+}]\}$
and further splits 
$\1\{|X_i|\notin [\sqrt{2},M_{\ta_+}]\}=\1\{|X_i|<\sqrt{2}\} + \1\{|X_i|>M_{\ta_+}]\}$. 
The contribution of the first term to (II) is upper-bounded by, using that $\te_0\in\cL_0[s_n;b]$,
\[ s_n P\left[|\veps_1|>\sqrt{2\log(n/s_n)}+b-\sqrt{2}\right]=o(s_n).\]
To bound the contribution of $\1\{|X_i|>M_{\ta_+}]\}$ to (II), the second part of Lemma \ref{lem:boundsv} gives 
\[ \1\{ s(X_i;\ta)>t, |X_i|>M_{\ta_+} \}
\le \1\{ CX_i^2\phi(X_i)/{\ta_-} >t , |X_i|>M_{\ta_+}\},\]
using $\ta_-\le \hat\ta\le 1$; the last set is empty for $n/s_n$ large enough, using that $M_{\ta_+}^2\phi(M_{\ta_+})/{\ta_-}=o(1)$ since
 $\ta_- \asymp  \ta_+ \asymp \ta_n(s_n)$. Putting the previous bounds together leads to
  \[ \limsup_{n,s_n} \sup_{\te_0\in\cL_0[s_n;b]} \FNR(\te_0, \hat\vphi_t^s) \le \bar\Phi(b). \] 
To complement this bound, one similarly needs the following control of the FDR of the $s$--value procedure:  uniformly over $\cL_0[s_n]$
 \[ \lim_{n,s_n} \sup_{\te_0\in\cL_0[s_n;b]} \FDR(\te_0, \hat\vphi_t^s) = 0. \] 
 The proof of the last display is generally similar to that for the FNR, but slightly more technical  (due to the random denominator appearing in the FDR) and provided in the Appendix  \ref{app_endproofthm1}.    
 This shows that the multiple testing risk $\cR=\FDR+\FNR$ of the $s$--values procedure is asymptotically bounded from above by $\bar\Phi(b)$ for large enough $n$. Combining this with the minimax lower bound \eqref{mmr} (see \cite{acr24}, Theorem 1)
 concludes the proof of Theorem \ref{thm:hsrisk}. $\qed$

\subsection{Main technical Lemmas}
\label{subsec:lemmas} 

Recall the definition $\ta_n(s) = (s/n) \sqrt{\log(n/s)}$.
\begin{proposition} \label{propmle}
Let $\hat\ta$ be the marginal maximum likelihood given by \eqref{defmmle}. Let $b$ be a fixed, arbitrary, real number. There exist positive constants $d_1 < d_2$ such that,  for $\cL_0[s_n;b]$ the class of signals \eqref{defclb},
\[ 
\sup_{\te_0 \in  \cL_0[s_n;b]} P_{\te_0}\Big( \hat\ta(X)
\notin \left[d_1 \ta_n(s_n) , d_2 \ta_n(s_n)\right] \Big) =o(1).
\] 
\end{proposition}
\noindent Proposition \ref{propmle} is proved in Appendix \ref{sec:estimtau},  the next two Lemmas in Appendix \ref{sec:bounds}.
 
\begin{lemma}  \label{lem:boundsv}
There exists a $\ta_0\in(0,1)$ such that the $s$--value is bounded as follows. 
For any  $\ta\in[\ta_-,\ta_+]$, for some $0<\ta_-\le \ta_+\le \ta_0$: 
\begin{itemize}
\item uniformly in $x\in[\sqrt{2},M_{\ta_+}]$,
\[ \frac{1-|O(\ta_+^{1/4})|}{1 +  \ta_{+} h(x)} \le s(x;\ta) \le \frac{1+O(\ta_+^{1/4})}{1 + \ta_- h(x)/4}, \]
\item uniformly in $x>M_\ta$,
\[ s(x;\ta) \le C(1+\sqrt{\ta})\frac{x^2}{\ta}\phi(x).\]
\end{itemize}
\end{lemma}

\begin{lemma} \label{lemla}
For any $v>0$ large enough, the following bounds hold for $\la(\cdot)$ as in \eqref{defla}
\begin{align}
\la(v) & \le \sqrt{2\log{v}+2\log\{2\log(v)\}-\log(2\pi)},       \label{laub}\\
\la(v) & \ge \sqrt{2\log{v}+2\log\{\log(v/\sqrt{8\pi})\}-\log(2\pi)}.  \label{lalb}
\end{align}
\end{lemma}

\begin{funding}
This work was funded by ANR BACKUP grant ANR-23-CE40-0018-01 and IIM Indore's research grant GSRC-2023-1.
\end{funding}

\begin{supplement}
\stitle{Supplementary material to `Multiple testing with the horseshoe'}
\sdescription{We provide the rest of the proofs of the results contained in the main article, some technical lemmas as well as additional simulations corroborating the theory.}
\end{supplement}

\bibliographystyle{imsart-number} 
\bibliography{bib_mt_hs}      

\pagebreak

 \appendix 

\title{Supplementary material to `Multiple testing with the horseshoe'}

\vspace{1cm}

This supplement contains the remaining proofs from the main paper in Section \ref{app:proofs}. In Section \ref{app:sims}, more details on the algorithms used for the simulations are provided, together with additional simulations in the several settings considered in the main paper. Section \ref{app_ssl} gives some more details and context on the comparison with the spike--and--slab LASSO prior.

\section{Proofs} \label{app:proofs}

\subsection{End of the proof of Theorem \ref{thm:hsrisk}}
\label{app_endproofthm1}

It remains to show that the FDR vanishes uniformly for vectors in the class $\cL_0[s_n;b]$. To do so, one first derives a bound on the number of true positives of the procedure $\hat\vphi^s_t$, that is
\[ TP := \sum_{i\in S_{\theta_0}} \1\{ s(X_i;\hat\ta)\le t \}.\]

{\em Controlling the number of true positives TP.}  We show that  the event, with $c_1:=\Phi(b)$,
\begin{equation} \label{defb}
\cB = \left\{   TP \ge (c_1-\delta)s_n\right\},
\end{equation}
has a probability going to $1$.  The event can also be written $\{FN\le (c_2+\delta)s_n\}$, with $c_2=\bar\Phi(b)$, using that the number of false negatives $FN$ equals $s_n-TP$. Recall from the beginning of the proof that 
FN can be written as
\begin{align*} 
\sum_{i\in S_{\theta_0}} \1\{ s(X_i;\hat\ta)>t \}
& = \sum_{i\in S_{\theta_0}} \1\{ s(X_i;\hat\ta)>t \} \1_{|X_i|\in[\sqrt{2},M_{\ta_+}]}
+\sum_{i\in S_{\theta_0}} \1\{ s(X_i;\hat\ta)>t \}\1_{|X_i|\notin[\sqrt{2},M_{\ta_+}]} \\
& = \quad \qquad \qquad\qquad (I)  \qquad\qquad \quad + \quad\qquad \qquad\qquad (II). 
\end{align*}
Also, we have already shown earlier that $(II)=o_{P_{\te_0}}(s_n)$ on the event $\cA$. It is thus enough to bound the probability, using the notation $\ta_{t,+}$ as in \eqref{tautp},
\[ P_{\te_0}\left[(I) \ge (c_2+\delta/2)s_n\, ,\, \cA\right]
\le   P_{\te_0}\left[\sum_{i\in S_{\te_0}} 
\1\left\{ |X_i| < \la\left(\ta_{t,+}\right)  \right\}\ge (c_2+\delta/2)s_n \right]. \]
Denoting by $N_1$ the random variable $ \sum_{i\in S_{\theta_0}} \1\{ |X_i|< \la(\ta_{t,+}) \}$, one notes that $N_1$ is a sum of independent Bernoulli variables of parameters, say, $a_i$. This is a Poisson--Binomial distribution $PBin(a)$ with parameters $a=(a_i)$, and by the same reasoning as for the FNR above,
\[ a_i = P_{\te_0}\left[|X_i| < \la(\ta_{t,+}) \right] \le \bar\Phi(b)+o(1)
=c_2+o(1), \]
where the $o(1)$ in fact does not depend on $i$. Since the Poisson--Binomial distribution is stochastically increasing in its parameters (Lemma S-25 in \cite{acr24}), we have that $PBin(a)$ is stochastically dominated by a binomial $Bin(s_n,\bar\Phi(b)+o(1))$ distribution -- denote by $N_2$ a variable following the latter --. Now one can bound from above, for small $\delta>0$, 
\begin{align*}
P_{\te_0}[\cB^c] & \le P[N_1\ge (c_2+\delta/2)s_n]+o(1) \le P[N_2 \ge (c_2+\delta/2)s_n]+o(1)\\
& \le P[N_2-E N_2 \ge (\delta/4)s_n]+o(1),
\end{align*}
for $n,s_n$ large enough, using that $E N_2=s_n(\bar\Phi(b)+o(1))=s_n(c_2+o(1))$. An application of Bernstein's inequality (Lemma \ref{bernstein}) for the binomially distributed $N_2$ now gives $P_{\te_0}[\cB^c]=o(1)$ as desired.\\

{\em FDR control.} Let $FP$ be the number of false positives of the $s$--value procedure and denote \[ \Delta:= \sum_{i\notin S_{\te_0}} \1_{|\veps_i|> M_{\ta_+}}. \]
By definition of $M_{\ta_+}$, we have $E_{\te_0}\Delta\le 2n\bar\Phi(\ta_+^{-1/4})$. Since  by definition
\[ \ta_+\leqa (s_n/n)\sqrt{\log(n/s_n)}\leqa \sqrt{s_n/n},\]  
we have $\ta_+^{-1/4}\geqa (n/s_n)^{-1/8}$ and so, using the standard bound $\bar\Phi(x)\le e^{-x^2/2}$ for $x\ge 0$,
\[ E_{\te_0}\Delta \le 2n\bar\Phi(\ta_+^{-1/4})\le 2s_n (n/s_n)\exp\{-C(n/s_n)^{-1/4}\}=o(s_n).\]   
Then we bound from above
\begin{align*}
FP & \le  \sum_{i\notin S_{\theta_0}} \1\{ s(\veps_i;\hat\ta) < t \}\1_{|\veps_i|\le M_{\ta_+}} 
+ \Delta.
\end{align*}
The indicator in the last display can be bounded from above as follows, on the event $\cA$,
\begin{align*}
  \1\left\{ s(\veps_i,\hat\ta )<t ,  |\veps_i|\le M_{\ta_+} \right\}
& \le \1\left\{  s(\veps_i,\hat\ta)<t , \sqrt{2}\le |\veps_i|\le M_{\ta_+} \right\} + \1\left\{  s(\veps_i,\hat\ta)<t ,  |\veps_i|<\sqrt{2}\right\}
\\
& \le
\1\left\{  \frac{1-|O(\ta_+^{1/4})|}{1+\ta_+ h(\veps_i)}<t \right\} + \1\left\{ 1-|O(\ta_+^{1/4})|<t \right\},
\end{align*}
where to bound  the first indicator in the sum we have used Lemma \ref{lem:boundsv}, while for the second the first estimate of Lemma \ref{lemsval} for $x\in[0,\sqrt{2}]$. Now the last indicator in the above display equals $0$ for large enough $n,s_n$ (uniformly in $i$) since $t<1$. This leads to, on the event $\cA$,
\begin{align*}
 FP & \le \sum_{i\notin S_{\theta_0}} \1\left\{ h(\veps_i) > \frac{1-t-|O(\ta_+^{1/4})|}{t\ta_+} \right\} +\Delta\\
& \le \left[\sum_{i\notin S_{\theta_0}} \1\left\{|\veps_i| > \la(C_1/\ta_n(s_n))\right\} +\Delta\right]=: F_+,
\end{align*} 
for some $C_1>0$ (and large enough $n, s_n$), using the definition of $\ta_+$. 

 Then, with $\cA, \cB$ the events in \eqref{evento}, \eqref{defb} respectively, for $F_+$ as in the last display,
\begin{align*}
 \FDR(\te_0,\hat\vphi_t^s) &= E_{\theta_0}\left[ \frac{FP}{(FP+TP)\vee 1} \right] 
\le E_{\theta_0}\left[\frac{F_+}{F_+ +cs_n} \1\{\cA \cap \cB \} \right] +o(1)\\
& \le E_{\theta_0}\left[\frac{F_+}{F_+ +cs_n}\right]+o(1) \le \frac{E_{\theta_0} F_+}{E_{\theta_0} F_+ +cs_n}+o(1),
\end{align*}
where for the first inequality one uses monotonicity of $x\to x/(a+x)$ for $a\ge 0$, the lower bound on $TP$ on $\cB$ and the fact that $P_{\te_0}[(\cA\cap \cB)^c]=o(1)$; for the last inequality one uses Jensen's inequality together with the fact that $x\to x/(x+d)$ is concave over $x\ge 0$, for any given $d>0$. One can further bound, denoting $z_n:=C_1/\ta_n(s_n)$,
\[ E_{\theta_0} F_+ \le nP[|\veps_1|>\la(z_n)] + o(s_n). \]
Using the lower bound in Lemma \ref{lemla}, the probability in the last display is bounded from above by, using $\bar\Phi(u)\le \phi(u)/u$ for $u>0$, with $z_n=C_1/\ta_n(s_n)$ and $C_1$ as above,
\begin{align*}
P&[|\veps_1|>\la(z_n)]  \le P\left[|\veps_1|>\sqrt{2\log{z_n}+2\log\{\log(Cz_n)\}-\log(2\pi)}\right]\\
& \leqa \phi\left(\sqrt{2\log\left\{C_1 \frac{n}{s_n}\frac{\log(Cz_n)}{\sqrt{\log(n/s_n)}}\right\}-\log(2\pi)}\right)/\sqrt{\log(n/s_n)}\\ 
& \leqa \phi\left(\sqrt{2\log\left\{C_2 \frac{n}{s_n}\sqrt{\log(n/s_n)} \right\}-\log(2\pi)}\right)/\sqrt{\log(n/s_n)}, 
\end{align*}
for some $C_2>0$, where for the last inequality one uses $\log(Cz_n)\geqa \log(n/s_n)$. 
This further gives, for a constant $C_3>0$,
\begin{align*}
P[|\veps_1|>\la(z_n)] 
& \le 
C_3 \frac{s_n}{n}(\log(n/s_n))^{-1}.
\end{align*}
Putting the previous inequalities together leads to $E_{\te_0} F\le C_3 s_n/\log(n/s_n)+o(s_n)=o(s_n)$. Going back to the first inequality on the FDR above, and using again  that $x\to x/(x+d)$ is increasing over $x\ge 0$, for any given $d>0$, one concludes that 
\[ 
 \FDR(\te_0,\hat\vphi_t^s)  \le \frac{E_{\theta_0} F_+}{E_{\theta_0} F_+ +cs_n}+o(1)\le \frac{o(s_n)}{o(s_n)+cs_n}+o(1)=o(1),
\]
so that the FDR of the $s$--value procedure goes to $0$, uniformly over $\cL_0[s_n;b]$, that is
 \[ \lim_{n,s_n} \sup_{\te_0\in\cL_0[s_n;b]} \FDR(\te_0, \hat\vphi_t^{s}) = 0. \] 
 This shows that the multiple testing risk FDR$\,$+$\,$FNR is asymptotically bounded from above by $\bar\Phi(b)$ for large enough $n$. Combining this with the minimax lower bound \eqref{mmr} (see \cite{acr24}, Theorem 1)
 concludes the proof of Theorem \ref{thm:hsrisk}. \qed

\subsection{Proofs for the $s$--value procedure: technical lemmas}
\label{app_svalbounds}

\subsubsection{Bounds on $s$--values}  \label{sec:bounds}

Let us recall the definition, for any $\tau\in(0,1)$ (see \eqref{defshort})
\[ M_\ta = \ta^{-1/4}.\]
Let us also verify that when bounding the $s$--value $s(x;\ta)$, it is enough to deal with the case $x\ge 0$. Indeed, by symmetry around $0$ of both the prior and the likelihood, the Bayesian distribution of $(-X,-\te)$ is the same as that of $(X,\te)$, which implies, for $x<0$, that $\Pi[\te_i>0\given X_i =  x]=\Pi[-\te_i<0\given -X_i = -x]=\Pi[\te_i<0\given X_i = -x]$, so that, for any real $x$,
\[ s(x;\ta) = 2\Pi[\te_i<0\given X_i = |x|]. \]
In particular, for the next lemma deals we can restrict to the case $x\ge 0$. \begin{lemma} \label{lemsval}
Let $\Pi=\Pi_\ta$ be the horseshoe prior with parameter $\tau>0$.  Suppose $\te\sim \Pi$ in the one-dimensional model $X\given \te \sim \cN(\te,1)$. 
Then there exists $\ta_0>0$ such that the following estimates hold uniformly in $\tau\in(0,\ta_0]$,
\begin{itemize}
\item if $\,0\le x\le M_\ta$, for $\Phi(0)=1/2$ 
and $\rho_\ta(\cdot)$ such that $\ \sup_{0 \le x \le M_\ta} |\rho_\ta(x)|\le C \ta^{1/4}$,
\[  \Pi_\ta[\te<0 \given X=x]  
= 
\begin{dcases}
\ \Phi(0)(1+\rho_\ta(x)) & \qquad \text{if }\quad 0\le x\le \sqrt{2},\\
\ \displaystyle{\frac{ \Phi(0) }
{1+ \ta x \int_1^{x^2/2} e^vv^{-3/2} dv /(\sqrt{2}\pi)}
(1+\rho_\ta(x))} & \qquad \text{if }\quad x\in [\sqrt{2},M_\ta];\\
\end{dcases}
\] 
\item if $\,x>M_\ta$, for $C>0$  a large enough universal constant,
\[ 0\le  \Pi_\ta[\te<0 \given X=x]  \le C(1+\sqrt{\ta}) \frac{x^2}{\ta}\phi(x). \]
\end{itemize}
\end{lemma}

\begin{proof}
By Lemma \ref{lem:Sval},  for $x\ge 0$, one has $\Pi[ \te <0 \given X=x]=N(x)/D(x)$, where 
\[
N(x) = \int_0^{\infty} \Phi \left( - \dfrac{\la \tau x}{\sqrt{1 + (\la\tau)^2}}\right) \phi_{0,1 + (\la \tau)^2}(x) \frac{2}{\pi} \dfrac{d\la}{1 + \la^2}, 
\]
and
\[
D(x) = \int_0^{\infty}  \phi_{0,1 + (\la\tau)^2}(x) \frac{2}{\pi}\dfrac{d\la}{1 + \la^2}.
\]
In their study of the marginal maximum likelihood, the authors in \cite{vsv17} have studied the expansion of the denominator $D$, which corresponds to the marginal density of $X$ in the Bayesian model.  The expansion, slightly re-expressed for our needs, is given in Lemma \ref{lemde}.

 For the numerator $N$, we still use an expansion in this spirit,   but with a few main differences: to handle  the extra term featuring $\Phi$ that we have here, the split in terms of $x$ is different; this induces different remainder terms, and in particular the regime for large $x$ is new.  For the numerator $N(x)$, we change variables and set $1+\ta^2\la^2=1/(1-z)$, so that $N(x)=\cF(x) \phi(x)\ta/\pi$, with
\[ \cF(x) := \int_0^1 \Phi\left(- \sqrt{z} x\right) 
e^{z x^2/2} \frac{dz}{\sqrt{z}((1-z)\ta^2+z)}.
\]
We now distinguish the cases $|x|\le M_\ta$ and $|x|>M_\ta$, where 
$M_\ta$ is as in \eqref{defshort}. 
Let us also denote, for any $\tau\in[0,1]$,
\begin{equation} \label{itau}
\cI(\ta):=\int_0^{1/\ta} \frac{1}{\sqrt{u}}\frac{1}{1+(1-\ta^2)u} du.
\end{equation}

In the case $|x|\le M_\ta$, one splits the integral defining $\cF(x)$ as $\int_0^1 = \int_0^\ta+\int_\ta^{\frac{2}{x^2}\wedge 1} + \int_{\frac{2}{x^2}\wedge 1}^{1}$, and write the corresponding integrals $\cF(x)=\cF_1(x)+\cF_2(x)+\cF_3(x)$. We first study $\cF_1(x)$ and change variables $z=u\ta^2$ to get
\begin{align*}
\cF_1(x) = \frac{1}{\ta} &\int_0^{1/\ta} \Phi(-\ta x \sqrt{u}) e^{u\ta^2x^2/2}\frac{1}{\sqrt{u}}\frac{1}{1+(1-\ta^2)u} du.
\end{align*}
Note that $\ta x\sqrt{u}\le \sqrt{\ta}x\le \ta^{1/4}\le 1$ on the interval of integration,  using that $|x|\le M_\ta$. 
We then use $-w \le \Phi(-w) -\Phi(0) \le 0$ for $w\ge 0$, as $\Phi'=\phi$ is bounded by $1/\sqrt{2\pi}\le 1$ and $\Phi$ is increasing on $(-\infty,0]$; and also   $|e^w-1|\le we^w\le w e$ for $w\in[0,1]$. Putting these inequalities together,  
one gets, for $\ta$ sufficiently small,
\[ 
(\Phi(0)-\sqrt{\ta}x)(1-\ta x^2e /2)
 \frac{1}{\ta} \cI(\ta)
\le
\cF_1(x)
\le
\Phi(0)(1+\ta x^2 e/2)
 \frac{1}{\ta}  \cI(\ta),
\]
where $\cI(\ta)$ is defined in \eqref{itau}. Lemma \ref{lemitau} provides bounds on  $\cI(\ta)$, which put together with the last display, using $|x|\le M_\ta$, give
\[ 
(\Phi(0)-\ta^{1/4})(1-2\ta^{1/2})
 \frac{\pi}{\ta}(1-C\sqrt{\ta})
\le
\cF_1(x)
\le
\Phi(0)(1+2\ta^{1/2})
 \frac{\pi}{\ta}(1+C\sqrt{\ta}).
 \] 
That is, for small enough $\ta$, and constants $c_1,C_1>0$ large enough,
\[ \frac{\pi}{\ta}\Phi(0)(1-c_1\ta^{1/4}) 
\le \cF_1(x) \le \frac{\pi}{\ta}\Phi(0)(1+C_1\ta^{1/4}). \]
For the integrals $\cF_2(x)$ and $\cF_3(x)$, following \cite{vsv17}, one changes variables setting $zx^2/2=v$ to get
\[ \cF_2(x) = \sqrt{x^2/2} \int_{x^2\ta/2}^{x^2/2\wedge 1} \Phi(- \sqrt{2v}) e^v \frac{1}{\sqrt{v}}\frac{1}{\ta^2x^2/2+(1-\ta^2)v} dv,
\]
and similarly for $\cF_3(x)$. For $\cF_2(x)$, 
one bounds the denominator from below by $(1-\ta^2)v$, the term $e^v$ from above by $e$ and the term featuring $\Phi$ by $1$, and the remaining integral is bounded from above by $Cx\int_{x^2\ta/2}^1 dv/v^{3/2}\le C'\ta^{-1/2}$. Finally we now deal with $\cF_3(x)$.  One may assume that $2/x^2<1$, otherwise the integral equals $0$, so that
\[ \cF_3(x) =\sqrt{x^2/2} \int_{1}^{x^2/2} \frac{1}{\sqrt{v}} \frac{\Phi(- \sqrt{2v}) e^v}{\ta^2x^2/2+(1-\ta^2)v} dv. \]
By bounding the denominator from below in the last display by $(1-\ta^2)v$ and combining with the inequality $\Phi(-u)=\bar{\Phi}(u)\le \phi(u)/u$ for $u>0$, one obtains, for $\ta^2<1/2$,
\[ \cF_3(x) \le \sqrt{x^2} \int_{1}^{x^2/2} \frac{dv}{4 v^2}
\le |x|/4\le \ta^{-1/4}/4,\]
which is of smaller order than $\ta^{-1/2}$. Gathering the previous estimates, one obtains, uniformly over $|x|\le M_\ta$,  for small enough $\ta$, and large enough constants $c_2,C_2>0$ 
\[ \Phi(0)\phi(x)(1-c_2\ta^{1/4}) 
\le N(x) \le  \Phi(0)\phi(x)(1+C_2\ta^{1/4}). \]
Putting this together with the estimate for $D(x)$ above gives the first part of the statement. 

In the case $|x|>M_\ta$, we use the standard bound $\Phi(-w)=\bar{\Phi}(w)\le e^{-w^2/2}$ for $w>0$, to get
\[  \tau\cdot \cF(x) \le \int_0^1  \frac{\ta dz}{\sqrt{z}((1-z)\ta^2+z)}
= \int_0^{1/\tau^2}  \frac{du}{\sqrt{u}(1+(1-\ta^2)u)}
\le \int_0^\infty  \frac{du}{\sqrt{u}(1+(1-\ta^2)u)},
\]
where we again changed variables by setting $z=u\ta^2$. By the first part of Lemma \ref{lemitau}, the last display is upper-bounded by $\pi+O(\ta)$, so that $N(x)\le \phi(x)(1+O(\ta))$ in that case. 
For the denominator $D(x)$, by Lemma \ref{lemde}, we have $D(x)=C\ta x^{-2}(1+O(\sqrt{\ta}))$. 
Putting these two bounds together gives the upper bound in the second part of the statement.
The result follows by noting that $\cF(x)\ge 0$ by definition.
\end{proof}

\begin{lemma} \label{lemitau}
For any $\tau\in(0,1)$, we have
\[ 0 \le  \int_0^\infty \frac{1}{\sqrt{u}} \frac{du}{1+(1-\ta^2)u} -\pi \le  \pi \frac{\ta^2}{1-\ta^2}= O(\ta^2). \]
The integral $I(\tau)$ defined in \eqref{itau} verifies
\[   \pi \frac{\ta^2}{1-\ta^2} - 2\sqrt{\tau} \le  I(\tau) - \pi  \le \pi \frac{\ta^2}{1-\ta^2}.  \]
\end{lemma}
\begin{proof}
A similar result was obtained by \cite{vsv17} within their proof of Lemma C--9 therein; here we make the constants explicit, and provide a proof for completeness. First, one notes that  $(1-\ta^2)(1+u)\le 1+(1-\ta^2)u\le 1+u$. Since $\int_0^\infty du/(\sqrt{u}(1+u)) = \pi$,  this gives the first statement of the lemma. The second statement follows by noting that $\int_{1/\ta}^\infty du/(\sqrt{u}(1+u))\le \int_{1/\ta}^\infty u^{-3/2} du=2\sqrt{\ta}$.
\end{proof}
 
\begin{lemma}[variation on Lemma C.9 in \cite{vsv17}] \label{lemde}
Let $D(x)=  \int_0^{\infty}  \phi_{0,1 + (\la\tau)^2}(x) (1 + \la^2)^{-1} (2/\pi)d\la$. For any real $x$,
\begin{equation} \label{deix}
D(x)=(\pi/\ta + \sqrt{x^2/2} \int_1^{\max(x^2/2,1)} v^{-3/2}e^vdv)(1+O(\sqrt{\ta})) \phi(x)(\ta/\pi),
\end{equation}
where the $O$ term is uniform with respect to $x$. Further, uniformly over $|x|\ge \ta^{-1/4}$, 
\[ D(x) = \frac{2\ta}{\pi\sqrt{2\pi} x^2}(1+O(\sqrt{\ta})).\]   
\end{lemma} 
\begin{proof}
The lemma is a small variation on Lemma C.9 of \cite{vsv17} (the link to the notation $I_{1/2}$ used therein is explained in Section \ref{sec:margintec} below, see \eqref{linkde} therein). For the first part of the statement, we note that \eqref{deix} slightly differs from their expression in that we have replaced the upper value in the integration interval $x^2/2$ by $\max(x^2/2,1)$. If $x^2\ge 2$ the expression is the same; if $x^2\le 2$, we note that for any $s\in(0,1]$, the quantity $s^{1/2}\int_s^1 v^{-3/2}e^vdv$ remains bounded. For such $x$'s the estimate in \cite{vsv17} writes $D(x)=(\pi/\ta+O(1))(1+O(\sqrt{\ta}))\phi(x)\ta/\pi$. This can also be written $D(x)=(\pi/\ta)(1+O(\sqrt{\ta}))\phi(x)\ta/\pi$ (our estimate for such $x$'s) so that both estimates coincide up to the precise form of remainder terms.

For the second part of the statement, we directly apply the second part of Lemma C.9 in \cite{vsv17}  with $\veps_\ta$ therein taken equal to $\ta^{1/4}$.
\end{proof}

\begin{lemma}[Bernstein's inequality] \label{bernstein}
Let $Z_i, 1\le i\le N$ be centered independent variables with $|Z_i|\le M$ and $\sum_{i=1}^N \text{Var}(Z_i)\le V$. Then for all $t>0$,
\[ P\left[ \sum_{i=1}^N Z_i >t \right] \le \exp\left\{-\frac12 \frac{t^2}{V+Mt/3} \right\}. \]
\end{lemma}

\subsubsection{Proof of remaining Lemmas}
\label{app:remaininglemsval}

\begin{lemma}  
For any $\ta\in(0,\ta_0]$, any $x$ such that $\sqrt{2}\le x\le M_\ta$, the $s$--value is bounded as follows
\[ \frac{1-|O(\ta^{1/4})|}{1 +  \ta h(x)} \le s(x;\ta) \le \frac{1+O(\ta^{1/4})}{1 + \ta h(x)/4}, \]
where the $O$--term is uniform with respect to $x$.  Further, if $\ta\in[\ta_-,\ta_+]$ for some $0<\ta_-\le \ta_+\le \ta_0$, uniformly in $x\in[\sqrt{2},M_{\ta_+}]$,
\[ \frac{1-|O(\ta_+^{1/4})|}{1 +  \ta_{+} h(x)} \le s(x;\ta) \le \frac{1+O(\ta_+^{1/4})}{1 + \ta_- h(x)/4}. \]
\end{lemma}
\begin{proof}[Proof of Lemma \ref{lem:boundsv}]
Using the expression obtained in  Lemma \ref{lemsval}, it is enough to find bounds for the function $b(y)=\int_1^y e^v v^{-3/2}dv/(\sqrt{2}\pi)$. For any $y\ge 1$, using $\int_1^{y/2} v^{-3/2}\le 2$,
\begin{align*}
(\sqrt{2}\pi)b(y) & = \int_1^{y/2} e^v v^{-3/2}dv+\int_{y/2}^y e^v v^{-3/2}dv \\
& \le 2e^{y/2} + (2/y)^{3/2} \int_{y/2}^y e^vdv \le \left(2e^{-y/2}+\frac{\sqrt{8}}{y^{3/2}}\right)e^y.
\end{align*}
Using the inequality $e^{-y/2}\le 1.2 y^{-3/2}$ for $y\ge 1$, one gets that the last display is bounded from above by $(\sqrt{8}+2.4)e^y/y^{3/2}\le 5.3 e^y/y^{3/2}$, so that, with $h$ defined in \eqref{defh} and $x\ge \sqrt{2}$,
\[ xb(x^2/2)\le \frac{5.3}{\sqrt{2}\pi} \frac{2^{3/2}}{2\sqrt{2\pi}} h(x)\le 0.68h(x) \le h(x).\]
which gives the lower bound in the first inequality of the statement. On the other hand, for any $y\ge 1$,
\[ (\sqrt{2}\pi)b(y) \ge y^{-3/2} \int_1^y e^v dv= y^{-3/2} (e^y - 1)\ge y^{-3/2}e^y/2,\]
using that $e^y\ge 2$ when $y\ge 1$. This yields $xb(x^2/2)\ge h(x)/4$ for $x\ge \sqrt{2}$, leading to the upper bound in the first inequality of the statement.

The second part of the statement follows from the first by noting that, under the condition on $\ta$, all intervals $[\sqrt{2},M_\ta]$ contain $[\sqrt{2},M_{\ta_+}]$ and next bounding $\ta$ in the bounds.
\end{proof}

\begin{proof}[Proof of Lemma \ref{lemla}]
The proof is very similar to that of Lemma S--12 in \cite{cr20}: here $v$ plays the role of the variable $1/u$ therein; and here $\ga(u):=1/(u^2/2)$ plays the role of $g$ therein. Since one may take $v$ larger than a given large constant, in the arguments leading to the bounds, we only need to take an argument $x$ of $\ga(x)$ that is larger than $1$ so that we can use the bound $\|\ga\|_\infty=2$, where $\|\cdot\|_\infty$ here means the supremum norm on $[1,\infty)$. The only other properties of $g$ that are used in \cite{cr20} are that $x\to g(x)$ is decreasing for $x>0$ large enough and that $\log g$ is $\La$--Lipschitz for some $\La>0$, which both clearly hold for $\ga(u)=2/u^2$, so that we can follow the proof as in \cite{cr20}.
\end{proof}

\subsection{Proofs for the $S$--value procedure}
\label{app:proofSval}

\subsubsection{Proof of the second main result}
Similar to the case of $s$--values above, note that by symmetry in the sequence model, for any real $x$,
\[ S(x,\ta) = 2\Pi[\te_i<0\given X_i\ge |x|], \]
so that below, when bounding $S(x,\tau)$, it is enough to consider the case $x\ge 0$. \\

{\em Roadmap of the proof of Theorem \ref{thm:fdrcontrol}.} 
At a high-level, the proof for $S$--values roughly follows the steps of that for $s$--values. However, several steps need (much) sharper bounds, as one needs to control exact constants all along the proof. 

The first step consists in deriving precise estimates from the $S$--value, which is done in Section \ref{sec-bSval}. This enables one to derive a fairly precise approximate threshold for the procedure, namely the signal value from which it starts rejecting the null hypothesis. Next, one derives a precise concentration for the parameter $\hat\ta$ of the horseshoe procedure in Section \ref{sec:mmle2}: the result is stronger compared to that of Theorem \ref{thm:hsrisk}, with matching upper and lower bounds. Finally, upper and lower bounds with high probability are obtained both for the number of true positives TP and the  FDR itself.

\begin{proof}[Proof of Theorem \ref{thm:fdrcontrol}]
Define the function, for $u>0$, 
\begin{equation} \label{defH}
H(u)= 1/ \left(u \bar\Phi(u) \right). 
\end{equation}
The map $H$ is easily seen to be strictly increasing on $[1,+\infty)$, so that  $H$ admits an inverse $\La(\cdot)$
\begin{equation} \label{defLa}
\La(\cdot):=H^{-1}(\cdot),
\end{equation}
that is $\La(H(x))=x$ and  $H(\La(y))=y$ in the appropriate ranges of $x,y$.

Let us introduce the event, for $d=\sqrt{2}\Phi(b)/C^*$ and small $\delta>0$, \begin{equation} \label{eventoS}
\cA_S:= \left\{\hat\ta(X) \in \left[(d-\delta)\ta_n(s_n) , (d+\delta) \ta_n(s_n)\right] \right\}.
\end{equation}
Proposition \ref{propeq}, applied with $\delta$ replaced by $C^*\delta$, and the fact that $\Phi$ is $C$--Lipschitz with $C=1/\sqrt{2\pi}<1$ imply that $P_{\te_0}[\cA_S^c]=o(1)$.

In the sequel, in slight abuse of notation (we use the same notation with possibly different constants in the proof of Theorem \ref{thm:hsrisk}) we denote $\ta_-:= (d-\delta) \ta_n(s_n)$, $\ta_+= (d+\delta) \ta_n(s_n)$. Recall that $M_\ta=\ta^{-1/4}$ and $\mu_n=\log\log(n/s_n)$.\\

{\em FNR control.} The number of false negatives, for $\te_0\in\cL_0[s_n;b]$, is
\begin{align*} 
\sum_{i\in S_{\theta_0}} \1\{ S(X_i;\hat\ta)>t \}
& = \sum_{i\in S_{\theta_0}} \1\{ S(X_i;\hat\ta)>t \} \1_{|X_i|\in[\mu_n,\frac{M_{\hat\ta}}{2}]}
+\sum_{i\in S_{\theta_0}} \1\{ S(X_i;\hat\ta)>t \}\1_{|X_i|\notin[\mu_n,\frac{M_{\hat\ta}}{2}]} \\
& = \quad \qquad \qquad\qquad (I)  \qquad\qquad \quad + \quad\qquad \qquad\qquad (II). 
\end{align*}
Since $P[\cA_S^c]=o(1)$, one can restrict to the event $\cA_S$ for the analysis of these terms.
Dealing first with term (I), using the first part of Lemma \ref{lemSval} with $\ta=\hat\ta$ on the event $\cA_S$ on which $\hat\ta\in[\ta_-,\ta_+]$,
\begin{align*}
 (I)\1_{\cA_S} & \le \sum_{i\in  S_{\theta_0}} \1\left\{ \frac{1+o(1)}{1 + C^*\ta_- H(|X_i|)/2}>t \right\} \1_{|X_i|\in[\mu_n,M_{\hat\ta}/2]} \\
 & \le \sum_{i\in  S_{\theta_0}} \1\left\{ H(|X_i|)< \frac{2}{C^*\ta_-}\left(\frac{1-t}{t}+o(1)\right) \right\} \\ 
& \le \sum_{i\in  S_{\theta_0}} \1\left\{ |X_i| < \La\left(\frac{2}{C^*\ta_-}\left(\frac{1-t}{t}+o(1)\right) \right) \right\}.
 \end{align*}
 Denote $\ta_{t,+}:=\{2/(C^*\ta_-)\}\{(1-t)/t+o(1)\}$. 
For $i\in S_{\theta_0}$ and  if $\te_{0,i}>0$, it holds $\1\left\{ |X_i| < \La(\ta_{t,+}) \right\} \le 
\1\left\{ \veps_i < \La(\ta_{t,+})-\te_{0,i} \right\}$; similarly, one writes $\1\left\{ |X_i| < \La(\ta_{t,+}) \right\} \le 
\1\left\{ \veps_i > -\La(\ta_{t,+})-\te_{0,i} \right\}$ if $\te_{0,i}<0$. Assuming without loss of generality that $\te_{0,i}>0$, using $\te_0\in\cL_0[s_n,b]$,
\[ \1\left\{ |X_i| < \La(\ta_{t,+}) \right\} 
\le  \1\left\{ \veps_i < -b + \La(\ta_{t,+}) - \sqrt{2\log(n/s_n)} \right\}.
\]
Lemma \ref{techlata} applied with $\ta_{t,+}$, as in the last statement of the Lemma, gives  $\La(\ta_{t,+}) - \sqrt{2\log(n/s_n)} = o(1)$. 
From this one deduces that $E_{\te_0}[ (I)\1_{\cA_S}] \le s_n \Phi(-b+o(1))=s_n\{\bar\Phi(b)+o(1)\}$.

To bound from above the term (II), one 
splits the indicator 
$\1\{|X_i|\notin [\mu_n,M_{\hat\ta}/2]\}=\1\{|X_i|<\mu_n\} + \1\{|X_i|>M_{\hat \ta}/2]\}$. 
The contribution of the first term to the expectation $E_{\te_0}(II)$ is upper-bounded by, using $|\te_{0,i}|\ge \sqrt{2\log(n/s_n)}+b$ for $i\in S_\te$,
\[ s_n P\left[|\veps_1|>\sqrt{2\log(n/s_n)}+b-\mu_n\right]=o(s_n),\]
noting that the probability in the last display goes to $0$. 
The contribution of $\1\{|X_i|>M_{\hat\ta}/2]\}$ to (II) is bounded from above using the second part of Lemma  \ref{lemSval}: 
\begin{align*}
 \1\{ S(X_i;\hat\ta)>t, |X_i|>M_{\hat\ta} /2\}
& \le \1\{ (4/C^*) \phi(M_{\ta_+}/2)/{\ta_-} (1+o(1))>t , |X_i|>M_{\ta_+}/2\},\\
& \le  \1\{ (4/C^*) \phi(M_{\ta_+}/2)/{\ta_-} (1+o(1))>t \}
\end{align*}
and the last set is empty for $n/s_n$ large enough, using that $\phi(M_{\ta_+})/{\ta_-}=o(1)$ since
 $\ta_- \asymp  \ta_+ \asymp \ta_n(s_n)$. Putting the previous bounds together leads to, noting that the previous inequalities are uniform over $\cL_0[s_n;b]$  
  \[ \limsup_{n,s_n} \sup_{\te_0\in \cL_0[s_n;b]} \FNR(\te_0, \hat\vphi_t^S) \le \bar\Phi(b). \] 
  In passing, we note that this upper bound enjoys a  stronger uniformity compared to that in the statement of Theorem \ref{thm:fdrcontrol}, which holds over $\cL_{0,=}[s_n;b]$ only, and the convergence of the FNR is in fact uniform over the larger $\cL_0[s_n;b]$. The restriction to $\cL_{0,=}[s_n;b]$ will be needed for the lower bound on the FDR below. 
  
{\em Number of true positives.}   Denote by $TP$ the number of true positives
\[ TP := \sum_{i\in S_{\theta_0}} \1\{ S(X_i;\hat\ta)\le t \}.\]
We now show that  the following event $\cB_S$, with $c_1:=\Phi(b)$ and $\delta>0$ being the same as  in the definition \eqref{eventoS} of the event $\cA_S$,
\begin{equation} \label{defbS}
\cB_S = \left\{   TP \ge (c_1-\delta)s_n\right\}, 
\end{equation}
has a probability going to $1$.  The event can also be written $\{FN\le (c_2+\delta)s_n\}$, with $c_2=\bar\Phi(b)$, using that the number of false negatives FN equals $s_n-TP$. Using the notation above $FN=(I)+(II)$; also, we have already shown above that $(II)=o_P(s_n)$ on the event $\cA_S$. It is thus enough to bound the probability, using the same notation $\ta_{t,+}$ as above,
\[ P_{\te_0}\left[(I) \ge (c_2+\delta/2)s_n, \cA_S\right]
\le   P_{\te_0}\left[\sum_{i\in S_{\theta_0}}  
\1\left\{ |X_i| < \La\left(\ta_{t,+}\right)  \right\}\ge (c_2+\delta/2)s_n \right]. \]
Denoting by $\cN_1$ the random variable $ \sum_{i\in S_{\theta_0}} \1\{ |X_i|< \La(\ta_{t,+}) \}$, one notes that $\cN_1$ is a sum of independent Bernoulli variables of parameters, say, $a_i$. This is a Poisson-Binomial distribution $PBin(a_i)$ with parameters $a_i$, and by the same reasoning as for the FNR above,
\[ a_i = P_{\te_0}\left[|X_i| < \La(\ta_{t,+}) \right] \le \bar\Phi(b)+o(1)
=c_2+o(1), \]
where the $o(1)$ in fact does not depend on $i$. Since the Poisson-Binomial distribution is stochastically increasing in its parameters (Lemma S-25 in \cite{acr24}), we have that $PBin(a_i)$ is stochastically dominated by a binomial $Bin(s_n,\bar\Phi(b)+o(1))$ distribution -- denote by $\cN_2$ a variable following the latter --. Now one can bound from above, for small $\delta>0$, 
\begin{align*}
P_{\te_0}[\cB_S^c] & \le P[\cN_1\ge (c_2+\delta/2)s_n]+o(1) \le P[\cN_2 \ge (c_2+\delta/2)s_n]+o(1)\\
& \le P[\cN_2-E \cN_2 \ge (\delta/4)s_n]+o(1),
\end{align*}
for $n,s_n$ large enough, using that $E\cN_2=s_n(\bar\Phi(b)+o(1))=s_n(c_2+o(1))$. An application of Bernstein's inequality for the binomially distributed $\cN_2$ now gives $P_{\te_0}[\cB_S^c]=o(1)$ as desired.\\

{\em FDR control.} Let $FP$ denote the number of false positives of the $S$--value procedure and 
\begin{equation}\label{Deltadef}
 \Delta:= \sum_{i\notin S_{\theta_0}} \1_{|\veps_i|> M_{\ta_+}/2}. 
\end{equation} 
By definition of $\ta_+$, we have $E_{\theta_0}\Delta\le 2n\bar\Phi(\ta_+^{-1/4}/2)=o(s_n)$. Then
\begin{align*}
FP & \le  \sum_{i\notin S_{\theta_0}} \1\{ S(\veps_i;\hat\ta) < t \}\1_{|\veps_i|\le M_{\ta_+}/2} 
+ \Delta.
\end{align*}
The indicator in the last display can be bounded from above as follows, on the event $\cA_S$,
\begin{align*}
  \1\left\{ S(\veps_i,\hat\ta )<t ,  |\veps_i|\le M_{\ta_+}/2 \right\}
& \le \1\left\{  S(\veps_i,\hat\ta)<t \,,\, \mu_n\le |\veps_i|\le M_{\ta_+}/2 \right\} + \1\left\{  S(\veps_i,\hat\ta)<t ,  |\veps_i|<\mu_n\right\}
\\
& \le
\1\left\{  \frac{1+o(1)}{1+C^*\ta_+ H(|\veps_i|)/2}<t \right\} + \1\left\{ 1+o(1)<t \right\},
\end{align*}
where to bound  the first indicator in the sum we have used Lemma \ref{lemSval} (for $\ta=\hat\ta$, and noting $M_{\ta_+}\le M_{\hat\ta}$ on $\cA_S$), while for the second we use the first estimate of Lemma \ref{lemSval} for $x\in[0,\mu_n]$. Now the last indicator in the above display equals $0$ for large enough $n,s_n$ (uniformly in $i$) since $t<1$. This leads to, setting $\ta_{t,-}:=\{2/(C^*\ta_+)\}\{(1-t)/t+o(1)\}$, on the event $\cA_S$, 
\begin{align*}
 FP & \le \sum_{i\notin S_{\theta_0}} \1\left\{ H(|\veps_i|) > 2\frac{1-t-o(1)}{t C^*\ta_+} \right\} +\Delta\\
& \le \left[\sum_{i\notin S_{\theta_0}} \1\left\{|\veps_i| > \La(\ta_{t,-})\right\} +\Delta\right]=: F_+.
\end{align*} 

 Then, with $\cA_S, \cB_S$ the events in \eqref{eventoS}, \eqref{defbS} respectively, for $F_+$ as in the last display,
\begin{align*}
 FDR &= E_{\theta_0}\left[ \frac{FP}{(FP+TP)\vee 1} \right] 
\le E_{\theta_0}\left[\frac{F_+}{F_+ +(c_1-\delta)s_n} \1\{\cA_S \cap \cB_S \} \right] +o(1)\\
& \le E_{\theta_0}\left[\frac{F_+}{F_+ +(c_1-\delta)s_n}\right]+o(1) \le \frac{E_{\theta_0} F_+}{E_{\theta_0} F_+ +(c_1-\delta)s_n}+o(1),
\end{align*}
where for the last inequality one uses Jensen's inequality together with the fact that $x\to x/(x+d)$ is concave over $x\ge 0$, for any given $d>0$. One can further bound
\[ E_{\theta_0} F_+ \le nP[|\veps_1|>\La(\ta_{t,-})] + o(s_n). \]
Now recall that by definition $H(u)=1/(u\bar\Phi(u))$ and $H(\La(v))=v$, so that 
\begin{align*}
P[|\veps_1|& > \La(\ta_{t,-}) ]  = 2\bar\Phi(\La(\ta_{t,-}))
 = (2/\La(\ta_{t,-}))/H(\La(\ta_{t,-})) =2/\{ \ta_{t,-} \La(\ta_{t,-})\}\\
& = C^*\frac{t}{1-t}\frac{(d+\delta)\ta_n(s_n)}{\La(\ta_{t,-})}(1+o(1))
 =  \frac{s_n}{n}\frac{t}{1-t}\frac{\sqrt{2\log(n/s_n)}}{\La(\ta_{t,-})}(\Phi(b)+\delta_1)(1+o(1))\\
 & =  \frac{s_n}{n}\frac{t}{1-t}(\Phi(b)+\delta_1)(1+o(1)),
\end{align*}
where $\delta_1=\delta C^*/\sqrt{2}$ and where the last line follows from Lemma \ref{techlata}.

Putting the previous inequalities together leads to $E_{\theta_0} F_+\le s_n(t/(1-t))(\Phi(b)+\delta_1)(1+o(1))$. Going back to the first inequality on the FDR above, and using that $x\to x/(x+d)$ is increasing over $x\ge 0$, for any given $d>0$, one concludes that 
\begin{align*}
\FDR(\te_0,\hat\vphi_t^S)  & \le \frac{E_{\theta_0} F_+}{E_{\theta_0} F_+ +(c_1-\delta)s_n}+o(1)\\
& \le 
\frac{s_n(t/(1-t))(\Phi(b)+\delta_1)(1+o(1))}{s_n(t/(1-t))(\Phi(b)+\delta_1)+\Phi(b)s_n-\delta s_n+o(s_n)}+o(1)\\
&\le \frac{t+t\delta_1/\Phi(b)}{1+(t\delta_1-\delta(1-t))/\Phi(b)+o(1)}(1+o(1))+o(1).
\end{align*}
Similar to the FDR argument above, the previous bounds are uniform over $\te_0\in \cL_0[s_n;b]$, so that we can add a supremum over this set at the beginning of the last display. 
By letting $n,s_n$ go to infinity, this implies
 \[ \limsup_{n,s_n} \sup_{\te_0\in\cL_0[s_n;b]} 
 \FDR(\te_0, \hat\vphi_t^{S}) \le  \frac{t+t\delta_1/\Phi(b)}{1+(t\delta_1-\delta(1-t))/\Phi(b)}. \] 
 Recalling that $\delta_1\leqa\delta$ and letting $\delta$ go to $0$ in the last display gives
 \[ \limsup_{n,s_n} \sup_{\te_0\in\cL_0[s_n;b]} 
 \FDR(\te_0, \hat\vphi_t^{S}) \le t.\]
We note in passing that  without checking uniformity in the previous inequalities, one could also have obtained uniformity over $\cL_{0,=}[s_n;b]$ (which is the class from the statement of the Theorem, and is {\em smaller}  than $\cL_0[s_n;b]$) by a direct argument, as follows.

Indeed, for two different $\te_1,\te_2\in \cL_{0,=}[s_n;b]$, one has $\FDR(\te_1,\hat\vphi_t^{S})=\FDR(\te_2,\hat\vphi_t^{S})$: this is because by definition of the class, $\te_1, \te_2$ only differ by the labelling of the coordinates, since signals are the same. On the other hand, if one relabels the coordinates, then the procedure remains identical (since the prior distribution, which makes the coordinates iid, treats the coordinates symmetrically; the same holds for the value of $\hat\ta$), up to the same relabelling of coordinates.   So $\sup_{\te_0\in\cL_{0,=}[s_n,b]} \FDR(\te_0, \hat\vphi_t^{S})$ equals the value  $\FDR(\te_0, \hat\vphi_t^{S})$ for any given $\te_0\in \cL_{0,=}[s_n,b]$, so that this is another way of proving
 \[ \limsup_{n,s_n} \sup_{\te_0\in\cL_{0,=}[s_n,b]} 
  \FDR(\te_0, \hat\vphi_t^{S}) \le t.\]  

This shows that both FDR and FNR of $\hat\vphi_t^S$ are asymptotically bounded {\em from above} by the limits appearing in the statement of Theorem \ref{thm:fdrcontrol}. To complete the proof, it remains to show the corresponding lower bounds. 
 
{\em Lower bound on the FNR.}  Instead of proving the lower bound `by hand', we provide a more conceptual proof based on the notion of sparsity preserving procedure introduced in \cite{acr24}, and that we recall in Section \ref{sec:sparsitypres} below. 
 Let $\cS=\cS(\Theta)$ be the class of all sparsity preserving procedures over $\Theta=\cL_{0,=}[s_n;b]$ up to a constant $C$ as in Proposition \ref{sparsity-pres}. Then
 \begin{equation} \label{fnr-spapres}
  \inf_{\vphi\in \cS(\Theta)}\sup_{\te\in\Theta}\ \FNR(\te,\vphi) = \bar\Phi(b)+o(1).
 \end{equation} 
 This fact is a minor variation on Theorem 3 in \cite{acr24}: the only difference with the more general statement therein is that here we consider a more restrictive class of signals with all coordinates equal to $\sqrt{2\log{n/s_n}}+b$ (instead of `greater or equal to' therein); however, it is easy to check that the lower-bound part of Theorem 3 in \cite{acr24} continues to hold with a supremum $\te\in\Theta$ only, since the prior distribution used for the proof draws signals exactly from the `boundary' class $\Theta=\cL_{0,=}[s_n;b]$. The upper-bound part of Theorem 3 in \cite{acr24}  continues to hold here as well, since the class therein is larger, so the supremum is larger as well.  
 
Combining \eqref{fnr-spapres} and that the $S$--value procedure is sparsity-preserving (Proposition \ref{sparsity-pres}),
\[ \liminf_{n,s_n}  \sup_{\te\in\Theta}\ \FNR(\te,\hat\vphi_t^{S}) \ge \bar\Phi(b).\]
Putting this together with the upper bound obtained above yields
\[\sup_{\te\in\Theta}\ \FNR(\te,\hat\vphi_t^{S}) = \bar\Phi(b)+o(1). \]
Finally, similar to what we noted for the FDR above, note that the supremum in the last display in fact takes the same value for any given $\te\in \Theta=\cL_{0,=}[s_n;b]$.

{\em Lower bound for the FDR.} 
The arguments below ressemble the ones for the upper bound, except we work `in probability' instead of `in expectation'. We give the details for completeness. 
 The number $FP$ of false positives of the S--value procedure verifies
\begin{align*}
FP & = \  \sum_{i\notin S_{\theta_0}} \1\{ S(\veps_i;\hat\ta) < t \}
\ge \sum_{i\notin S_{\theta_0}} \1\{ S(\veps_i;\hat\ta) < t \}\1_{\mu_n<|\veps_i|\le M_{\hat \ta}/2} 
\end{align*}
The indicator in the last display can be bounded from below as follows, on the event $\cA_S$,
\begin{align*} 
  \1&\left\{ S(\veps_i,\hat\ta )<t ,  \mu_n<|\veps_i|\le M_{\hat \ta}/2 \right\}
 \ge \1\left\{   \frac{1+o(1)}{1+C^*\ta_- H(|\veps_i|)/2}<t \,,\, \mu_n\le |\veps_i|\le M_{\ta_+}/2 \right\} 
\\
& \ge \1\left\{  H(|\veps_i|) > \frac{2}{C^*\ta_-}\{(1-t)/t+o(1)\}  \,,\, \mu_n\le |\veps_i|\le M_{\ta_+}/2  \right\}\\
& \ge \1\left\{  |\veps_i| > \Lambda(\ta_{t,+}) \,,\, \mu_n\le |\veps_i|\le M_{\ta_+}/2  \right\},
\end{align*}
where we recall the notation $\ta_{t,+}=\{2/(C^*\ta_-)\}\{(1-t)/t+o(1)\}$, 
and where one uses Lemma  \ref{lemSval} together with the fact that the inverse $\Lambda$ of $H$ is increasing. The first inequalities of Lemma \ref{techlata} applied to $\ta_{t,+}$ give that for $n$ large enough, $\mu_n<\Lambda(\ta_{t,+})<M_{\ta_+}/2$, so that
\begin{align*}
\1& \left\{  |\veps_i| > \Lambda(\ta_{t,+}) \,,\, \mu_n\le |\veps_i|\le M_{\ta_+}/2  \right\}
 = \1\left\{  |\veps_i| > \Lambda(\ta_{t,+}) \,,\,  |\veps_i|\le M_{\ta_+}/2  \right\}\\
&=\1\left\{  |\veps_i| > \Lambda(\ta_{t,+})\right\}-
\1\left\{  |\veps_i| > \Lambda(\ta_{t,+}), |\veps_i|> M_{\ta_+}/2  \right\}\\
&  = \1\left\{  |\veps_i| > \Lambda(\ta_{t,+})\right\}-
\1\left\{ |\veps_i|>M_{\ta_+}/2  \right\}.
\end{align*}
This implies that, on the event $\cA_S$, recalling the definition \eqref{Deltadef} of $\Delta$,
\begin{align*}
 FP & \ge 
 \left[\sum_{i\notin S_{\theta_0}} \1\left\{|\veps_i| > \La(\ta_{t,+})\right\} -\Delta\right]=: F_-,
\end{align*} 
where we have already seen earlier that $E_{\theta_0}\Delta=o(s_n)$ (so that, also  $\Delta=o_{P_{\theta_0}}(s_n)$).

Similarly, one obtains a lower bound on the number of false negatives FN: 
arguing as for the upper-bound on FN obtained above, for $\te_0\in\cL_0[s_n;b]$ and on the event $\cA_S$,
\begin{align*} 
FN = \sum_{i\in S_{\theta_0}} \1\{ S(X_i;\hat\ta)>t \}
& \ge \sum_{i\in S_{\theta_0}} \1\{ S(X_i;\hat\ta)>t \} \1_{|X_i|\in[\mu_n,M_{\hat\ta}/2]}.
\end{align*}
Using  the first part of Lemma \ref{lemSval}, one can further bound from below, on $\cA_S$,
\begin{align*} 
  \1&\{ S(X_i;\hat\ta)>t \}  \1_{|X_i|\in[\mu_n,M_{\hat\ta}/2]}
   \ge 
   \1\left\{ \frac{1+o(1)}{1 + C^*\ta_+ H(|X_i|)/2}>t \right\} \1_{|X_i|\in[\mu_n,M_{\hat\ta}/2]} \\
& \ge  \1\left\{ |X_i| < \La\left(\frac{2}{C^*\ta_+}\left(\frac{1-t}{t}+o(1)\right) \right) \right\}
- \1\left\{ |X_i|\le \mu_n \right\} - \1\left\{ |X_i|> M_{\ta_+}/2 \right\}.
 \end{align*}
Recalling the notation $\ta_{t,-}:=\{2/(C^*\ta_+)\}\{\left(1-t\right)/t+o(1)\}$, one obtains
\begin{align*}
 FN & \ge 
 \left[\sum_{i\in S_{\theta_0}} \1\left\{|X_i| < \La(\ta_{t,-}) \right\} -\Delta_1-\Delta_2\right]=: G_-,
\end{align*} 
with $\Delta_1=\sum_{i\in S_{\theta_0}}  \1\left\{ |X_i|\le \mu_n \right\} 
, \Delta_2 = \sum_{i\in S_{\theta_0}} \1\left\{ |X_i|> M_{\ta_+}/2 \right\}$. As before, this time using that $\theta_0\in\cL_{0,=}[s_n]$, we have that $\Delta_1,\Delta_2$ are $o_{P_{\theta_0}}(s_n)$. 
 
Let us introduce the event, for $r$ a sufficiently large multiple of $\delta$, \begin{equation} \label{events}
\cC_S := \left\{\, FP \ge  \{t/(1-t)\}(\Phi(b)-r)s_n,\quad FN\ge (1-\Phi(b)-r)s_n\,\right\}.
\end{equation}
Similar to the control of the events $\cA_S, \cB_S$ above, we verify below that $P_{\theta_0}[\cC_S]=1+o(1)$ using Bernstein's inequality. Assuming for a moment that this has been obtained, the lower bound on the FDR is obtained as follows, using $TP=s_n-FN$ and recalling $c_1=\Phi(b)$,
\begin{align*} 
 \FDR(\te_0,\hat\vphi_t^{S}) &= E_{\theta_0}\left[ \frac{FP}{(FP+TP)\vee 1} \right] 
\ge E_{\theta_0}\left[\frac{\{t/(1-t)\}(c_1-r)s_n }{\{t/(1-t)\}(c_1-r)s_n +(c_1+r)s_n} \1\{\cC_S \} \right] \\
& \ge \frac{t(c_1-r)s_n }{t(c_1-r)s_n +(1-t)(c_1+r)s_n}P_{\theta_0}[\cC_S]
= \frac{t(1-r/c_1)}{1 + r(1-2t)/c_1}(1+o(1)).
\end{align*}
By letting $n,s_n$ go to infinity, this implies, for $\te_0\in\cL_{0,=}[s_n,b]$, 
 \[ \liminf_{n,s_n} \ \FDR(\te_0, \hat\vphi_t^{S}) \ge  \frac{t-tr/c_1}{1+r/c_1}. \] 
 Recalling that $r\leqa\delta$, that the FDR in the last display takes the same value for any given $\te_0\in\cL_{0,=}[s_n;b]$, and letting $\delta$ go to $0$ in the last display,  gives 
 \[ \liminf_{n,s_n} \ \sup_{\te_0\in \cL_{0,=}[s_n;b]}\  \FDR(\te_0, \hat\vphi_t^{S}) \ge t.\]
  
It now remains to check that $P_{\theta_0}[\cC_S]=1+o(1)$. Up to choosing a slightly larger constant $r$, given the lower bounds obtained above for FP and FN, one can replace these quantities respectively by $F_-$ and $G_-$ and, since $\Delta, \Delta_1, \Delta_2$ are all $o_{P_{\theta_0}}(s_n)$, one is left with studying two sums of independent Bernoulli variables. Considering first the sum appearing in $F_-$, it has expectation
\[ E_{\theta_0}  \sum_{i\notin S_{\theta_0}} \1\left\{|\veps_i| > \La(\ta_{t,+})\right\} = 2(n-s_n)\bar\Phi(\La(\ta_{t,+})).\]
Similar to the computation above on $2\bar\Phi(\La(\ta_{t,-}))$, using the definition of $H$ and $\La$ one gets
 \begin{align*}
 2\bar\Phi(\La(\ta_{t,+}))
& = (2/\La(\ta_{t,+}))/H(\La(\ta_{t,+})) 
=2/\{ \ta_{t,+} \La(\ta_{t,+})\}\\
& = C^*\frac{t}{1-t}\frac{(d-\delta)\ta_n(s_n)}{\La(\ta_{t,+})}(1+o(1))\\
& =  \frac{s_n}{n}\frac{t}{1-t}\frac{\sqrt{2\log(n/s_n)}}{\La(\ta_{t,+})}(\Phi(b)-\delta_1)(1+o(1))\\
 & =  \frac{s_n}{n}\frac{t}{1-t}(\Phi(b)-\delta_1)(1+o(1)),
\end{align*}
where $\delta_1=\delta C^*/\sqrt{2}$ and where the last line follows from the second assertion of Lemma \ref{techlata}. Now if one chooses the constant $r$ in \eqref{events} strictly larger than $\delta_1$ (e.g. $r=2\delta_1=\sqrt{2}C^*\delta$) then an application of Bernstein's inequality (Lemma \ref{bernstein}; the argument being similar to the one for the events $\cA_S, \cB_S$ above, details are left to the reader) with $t$ a small multiple of $s_n$ gives that the probability that the number of false positives FP is bounded from below as in the definition of $\cC_S$ has probability at least $1-\exp\{-cs_n\}$ for small enough $c>0$.

In order to control the sum appearing in $G_-$, it is enough to control
\[E_{\theta_0} \left[\sum_{i\in S_{\theta_0}} \1\left\{|X_i| < \La(\ta_{t,-}) \right\}\right]. \]
For $i\in S_{\theta_0}$ and  if $\te_{0,i}>0$, it holds $\1\left\{ |X_i| < \La(\ta_{t,-}) \right\} \le 
\1\left\{ \veps_i < \La(\ta_{t,-})-\te_{0,i} \right\}$; similarly, one writes $\1\left\{ |X_i| < \La(\ta_{t,-}) \right\} \le 
\1\left\{ \veps_i > -\La(\ta_{t,-})-\te_{0,i} \right\}$ if $\te_{0,i}<0$. Assuming without loss of generality that $\te_{0,i}>0$, using $\theta_0\in\cL_0[s_n,b]$,
\[ \1\left\{ |X_i| < \La(\ta_{t,-}) \right\} 
\le  \1\left\{ \veps_i < -b + \La(\ta_{t,-}) - \sqrt{2\log(n/s_n)} \right\}.
\]
Lemma \ref{techlata} gives $\La(\ta_{t,-}) - \sqrt{2\log(n/s_n)} =o(1)$.  
From this one deduces that the $E_{\te_0}$--expectation of the last display  is less than $s_n \Phi(-b+o(1))=s_n\{\bar\Phi(b)+o(1)\}$. By the same argument as for $F_-$ above, one can apply Bernstein's inequality to get that the number of false negatives FN verifies the lower bound in \eqref{events} with probability at least $1-\exp\{-cs_n\}$,  for small enough $c>0$. Combining with the result for FP above yields
 $P_{\theta_0}[\cC_S]=1+o(1)$ as desired.

 This concludes the proof of Theorem \ref{thm:fdrcontrol}.

\end{proof}

\subsubsection{Bounds for the $S$--value} \label{sec-bSval}

\begin{lemma} \label{lemSval}
Let $\Pi=\Pi_\ta$ be the horseshoe prior with parameter $\tau>0$.  Suppose $\te\sim \Pi$ in the one-dimensional model $X\given \te \sim \cN(\te,1)$. 
Then there exists $\ta_0>0$ such that the following estimates hold uniformly in $\tau\in [1/n,(s_n/n)\log(n/s_n)]=:\cT_n$, for $\mu_n:=(\log\log(n/s_n))^{1/2}$,
\begin{itemize}
\item if $0\le x\le M_\ta/2$, for 
$C^*=(2/\pi)^{3/2}$,

\[  \Pi_\ta[\te<0 \given X\ge x]  
= 
\begin{dcases}
\ (1/2)(1+\rho_\ta(x)) & \qquad \text{if }\quad 0\le x\le \mu_n,\\
\ \displaystyle{\frac{  \bar\Phi(x)}
{2\bar\Phi(x)+ C^*\ta/x}
(1+R_\ta(x))} & \qquad \text{if }\quad x\in [\mu_n,M_\ta/2],\\
\end{dcases}
\]
where  $\sup_{\ta\in\cT_n}\{ \sup_{0 \le x \le \mu_n} |\rho_\ta(x)|+\sup_{\mu_n \le x \le M_\ta/2} |R_\ta(x)|\}=o(1)$;
\item if $x>M_\ta/2$,
\[ 0\le  \Pi_\ta[\te<0 \given X \ge x]  \le (2/C^*) \ta^{-1}\phi(M_\ta/2) (1+Q_\ta(x)), \]
where  $\sup_{\ta\in\cT_n}\sup_{x > M_\ta/2} |Q_\ta(x)|=o(1)$.

 In particular, $\Pi[\te<0 \given X \ge x]=o(1)$ uniformly in $\ta\in \cT_n$ and $x>M_\ta/2$.
\end{itemize}
\end{lemma}
\begin{proof}
From formula \eqref{condexp}, we see that it is enough to bound the functions $\int_x^\infty N(y)dy$ and $\int_x^\infty D(y)dy$ respectively.  

Starting with the numerator, we can gather the two bounds obtained for $N$ within the proof of Lemma \ref{lemsval} in the regimes $0\le x\le M_\ta$ and $x\ge M_\ta$. For small enough $\ta$, uniformly in $0\le y\le M_\ta$, 
\begin{equation} \label{techin}
 \phi(y)(1-C_2\ta^{1/4})/2 \le N(y)\le \phi(y)(1+C_2\ta^{1/4})/2. 
\end{equation} 
Also, for $y>M_\ta$, we obtained earlier that $N(y)\le \phi(y)(1+O(\ta))$. 
Integrating this gives, uniformly over $\ta\in \cT_n$, for any $x\in[0,M_\ta]$,
\begin{align*}
 \int_x^\infty N(u) du & \le \int_x^{M_\ta} \phi(u)(1+C_2\ta^{1/4})du /2
 + \int_{M_\ta}^\infty \phi(u)(1+O(\ta))du \\
 & \le (1/2)(\bar\Phi(x)-\bar\Phi(M_\ta))(1+o(1))+\bar\Phi(M_\ta)(1+o(1)).
\end{align*}
For $x\le M_\ta/2$, this gives as upper-bound $(1/2)\bar\Phi(x)(1+o(1))$, noting that $\bar\Phi(M_\ta)=o(\bar\Phi(M_\ta/2))$ uniformly over the range of $\ta$'s considered, and next using $\bar\Phi(M_\ta/2)\le \bar\Phi(x)$ for $x\le M_\ta/2$. For $x\in[M_\ta/2,M_\ta]$, the bound above gives 
$\int_x^\infty N(u) du  \le (1/2)(\bar\Phi(x)+\bar\Phi(M_\ta))(1+o(1))\le \bar\Phi(x)(1+o(1))$. For $x>M_\ta$, we use the uniform bound $N(y)\le \phi(y)(1+O(\ta))$ which gives $\int_x^\infty N(u) du  \le \bar\Phi(x)(1+o(1))$ as well. In order to bound the numerator from below, using now the lower bound in \eqref{techin}, one obtains for any $x\in[0,M_\ta/2]$,
\[ \int_x^\infty N(u)du \ge \int_x^{M_\ta} N(u)du\ge  
(1+o(1))(\bar\Phi(x) - \bar\Phi(M_\ta))/2.
\]
Arguing as we just did for the upper-bound, we also have $\bar\Phi(M_\ta)=o(\bar\Phi(x))$ for such $x$'s, so that $\int_x^\infty N(u)du \ge (1+o(1))\bar\Phi(x)/2$ for any $x\in[0,M_\ta/2]$.
 
For the denominator, in the regime $x\ge\mu_n$,  one combines the expression \eqref{linkde} of $D$ with Lemma \ref{lemprec} applied with $\veps_\ta=(\mu_n/2)^{-1}$: since the remainder term $r_\ta$ is bounded  uniformly with respect to $x\ge \mu_n$ by $O(\sqrt{\ta}+\mu_n^{-2})=o(1)$,  one can integrate the expansion of the Lemma to obtain, for any $x\ge\mu_n$,
\begin{align*}
 \int_x^\infty D(u)du & = \frac{\ta}{\pi} \int_x^\infty \phi(u)I_{-1/2}(u)du= 
 \int_x^\infty (\phi(u) + \frac{\sqrt{2}\tau}{\pi^{3/2}} u^{-2})(1+r_\ta(u))du
 \\
 & = \left(\bar\Phi(x)+ \frac{\ta}{x}C^*/2 \right)(1+o(1)).  
\end{align*}
When $0\le x\le \mu_n$, one uses the expression \eqref{deix} of $D$ from Lemma \ref{lemde}. The integral therein is nonnegative, and bounded from above by $e^{\mu_n^2/2}\int_1^{1\vee (x^2/2)} v^{-3/2}dv\le 2e^{\mu_n^2/2}$, which is always a $o(\pi/\ta)$, uniformly in $\ta\in\cT_n$ and $x\in[0,\mu_n]$. This shows that $D(x)=\phi(x)(1+o(1))$ uniformly for the considered $x$'s and $\tau$'s.  From this one deduces that,   for $x\in[0,\mu_n]$, 
\begin{align*} 
 \int_x^{\mu_n} D(y)dy & = \int_x^{\mu_n}\phi(y)dy + \int_x^{\mu_n} \phi(y) o(1)dy\\
& =\bar\Phi(x)-\bar\Phi(\mu_n)+\bar\Phi(x)o(1).
\end{align*} 
On the other hand, using the estimate obtained in the one but last display with $x=\mu_n$,  
 we have $\int_{\mu_n}^\infty D(u)du=
(\bar\Phi(\mu_n)+ \ta\mu_n^{-1} C^*/2 )(1+o(1))=\bar\Phi(\mu_n)(1+o(1))$ using that $\ta/\mu_n\le (s_n/n)\log(n/s_n)/\mu_n=o(\bar\Phi(\mu_n))$ for our choice of $\mu_n$.  Gathering the previous estimates gives, for $0\le x\le \mu_n$,
\[  \int_x^\infty D(u)du = \bar\Phi(x)(1+o(1)).  \] 
Combining the just obtained respective bounds for numerator and denominator in \eqref{condexp} gives the desired result in the regime $0\le x\le M_\ta/2$.

For $x>M_\ta/2$, note first that  $\Pi[\te<0\given X\ge x]\ge 0$ by definition. The above bounds give
\[ \Pi[\te<0\given X\ge x] \le \frac{\bar\Phi(x)}{\bar\Phi(x)+ \ta C^*/(2x)}(1+o(1))\le 
\frac{2x}{C^*\ta}\bar\Phi(x)(1+o(1)).
\]
Using the standard bound $\bar\Phi(x)\le \phi(x)/x$ for $x>0$ and next that $\phi$ is decreasing on $(0,+\infty)$  gives the last upper-bound in the Lemma. The last part of the statement now follows by noting that $\ta\to \ta^{-1}\phi(M_\ta/2)$ goes to $0$ as $\ta\to 0$. 
\end{proof}  

\subsubsection{Sparsity preserving procedures} \label{sec:sparsitypres}

A multiple-testing procedure $\vphi=\vphi(X)\in\{0,1\}^n$ is said to be {\em sparsity preserving} (\cite{acr24}, Definition 1) over a subset $\Theta$ of sparse vectors $\ell_0[s_n]$  up to a multiplicative factor $(A_n)$ if, as $n\to\infty$,
\[ \sup_{\te\in \Theta} P_\te\left[ \sum_{i=1}^n \vphi(X) > A_ns_n \right] \to 0. \]
This property simply expresses that the procedure $\vphi_n$ does not overshoot the true sparsity $s_n$ by a too large multiplicative factor, and can be seen as a (very) weak form of type I error control, in the sense that if it holds, then the number of false discoveries must be at most $A_n s_n$, so not much larger than $s_n$ if $A_n$ is a constant or grows slowly.
 \begin{proposition} \label{sparsity-pres}
For any fixed $t\in(0,1/2)$, the procedure $\hat\vphi^{S}_t(X)$ is sparsity-preserving over $\Theta=\cL_{0,=}[s_n;b]$ up to a large enough (fixed, i.e. non--$n$ dependent) multiplicative constant $C>0$. 
\end{proposition} 
\begin{proof}
Within the (first part of the) proof of Theorem \ref{thm:fdrcontrol} the following has been shown on the number FP of false positives of the S--value procedure, for $\te$ any element of $\Theta=\cL_{0,=}[s_n;b]$: 
  $FP\le F_+$ on the event $\cA_S$ and 
\[E_\te F_+ \le nP[|\veps_1|>\La(\ta_{t,-})] + o(s_n) \le 
 s_n\frac{t}{1-t}(\Phi(b)+\delta_1)(1+o(1))+o(s_n)\le Cs_n.\]
Similarly, one sees that $\text{Var}_\te F_+ \le Cs_n$, so that $P_\te[F_+>2Cs_n, \cA_S]=o(1)$  by Tchebychev's inequality. Since $P_{\te}[\cA_S^c]=o(1)$, one concludes $P_\te[FP>2Cs_n]=o(1)$. Since the total number of discoveries is at most $TP+FP\le s_n+FP$, one deduces that the total number of discoveries is at most $(2C+1)s_n$ with probability going to $1$, so that the procedure is sparsity preserving, as desired.  
\end{proof}
 

\subsubsection{Technical Lemmas}

\begin{lemma} \label{lemLa}
For any $v>0$ large enough, the following bounds hold for $\La(\cdot)$ as in \eqref{defLa}
\[ 
\sqrt{2\log{v}-\log(8\pi)}\le \La(v)  \le \sqrt{2\log{v}-\log(2\pi)}.
\]
\end{lemma}
\begin{proof}
By definition of $\La$ as the reciprocal of $H$, we have $H(\La(v))=v$ that is $v^{-1}=\La(v)\bar\Phi(\La(v))$. By using the bound $u\bar\Phi(u)\le\phi(u)$ for $u>0$, one deduces taking logarithms that $-\La(v)^2/2\ge \log(\sqrt{2\pi}/v)$. Rearranging the expression gives the upper-bound in the Lemma. The lower bound follows similarly by using the standard bound $u\bar\Phi(u)\ge(1-1/u^2)\phi(u)$ which we use for $u>\sqrt{2}$ in the form $u\bar\Phi(u)\ge \phi(u)/2$.
\end{proof}  
 
\begin{lemma} \label{techlata}
For any fixed $t\in(0,1)$, there exists a constant $C=C(t)>0$ such that
\[ \left| \La(\ta_{t,-})  -  \sqrt{2\log(n/s_n)}  \right| 
\le \frac{C+\log\log(n/s_n)}{\sqrt{2\log(n/s_n)}}, \] 
where we recall $\ta_{t,-}=\{2/(C^*\ta_+)\}\{(1-t)/t+o(1)\}$ and $\ta_+= (d+\delta) (s_n/n)\sqrt{\log(n/s_n)}$. In particular, as $n,s_n\to\infty$ with $n/s_n\to\infty$, 
\[
\frac{\sqrt{2\log(n/s_n)}}{\La(\ta_{t,-})} \sim 1,
\]
The statements in both displays of the Lemma also hold with $\ta_{t,-}$ replaced by $\ta_{t,+}=\{2/(C^*\ta_-)\}\{(1-t)/t+o(1)\}$ with $\ta_-=(d-\delta) (s_n/n)\sqrt{\log(n/s_n)}$.
\end{lemma}
\begin{proof}
The result follows by combining Lemma \ref{lemLa} with the definitions of $\ta_{t,-}$ and  $\ta_+$, and similarly for the equivalent with $\ta_{t,+}$.
\end{proof}

\subsection{Concentration results for the MMLE $\hat{\tau}$} \label{sec:estimtau}

In this Section, we study a data-driven choice of $\tau$, namely the marginal maximum likelihood estimator. Recall that following \cite{vsv17}, we set, for $h_\ta$ the horseshoe density with parameter $\ta$,
\begin{equation*} 
\hat\ta = \hat\ta(X) = \underset{\ta\in[1/n,1]}{\text{argmax}}\, \prod_{i=1}^n \int \phi(X_i-\te)
h_\ta(\te)d\te.
\end{equation*}
We shall use a number of results on the marginal likelihood function that were established in \cite{vsv17}. For the reader's convenience these results are gathered in Section \ref{sec:margintec}.

\subsubsection{Proof of Proposition \ref{propmle}, used for $s$--values}

Let us recall the notation
\begin{equation*}
\ta_n(s) = \frac{s}{n}  \sqrt{\log(n/s)},
\end{equation*}  
 for $1\le s\le n$.  
Proposition \ref{propmle}, which we now prove, states that for 
 $\hat\ta$ the marginal maximum likelihood given by \eqref{defmmle} and $b$ a fixed, arbitrary, real number, there exist positive constants $d_1 < d_2$ such that,  for $\cL_0[s_n;b]$ the class of signals \eqref{defclb},
\[ 
\sup_{\te_0 \in  \cL_0[s_n;b]} P_{\te_0}\Big( \hat\ta(X)
\notin \left[d_1 \ta_n(s_n) , d_2 \ta_n(s_n)\right] \Big) =o(1).
\]
\begin{proof}[Proof of Proposition \ref{propmle}]
The fact that the bound $\hat\ta(X)\le d_2 \ta_n(s_n)$ holds with overwhelming probability is Theorem 1 in \cite{vsv17}, which gives the result uniformly over all sparse vectors $\te_0\in \ell_0[s_n]$. For $\te_0\in \cL_0[s_n;b]$, we show that this result is in fact sharp by proving that the lower bound $\hat\ta(X)\ge d_1 \ta_n(s_n)$ holds with high probability. 

Let us first gather some useful notation from \cite{vsv17}. The log-marginal likelihood is 
\begin{equation} \label{Mtau}
 \mathbb{M}_\ta(X):=\sum_{i=1}^n \log \int \phi(X_i-\te) h_\ta(\te)d\te.
\end{equation}
Its derivative with respect to $\ta$ can be written, according to Lemma C.1 in \cite{vsv17}, in the form
\begin{align} \label{decompM}
\frac{\partial}{\partial\ta}  \mathbb{M}_\ta(X)  = \frac{1}{\ta}\sum_{i=1}^n m_\ta(X_i)  & = \frac{1}{\ta}\sum_{i\in S_0^c} m_\ta(X_i) + \frac{1}{\ta}\sum_{i\in S_0} m_\ta(X_i)\\
& =: \cM_0(X,\ta)+\cM_1(X,\ta), 
\end{align}
recalling that $S_0=S_{\te_0}$ denotes the support of $\te_0$, with $S_0^c=\{1,\ldots,n\}\setminus S_0$ and where 
\begin{equation} \label{mtau}
 m_\ta(x)=x^2\frac{I_{1/2}(x)-I_{3/2}(x)}{I_{-1/2}(x)} - \frac{I_{1/2}(x)}{I_{-1/2}(x)}.
\end{equation}
Let us further denote $\zeta_\ta:= \sqrt{2\log(1/\ta)}$ and $\cJ_n:=[1/n, d_2\ta_n(s_n)]$. 
 
Dealing first with the term $\cM_0(X,\ta)$ in \eqref{decompM}, Lemma \ref{lemmzero} below gives
\[  \sup_{\ta\in  \cJ_n} \left|\cM_0(X,\ta)\zeta_\ta/n + C^* \right| = o_{P_{\theta_0}}(1), \]
 by taking $\veps_n=d_2\ta_n(s_n)=o(1)$ therein.

Let us now deal with the term $\cM_1(X,\ta)$ in \eqref{decompM}. Denote
\begin{align}
 \mu_\ta(x) & =\frac{\ta}{\ta+\pi x^2 e^{-x^2/2}/2} \label{mutau}\\ 
  m_n & :=\log\log(n/s_n), \label{defmn}
\end{align}
Lemma \ref{lemmun} below gives  $m_\ta(x)/\mu_\ta(x)=1+o(1)$,  uniformly over $|x|\ge \veps_\ta^{-1}$, for any sequence $\veps_\ta=o(1)$ as $\tau$ goes to $0$. Since we can restrict to $\ta$'s in $\cI_n$, the bound  $m_\ta(x)/\mu_\ta(x)=1+o(1)$ can be used for such $\ta$'s and any $|x|\le m_n$ for $m_n\to \infty$ as in \eqref{defmn}.
 
Let us introduce the set of indices, for $m_n$ as in \eqref{defmn},
\begin{align}
S_{small}&:=\{i\in S_0:\ |X_i|\le m_n \}
\end{align}
and define the events, for $\delta>0$, 
\begin{align} 
\cB_1 & = \left\{ |\{i\in S_0:\ |X_i|>\zeta_{s_n/n}+\delta/2 \}| >s_n\Phi(b-\delta) \right\} \label{eventB}\\ 
\cB_{small} & = \left\{ |S_{small}| \le s_n\sqrt{s_n/n} \right\}. \label{eventBsmall}
\end{align}
By Lemma \ref{lem:b} below, the complements $\cB_1^c, \cB_{small}^c$ of $\cB_1, \cB_{small}$ have vanishing probability. On the event $\cB_1$, denote by $j_1,\ldots,j_N$ the indices satisfying the inequality in the definition of $\cB_1$. By definition $N\ge s_n\Phi(b-\delta)$ and for any such index $j_q$, since $x^2e^{-x^2/2}$ is decreasing for large enough $x$, we have that for large enough $n/s_n$,
\[ X_{j_q}^2e^{-X_{j_q}^2/2} \le 
(\zeta_{s_n/n}+\delta/2)^2 e^{-(\zeta_{s_n/n}+\delta/2)^2/2}
\le 4\frac{s_n}{n} \zeta_{s_n/n}^2e^{-\delta \zeta_{s_n/n}/2},
 \]
 using $\delta\le \zeta_{s_n/n}$ and $e^{-\delta^2/8}\le 1$. The last display is thus a $o(\zeta_{s_n/n}s_n/n)$, uniformly in $\ta$. Since it is  positive, one may denote it $|o(\zeta_{s_n/n}s_n/n)|$. Coming back to \eqref{decompM}, on the event $\cB_1$,
 \[ \cM_1(X,\ta) \ge \frac{1}{\ta}\sum_{i\in S_{small}} m_\ta(X_i) + \frac{s_n\Phi(b-\delta)}{\ta + |o(\zeta_{s_n/n}s_n/n)|}. \]
Indeed, we use that for indices $j$ not in $S_{small}$, $m_\ta(X_j)$ is always positive (because $\mu_\ta(X_j)$ is, for large $n/s_n$), and for indices satisfying the inequality in the definition of $\cB_1$, which are at least $s_n\Phi(b-\delta)$, we use the previous inequalities. Using the global bound $m_\ta(X_i)\ge -1$ (Lemma \ref{lemgenmtau}) for $i\in S_{small}$, one deduces that, on the event $\cB_1\cap \cB_{small}$,
\[ \cM_1(X,\ta) \ge - \frac{s_n}{\ta}\sqrt{s_n/n} + \frac{s_n\Phi(b-\delta)}{\ta + |o(\zeta_{s_n/n}s_n/n)|}\ge \frac{s_n\Phi(b-\delta)+o(1)}{\ta + |o(\zeta_{s_n/n}s_n/n)|}. \] 
One deduces that $\partial \mathbb{M}_\ta(X)/\partial\ta>0$ if the last display minus $C^*n/\zeta_\ta$ is positive, that is if
\[ \ta/\zeta_\ta \le \frac{s_n}{n}\frac{\Phi(b-\delta)+o(1)}{C^*} - \frac{s_n}{n}o(\frac{\zeta_{s_n/n}}{\zeta_\ta}). \] 
The second part of Lemma \ref{lemzeta} gives that $\zeta_{s_n/n}/\zeta_\ta=O(1)$ if $\ta/\zeta_\ta\leqa s_n/n$  so that for the last display to hold it is enough that
\[ \ta/\zeta_\ta \le \frac{s_n}{n}\frac{\Phi(b-\delta)}{C^*}(1+o(1)).\]
A sufficient condition for this inequality to hold is, using  $\zeta_{s_n/n}/\zeta_\ta=O(1)$ if $\ta/ \zeta_\ta\leqa s_n/n$, 
\begin{equation} \label{tau_one}
  \ta \le \zeta_{s_n/n}\frac{s_n}{n}\frac{\Phi(b-\delta)}{C^*}(1+o(1))=:\ta_1.
\end{equation}  
This shows that, on the event $\cB_1\cap \cB_{small}$, the map $\ta\to M_\ta(X)$ has a positive derivative on the interval $[1/n,\ta_1]$, which implies $\hat\ta\ge \ta_1$ on $\cB_1\cap \cB_{small}$ and concludes the proof.
\end{proof}
\begin{lemma} \label{lem:b}
Let $S_0=\{i:\ \te_{0,i}\neq 0\}$. For any real $b$ and $\delta>0$, let
\[ \cB_1 = \left\{ |\{i\in S_0 :\ |X_i|>\zeta_{s_n/n}+\delta/2 \}| >s_n\Phi(b-\delta) \right\}. \]
Then for any such $b$ and $\delta$,
\[ \sup_{\te_0\in \cL_0[s_n;b]} P_{\te_0}[\cB_1^c] = o(1). \]
Further denote the (random) set of integers $S_{small}=\{i:\ |X_i|\le \log\log(n/s_n)\}$. Then if
\[ \cB_{small}=\left\{ |S_{small}|\le s_n\sqrt{s_n/n} \right\},\]
it holds $P_{\te_0}[\cB_{small}^c] = o(1)$ uniformly over $\te_0\in \cL_0[s_n;b]$.
\end{lemma}
\begin{proof}
First note that without loss of generality, up to changing the signs of some of the $\veps_i$'s, which does not change the joint distribution of the observations $(X_i)$, one may assume that $\te_{0,i}\ge 0$ for all $i=1,\ldots,n$. The following inclusion holds
\[ \{i\in S_0 :\ \veps_i>\zeta_{s_n/n}-\te_{0,i}+\delta \} 
\subset \{i\in S_0 :\ |X_i|>\zeta_{s_n/n}+\delta \}.
\]
By definition of the class $\cL_0[s_n;b]$ and since we have assumed $\te_{0,i}\ge 0$, we also have that  $\{i\in S_0 :\ \veps_i>\zeta_{s_n/n}-\te_{0,i}+\delta/2 \}$ contains $ \{i\in S_{\te_0} :\ \veps_i>-b+\delta/2 \}$. The cardinality of the latter set is binomially distributed with parameters $s_n$ and 
$P[\veps_i>-b+\delta/2]$, so that
\begin{align*}
 P_{\te_0}\left[ \cB_1^c \right] 
& \le P_{\te_0}\left[ \text{Bin}(s_n, \bar\Phi(-b+\delta/2)) \le s_n\Phi(b-\delta) \right] \\
&  \le P_{\te_0}\left[ \text{Bin}(s_n, \Phi(b-\delta/2))
- s_n \Phi(b-\delta/2) \le s_n\{\Phi(b-\delta)-\Phi(b-\delta/2)\} \right].
\end{align*}
By the mean value theorem, there exists $c_{b,\delta}\in(b-\delta,b-\delta/2)$ such that $\Phi(b-\delta)-\Phi(b-\delta/2)=\Phi'(c_{b,\delta})(-\delta/2)=-C_2\delta$, for a constant $C_2=C_2(b,\delta)=\phi(c_{b,\delta})>0$. An application of Bernstein's inequality (Lemma \ref{bernstein} applied to $-Z_i$ and with $M=1$, $N=s_n$ and $V=s_n\Phi(b-\delta/2)$) gives
\[  P_{\te_0}\left[ \cB_1^c \right]  \le \exp\left\{- \frac12 \frac{(s_nC_2\delta)^2}{s_n\Phi(b-\delta/2) 
+s_nC_2\delta/3}\right\}=e^{-C_3s_n}=o(1),\]
for some constant $C_3>0$ depending only on $\delta,b$, which gives the result for $\cB_1$.

For $\cB_{small}$, one proceeds similarly and bounds $P_{\te_0}\left[\cB_{small}^c]=P_{\te_0}[|S_{small}|>s_n\sqrt{s_n/n}\right]$. First one notes that $S_{small}$ is, with  $m_n=\log\log(n/s_n)$,  included in
\[ \{i\in S_0,\ |\veps_i|\ge |\te_{0,i}|-m_n\}\subset \{i\in S_0,\ |\veps_i|\ge \sqrt{2\log{n/s_n}}+b-m_n\}.\]
The cardinality of the latter set is  distributed as  $\text{Bin}(s_n,2\bar\Phi(\sqrt{2\log{n/s_n}}+b-m_n))$. Since $\sqrt{2\log{n/s_n}}+b-m_n\ge \sqrt{\log{n/s_n}}$ for large enough $n/s_n$, the result follows by using Bernstein's inequality in a similar way as for $\cB_1$.

\end{proof}

\subsubsection{Properties of the MMLE used for the $S$--value} \label{sec:mmle2}

Recall the definition of 
the marginal maximum likelihood estimator $\hat\ta$ in \eqref{defmmle}, that $\ta_n(s_n)=(s_n/n)\sqrt{\log(n/s_n)}$ 
and the definition
\begin{equation*} 
 \cL_{0,=}[s_n;b]= \bigg\{ \theta\in \ell_0[s_n] \::\: |\theta_i| = \sqrt{2\log \frac{n}{s_n}}  + b\ \text{ for all } i\in S_\te,\ |S_\te|=s_n \bigg\}. 
\end{equation*}
\begin{proposition} \label{propeq}
Let $\hat\ta$ be the marginal maximum likelihood given by \eqref{defmmle}. Let $b$ be a fixed real number and let $\delta>0$ arbitrary. Then,  for $\cL_{0,=}[s_n;b]$ as above and $C^*=(2/\pi)^{3/2}$,
\[ 
\sup_{\te_0 \in  \cL_{0,=}[s_n;b]} P_{\te_0}\Big( \hat\ta(X)
\notin \left[ \ta_n(s_n)\frac{\Phi(b-\delta)}{C^*}(1+o(1)) \,,\,
\ta_n(s_n)\frac{\Phi(b+\delta)}{C^*}(1+o(1)) \right] \Big) =o(1).
\]
\end{proposition}
\begin{proof}
The fact that with high probability $\hat{\ta}$ is larger than the left-endpoint of the interval in the statement has already been derived within the proof of Proposition \ref{propmle} (see \eqref{tau_one}, where the precise constant $\Phi(b-\delta)/C^*$ in front of  $\zeta_{s_n/n} s_n/n$ is obtained and where it is actually proved for the larger class $\cL_0[s_n;b]$). We now prove the upper-bound part.
Recall that $\hat\ta$ is a zero of 
\begin{equation}\label{MpropS}
 \frac{\partial}{\partial\ta}  \mathbb{M}_\ta(X) = \frac{1}{\ta}\sum_{i\in S_0^c} m_\ta(X_i) + \frac{1}{\ta}\sum_{i\in S_0} m_\ta(X_i)
 =: \cM_0(X,\ta)+\cM_1(X,\ta). 
\end{equation} 
We already know from the proof of Proposition \ref{propmle}, for $\ta_1$ defined in \eqref{tau_one}, that  $(\partial\mathbb{M}_\ta/\partial\ta)(X)>0$ for $\ta\le\ta_1$. So it is enough to focus on $\ta$'s that verify  $\ta \ge \ta_1$. Similarly, since we already know that $\hat\ta<\ta_2:=d_2\ta_n(s_n)$, it is enough to focus on $\ta\le \ta_2$.  Also, we have already proved that the part with no signal $\cM_0$ verifies, for $\cJ_n:=[1/n, d_2\ta_n(s_n)]$, 
\begin{equation} \label{em0}
  \sup_{\ta\in  \cJ_n} \left|\cM_0(X,\ta)\zeta_\ta/n + C^* \right| = o_{P_{\te_0}}(1). 
\end{equation}  
Let us now deal with the term $\cM_1(X,\ta)$. Below we split the indices in the sum defining $\cM_1$ in three subsets and control their respective contributions to the sum.  
Recall the notation
\[ \mu_\ta(x)=\frac{\ta}{\ta+\pi x^2 e^{-x^2/2}/2}, \]
and that Lemma \ref{lemmun} gives, uniformly over $|x|\ge \veps_\ta^{-1}$, for $\veps_\ta=o(1)$,
that $m_\ta(x)/\mu_\ta(x)=1+o(1)$. Recall the notation from the proof of Proposition \ref{propmle}
\[ \cS_{small} = \{i\in S_0,\ |X_i|\le m_n\}. \]
For indices $i$ in $\cS_{small}$, we use below that $m_\ta$ is bounded by a constant (Lemma \ref{lemgenmtau}). 
To deal with other indices, let us define the event, for $\delta>0$, 
\begin{equation} \label{eventB2}
\cB_2 = \left\{ |\{i\in S_0:\ |X_i|\le \zeta_{s_n/n}-\delta/2 \}| > s_n\bar\Phi(b+\delta) \right\}.
\end{equation}
By Lemma \ref{lem:b2} below, the complement $\cB_2^c$ of $\cB_2$ has vanishing probability. On the event $\cB_2$, let $S_1$ denote the set of indices $i$'s appearing in the definition of $\cB_2$ in the last display and $S_2:=S_0\setminus S_1$. By definition, $|S_2|\le s_n(1-\bar\Phi(b+\delta)) = s_n\Phi(b+\delta)$ on $\cB_2$.

For any  $j\in S_1$, since $x^2e^{-x^2/2}$ is decreasing for large enough $x$, 
\begin{equation} \label{techm1}  
X_{j}^2e^{-X_{j}^2/2} \ge 
(\zeta_{s_n/n}-\delta/2)^2 e^{-(\zeta_{s_n/n}-\delta/2)^2/2}
\ge \frac{s_n}{n} \zeta_{s_n/n}^2e^{\delta \zeta_{s_n/n}/2}/2,
\end{equation}
  for large enough $n/s_n$. 
For any  $j\in S_2$, we use the fact that $\mu_\ta(X_j)\le 1$ to deduce that $m_\ta(X_j)\le (1+o(1))$ uniformly over $\ta\leqa  (s_n/n)\zeta_{s_n/n}$ and $j\in S_2$. 

  Coming back to \eqref{MpropS}, on the event $\cB_2\cap \cB_{small}$ for 
  $\cB_{small}$ as in \eqref{eventBsmall}, using the bounds established in the previous paragraphs for $i$ in $S_{small}, S_1, S_2$, one obtains
   \begin{align*} 
 \cM_1(X,\ta) &= \frac{1}{\ta}\Big(\sum_{i\in S_{small}} + \sum_{i\in S_1\setminus S_{small}}+\ 
 \sum_{i\in S_2}\ \Big) m_\ta(X_i) \\
& \le C\frac{|S_{small}|}{\ta} + |S_1\setminus S_{small}|\frac{1+o(1)}{\ta+ \frac{s_n}{n} \zeta_{s_n/n}^2e^{\delta \zeta_{s_n/n}/2}\pi/4}+|S_2|\frac{(1+o(1))}{\ta} \\
& \le \frac{4n(1+o(1))}{\pi e^{\delta \zeta_{s_n/n}/2}} +\frac{s_n\Phi(b+\delta)(1+o(1))}{\ta}, 
\end{align*}  
using that $|S_{small}|=o(s_n)$ on $\cB_{small}$ as well as $|S_1\setminus S_{small}|\le s_n$ and $\zeta_{s_n/n}^2\ge 1$ and bounding from below the denominator in the middle term in the sum in the last display  by removing $\ta$, which makes the denominator smaller.

From this one sees that $\partial \mathbb{M}_\ta(X)/\partial\ta<0$ if the last display minus $C^*n/\zeta_\ta$ is negative. Note that since as noted before we only need to focus on $\ta$'s for which $\ta\geqa s_n\zeta_{s_n/n}/n$, we have $\zeta_\ta=o( e^{\delta\zeta_{s_n/n}/2} )$, so that the first term on the right hand side of the last display is negligible compared to $C^*n/\zeta_\ta$.
This shows that  $\partial \mathbb{M}_\ta(X)/\partial\ta<0$ if
\[ \frac{s_n\Phi(b+\delta)(1+o(1))}{\ta} - \frac{C^* n}{\zeta_\ta}(1+o(1))<0, \]
where the $o(1)$ in factor of $(C^*n/\zeta_\ta)$ in the last display is chosen in such a way that it cancels both the negligible term in the bound above on $\cM_1$ and the $o_{P_{\te_0}}$--term arising from  \eqref{em0}.  
That is,  $\partial \mathbb{M}_\ta(X)/\partial\ta<0$ if
\[ \frac{\ta}{\zeta_\ta} > \frac{s_n}{C^* n}\Phi(b+\delta)(1+o(1)). \] 
Using Lemma \ref{lemzeta} below, and denoting $y:=s_n\Phi(b+\delta)(1+o(1))/(C^* n)$, one deduces that  the map $\ta\to \mathbb{M}_\ta(X)$ has a negative derivative 
if $\ta>y\zeta_y$.  To conclude, it suffices to check that  $\zeta_{y}=\zeta_{s_n/n}(1+o(1))$. This directly follows from the explicit expressions from which one derives $\zeta_{y}-\zeta_{s_n/n}=o(\zeta_{s_n/n})$, which concludes the proof.
\end{proof}

\begin{lemma} \label{lem:b2}
Let $S_0=\{i:\ \te_{0,i}\neq 0\}$. For any real $b$ and $\delta>0$, let
\[ \cB_2 = \left\{ |\{i\in S_0:\ |X_i|\le \zeta_{s_n/n}-\delta/2 \}| >s_n\bar\Phi(b+\delta) \right\}. \]
Then for any such $b$ and $\delta$,
\[ \sup_{\te_0\in \cL_{0,=}[s_n;b]} P_{\te_0}[\cB_2^c] = o(1). \]
\end{lemma}
\begin{proof}
As in the proof of Lemma \ref{lem:b} we assume without loss of generality that $\te_{0,i}\ge 0$ for  $i=1,\ldots,n$, so that by definition of  $\cL_{0,=}[s_n;b]$ the nonzero values equal $\zeta_{s_n/n}+b$, and
\begin{align*}
\{i\in S_0& :\ |X_i|\le \zeta_{s_n/n}-\delta/2 \} \\
& =  \{ i\in S_0 :\ -\zeta_{s_n/n}-\te_{0,i} + \delta/2< \veps_i < \zeta_{s_n/n}-\te_{0,i} -\delta/2 \} 
\\
& = \{ i\in S_0 :\ -2\zeta_{s_n/n}-b + \delta/2 < \veps_i< -b-\delta/2 \} 
\end{align*}
The cardinality of the latter set is binomially distributed with parameters $s_n$ and 
$\kappa_n:=\Phi(-b-\delta/2)-\Phi(-2\zeta_{s_n/n}-b+\delta/2)=\bar\Phi(b+\delta/2)+o(1)$, so that
\begin{align*}
 P_{\te_0}\left[ \cB_2^c \right] 
& = P_{\te_0}\left[ \text{Bin}(s_n, \kappa_n) \le s_n\bar\Phi(b+\delta) \right] \\
&  \le P_{\te_0}\left[ \text{Bin}(s_n, \kappa_n) - s_n \kappa_n
 \le s_n\{\bar\Phi(b+\delta)-\bar\Phi(b+\delta/2)-o(1)\} \right].
\end{align*}
By the mean value theorem, there exists $d_{b,\delta}\in(b+\delta/2,b+\delta)$ such that $\bar\Phi(b+\delta)-\bar\Phi(b+\delta/2)=\bar\Phi'(d_{b,\delta})(\delta/2)=-C_3\delta$, for a constant $C_3=C_3(b,\delta)=\phi(d_{b,\delta})/2>0$. We write the binomial distribution as $\sum_{j=1}^{s_n} \zeta_j$ with $\zeta_j$ iid Be$(\kappa_n)$ and set $Z_i=-(\zeta_j-\kappa_n)$. An application of Bernstein's inequality (Lemma \ref{bernstein} applied to $Z_i$ and with $M=1$, $N=s_n$ and $V\le s_n\kappa_n\le s_n$) gives
\[  P_{\te_0}\left[ \cB_2^c \right]  \le \exp\left\{- \frac12 \frac{(s_nC_4\delta)^2}{s_n
+s_nC_4\delta/3}\right\}=e^{-C_5s_n}=o(1),\]
for some constant $C_5>0$ depending only on $\delta,b$, which concludes the proof.
\end{proof}

\begin{lemma} \label{lemzeta}
The map $F:\ta\to \ta/\sqrt{2\log(1/\ta)}=\ta/\zeta_\ta$ is strictly increasing and invertible from $(0,1/4]$ into $(0,(\sqrt{32\log{4}})^{-1}]=:K$ and its inverse $F^{-1}$ verifies 
\[
F^{-1}(t) \le t\zeta_t,\qquad \text{for any } t\in K.\]
Also, if $\ta/\zeta_\ta\le Cs_n/n$ for $C>0$, then $\zeta_{s_n/n}\le D \zeta_\ta$ for some $D>0$, for large enough $n/s_n$. 
\end{lemma}
\begin{proof}
First, one checks that $F'>0$ on the considered interval so that $F$ has an inverse $F^{-1}$. Next, using $F^{-1}\circ F(t)=t$ one gets $F^{-1}(t)=t\zeta_{F^{-1}(t)}\ge t$ using $\zeta_u\ge 1$ for $u$ in the considered range of $F^{-1}$. Since $u\to \zeta_u$ is decreasing, using the previous inequality again one gets $F^{-1}(t)\le t\zeta_t$. 
 
To prove the second part of the statement, one applies $F^{-1}$ to both parts of the inequality $\ta/\zeta_\ta\le Cs_n/n$ to get $\ta \le (Cs_n/n)\zeta_{Cs_n/n}$. Since $u\to  \zeta_u$ is decreasing, one gets $\zeta_\ta \ge \zeta_{(Cs_n/n)\zeta_{Cs_n/n}}$. One further writes $\zeta_{(Cs_n/n)\zeta_{Cs_n/n}}^2=\zeta_{s_n/n}^2-2\log(C\zeta_{s_n/n})\ge \zeta_{s_n/n}^2/2$ for large enough $n/s_n$, which shows that $\zeta_\ta\ge \zeta_{s_n/n}/\sqrt{2}$, which concludes the proof.
\end{proof}

\subsection{Properties of the marginal likelihood function} \label{sec:margintec}

This section gathers a number of results (and well as some useful notation) that were established in \cite{vsv17} on the marginal likelihood from which $\hat\ta$ is defined. 

{\em The marginal density of $X$.} By definition, the marginal density of $X_1=\theta_1+\veps_1$ in the Bayesian model at point $X_1=x$  equals the convolution $(\phi*h_\ta)(x)$, for $h_\ta$ the horseshoe density of parameter $\ta$. 
This quantity plays a central role for $s$--values, since it is the denominator in the $\Pi[ \te <0 \given X=x]=N(x)/D(x)$. In Lemma \ref{lem:Sval}, we have seen that 
\[
D(x) = \int_0^{\infty}  \phi_{0,1 + (\la\tau)^2}(x) \frac{2}{\pi}\dfrac{d\la}{1 + \la^2}.
\]
An equivalent way of writing $D(x)$ is via the following functions $I_k$ introduced in \cite{vsv17}: for $k\in\{-1/2,1/2,3/2\}$, 
\begin{equation} \label{defik}
 I_k(x):=\int_0^1 z^k \frac{1}{\ta^2+(1-\ta^2)z} e^{y^2 z/2}. 
\end{equation} 
A direct computation shows (Equation (C.2) in Appendix C of \cite{vsv17}) that 
\begin{equation} \label{linkde}
 D(x) =  \frac{\ta}{\pi} I_{-1/2}(x)\phi(x).
\end{equation}

{\em The marginal likelihood and its properties.} The log-marginal likelihood is 
\begin{equation*} 
 \mathbb{M}_\ta(X):=\sum_{i=1}^n \log \int \phi(X_i-\te) h_\ta(\te)d\te.
\end{equation*}

Its derivative with respect to $\ta$ can be written, according to Lemma C.1 in \cite{vsv17}, in the form
\begin{align*}
\frac{\partial}{\partial\ta}  \mathbb{M}_\ta(X)  = \frac{1}{\ta}\sum_{i=1}^n m_\ta(X_i)  & = \frac{1}{\ta}\sum_{i\in S_0^c} m_\ta(X_i) + \frac{1}{\ta}\sum_{i\in S_0} m_\ta(X_i)\\
& =: \cM_0(X,\ta)+\cM_1(X,\ta), \label{decompM}
\end{align*}
recalling that $S_0=S_{\te_0}$ denotes the support of $\te_0$, with $S_0^c=\{1,\ldots,n\}\setminus S_0$ and where 
\begin{equation*} 
 m_\ta(x)=x^2\frac{I_{1/2}(x)-I_{3/2}(x)}{I_{-1/2}(x)} - \frac{I_{1/2}(x)}{I_{-1/2}(x)}.
\end{equation*}

The following general properties of $m_\ta$ are established in \cite{vsv17},
\begin{lemma}[Lemma C.7 (i) of \cite{vsv17}] \label{lemgenmtau}
The function $m_\ta$ as above satisfies the global bound, for some universal constant $C$, all $\ta\in[0,1]$ and $x\in\RR$,
\[ -1 \le m_\ta(x) \le C. \]
\end{lemma}

{\em Study of the term $\cM_0$.} The following Lemma is a direct consequence of bounds obtained in \cite{vsv17}. Let us recall the notation $\zeta_\ta:= \sqrt{2\log(1/\ta)}$. 

\begin{lemma} \label{lemmzero}
For any $\veps_n\to 0$, for $S_0^c=\{i: \te_{0,i}=0\}$, $\cM_0$ as in \eqref{decompM}, $C^*=(2/\pi)^{3/2}$,
\[ \sup_{1/n\le \ta\le \veps_n} 
\frac{\zeta_\ta}{n} \cM_0(X,\ta) = - C^* \left(1+o_{P_{\te_0}}(1)\right). \]
\end{lemma}
\begin{proof}
 Lemma C.3 in  \cite{vsv17} grants that $\cM_0(X,\ta)$  in \eqref{decompM} is close to its expectation up to a term $o_P(n/\zeta_\ta)$, that is
\[ \sup_{1/n\le \ta\le \veps_n}  \frac{\zeta_\ta}{|S_0^c|}\left|
  \cM_0(X,\ta) -   
E_{\te_0} \cM_0(X,\ta)
 \right| =o_{P_{\te_0}}(1).\]
Proposition C.2 in \cite{vsv17} implies a control of the expectation itself 
\[ \sup_{1/n\le \ta\le \veps_n} 
\frac{\zeta_\ta}{|S_0^c|} E_{\te_0} \cM_0(X,\ta) = -(2/\pi)^{3/2} (1+o(1)). \]
Combining both displays with $|S_0^c|=n(1+o(1))$ gives the result.
\end{proof}
 
{\em Study of the term $\cM_1$.} The following result  is useful when studying $\cM_1$. Let us denote, for any positive $x,\ta$,
\[ \mu_\ta(x)=\frac{\ta}{\ta+\pi x^2 e^{-x^2/2}/2}. \]
\begin{lemma}[Lemma C.7 (vi) in \cite{vsv17}] \label{lemmun}
For $m_\ta$ as in \eqref{mtau} and $\mu_\ta$ as above, for any $\veps_\ta=o(1)$, as $\ta\to 0$,
\[ \sup_{|x|\ge 1/\veps_\ta} \left|\frac{m_\ta(x)}{\mu_\ta(x)} - 1 \right| =o(1).\]
\end{lemma}

\begin{lemma}[Lemma C.9 in \cite{vsv17}] \label{lemprec}
Let $I_{-1/2}$ be defined by \eqref{defik} with $k=-1/2$. Given $\veps_\ta=o(1)$ as $\ta\to 0$, there exist functions $r_\ta$ with $\sup_{x> 1/\veps_\ta} |r_\ta(x)| = O(\sqrt{\ta}+\veps_\ta^2)$ as $\ta\to 0$ such that
\[ I_{-1/2}(u) = \left(\frac{\pi}{\ta} + \frac{e^{u^2/2}}{u^2/2} \right)(1+r_\ta(u)). \]
\end{lemma}

\subsection{Link between $s$--values and marginal credible intervals procedures}\label{app:linkmci}
Let us recall the expression of the $s$--value
\begin{align*} 
	s(x;\ta) & = 2\left( \Pi(\te<0 \given X=x) \wedge  \Pi(\te>0 \given X=x)\right)
\end{align*}
and suppose we are in the sequence model to fix ideas (the argument goes through under a mild assumption otherwise, see below). Since the posterior distribution has a density on $\RR$, it does not charge $0$, and $\te<0$ in the last display (respectively $\te\ge 0$) can also be written $\te\le 0$ (resp. $\te>0$). 
In what follows for $t\in(0,1)$ we denote by $z^t(x)$  the quantile at level $t$  of the
marginal posterior distribution of $\theta$ given $X = x$, and the quantile function is the usual inverse of the posterior distribution function, since the posterior density is (strictly) positive on $\RR$. One can write, for any $t\in(0,1)$,
\begin{align*}
	\vphi^s_t(x) &= \1\left\{ s(x;\ta) < t \right\}\\
	&=\1\left\{ \Pi(\te \le 0 \given X=x) <t/2 \right\} + 
	\1\left\{ \Pi(\te > 0 \given X=x) <t/2 \right\} 
	\\
	&=\1\left\{ z^{t/2}(x) > 0 \right\} + \1\left\{ z^{1-t/2}(x) < 0 \right\}.
	\end{align*} 
Let us denote by $\cI(x)$ the centered two--sided posterior credible interval at level $1-t$, that is
\[ \cI_{1-t}(x) = [z^{t/2}(x),z^{1-t/2}(x)]. \] 
Noting $\left\{ z^{t/2}(x) > 0 \right\} \cup \left\{ z^{1-t/2}(x) < 0 \right\} = \left\{ 0\notin \cI(x) \right\}$, we have
\[ \vphi^s_t(x) = \1\left\{ 0\notin \cI_{1-t}(x) \right\}.  \]
In \cite{cr20}, Section 5.1, in the sequence model for the  spike--and--slab prior and the $\ell$--value procedure, an analogous relationship was obtained, but with a two--sided posterior credible interval of credibility level $1-2t$ (and {\em not} $1-t$ as is the case here; this important difference comes from the fact mentioned earlier in the paper that for spike--and--slab posteriors, $0$ already catches the whole mass of the spike). 

More generally, it is not hard to check that the previous argument generalises for other models (see Section \ref{subsec:gen}), as long as the marginal posterior distribution of $\te_i$ given the data has a (strictly) positive density.

\section{Simulations}   \label{app:sims}

\subsection{Sequence model} \label{app:seqmod}

In this Section, we give more details about the algorithms behind the simulations in Section \ref{sec:simu} in the case of data coming from the sequence model. First, for the empirical Bayes approach used therein, we give details about how the marginal maximum likelihood estimator $\hat\ta$ is obtained. Then, since the R package {\tt horseshoe} is used to generate samples from the posterior, we explain how in the present multiple testing context one can use an adaptive grid that avoids sampling from a number $n$ of coefficients (given that here we take $n$'s of the order $10^4$, using the original package would be prohibitive). Next, we provide details on implementing  the $S$--value procedure, for which we used an approximation of its explicit expression. Finally, we provide further numerical experiments in this setting: the effect of taking a larger sparsity level, as well as implementing hierarchical Bayes.

\subsubsection{Estimate $\tau$ with MMLE } \label{Appendix:tauMMLE}

The hyperparameter $\tau$ is estimated following the marginal maximum likelihood method (see (\ref{defmmle})), 
\[
\hat\ta = \hat\ta(X) = \underset{\ta\in[1/n,1]}{\text{argmax}}\, 
\mathbb{M}_\ta(X)
\] 
where  	$\mathbb{M}_\ta(X)$ is the log-marginal likelihood defined,  for $h_\ta$ the horseshoe density with parameter $\ta$, as
\begin{equation*} 
	\mathbb{M}_\ta(X)=\sum_{i=1}^n \log \int \phi(X_i-\te) h_\ta(\te)d\te.
\end{equation*}
The work \cite{vsv17} proposed an  estimator of $\tau$ in the Sequence model with a Horseshoe prior based on maximum marginal likelihood estimateur, that is obtained by optimisation of the function $\mathbb{M}_{\tau}(X)$. However the function `HS.MMLE' in the R package {\tt horseshoe} \cite{horseshoe_package} runs for  vectors of length of the order $400$ only if the true signal value is very sparse.

In order to be able to perform inference for high-dimensions (of the order of $n=10^4$), here we propose to use an approximation; note that we do not perform this directly at the level of the criterion function $\mathbb{M}_\ta$, which would result in numerical instability when computing its derivative to find a zero thereof. Rather, we perform the approximation at the derivative level. The derivative of $\mathbb{M}_\ta$ with respect to $\ta$ can be written, according to Lemma C.1 in \cite{vsv17}
\begin{equation*} 
	\frac{\partial}{\partial\ta}  \mathbb{M}_\ta(X)  = \frac{1}{\ta}\sum_{i=1}^n m_\ta(X_i),  
\end{equation*}
where 
\begin{equation*}
	m_\ta(x)=x^2\frac{I_{1/2}(x)-I_{3/2}(x)}{I_{-1/2}(x)} - \frac{I_{1/2}(x)}{I_{-1/2}(x)}.
\end{equation*}
Lemma C.9 and Lemma C.10 in \cite{vsv17} give 
\[
I_{-1/2}(x) =(\frac{\pi}{\tau} +  \sqrt{\frac{x^2}{2}} \int_1^{\min(1,\frac{x^2}{2})} \frac{e^v}{v^{3/2}}dv) (1+O(\sqrt{\tau})),
\]
\[
I_{1/2}(x) = (\frac{1}{\sqrt{\frac{x^2}{2}}} \int_0^{\frac{x^2}{2}}  \frac{e^v}{\sqrt{v}}dv)
(1+O(\sqrt{\tau})),
\]
\[
I_{1/2}(x)-I_{3/2}(x) = (\frac{1}{\sqrt{\frac{x^2}{2}}} \int_0^{\frac{x^2}{2}} \frac{1-\frac{2v}{x^2}}{\sqrt{v}} e^v dv)(1+O(\sqrt{\tau})).
\]
This motivates the definition of the following approximation of $m_{\tau}(x)$.
\begin{equation} \label{mtildetau}
	\tilde{m}_{\tau}(x) = \frac{\sqrt{2} |x| \int_0^{\frac{x^2}{2}} \frac{1-\frac{2v}{x^2}}{\sqrt{v}} e^v dv - \frac{\sqrt{2}}{|x|} \int_0^{\frac{x^2}{2}}  \frac{e^v}{\sqrt{v}}dv}{\frac{\pi}{\tau} +  \frac{|x|}{\sqrt{2}} \int_1^{\min(1,\frac{x^2}{2})} \frac{e^v}{v^{3/2}}dv}
\end{equation}
We propose to estimate $\tau$ by solving with respect to $\tau$ the equation 
\begin{equation} \label{eqtauhat}
	\sum_{i=1}^{n} \tilde{m}_{\tau}(X_i)= 0,
\end{equation}
over $\tau \in [1/n,1]$ (the restriction to $\tau$'s not too close to $0$ is standard in the present sparsity setting; procedures with even smaller $\tau$'s would tend to be overly conservative). The function  $\tau \in [1/n,1] \rightarrow \sum_{i=1}^{n} \tilde{m}_{\tau}(X_i)$ is increasing in $\tau$ and the solution $\hat{\tau}$ of (\ref{eqtauhat}) is obtained by evaluating each integral in (\ref{mtildetau}) numerically.\\

\subsubsection{Comparison between the binary search and the {\tt horseshoe}  package for the $s$--value procedure} \label{Appendix:comparison.VdP.dichotomie}

To apply the $\vphi^{s}_t(X)$ procedure, we need to estimate the $s$--values $s(X_i;\tau)$ from a  sample of the posterior distribution $\Pi( \cdot \given X=X_i)$. This can be done using the R package {\tt horseshoe} \cite{horseshoe_package}. However, when $n$ is very large, the MCMC used to generate such a sample can be computationally expensive and we propose an approach to avoid to generate a posterior sample for each $i \in \{1, \ldots, n\}$.

Numerically, we can evaluate the $s$--values
only at a few points $x$  (but not $n$ of them), using posterior samples.  When $0<x_1<x_2$, we observe numerically that $s(x_1,\ta)>s(x_2,\ta)$: that is, the function $x\to s(x;\ta)$ appears numerically to be monotone on $\RR^+$  (we believe this is indeed the case, although do not prove it here; note that due to the expression of $s(x;t)$ as a ratio of continuous mixtures, the study of its monotonicity is not straightforward). The idea is therefore to `invert' the map $x\to s(x;t)$, viewing the $s$--value procedure $\vphi^{s}_t$
as a thresholding procedure that selects the $|X_i |$'s larger than some threshold 
\[
s(X_i;\ta) < t \ \Leftrightarrow \ |X_i| > \hat{\xi}(\tau,t).
\]

Suppose $x>0$. 
The value $\hat{\xi}(\tau,t)$ at which the $s$--value crosses $t$ can be found by a binary search algorithm as follows. 
We start with two initial boundary values $
0<z_1^{(0)}<z_2^{(0)}$, with a very small $z_1^{(0)}$ and large $z_2^{(0)}$, so that 
\[ s(z_1^{(0)};\ta)-t>0>s(z_2^{(0)};\ta)-t. \]
Next, we set $z_3^{(0)}=(z_1^{(0)}+z_2^{(0)})/2$ to be their midpoint. 
For some small parameter $\delta>0$, we evaluate the $s$--value  $s(\cdot;\ta)$  at the three points $z_-^{(0)}=z_3^{(0)}-\delta$, $z_3^{(0)}$  and $z_+^{(0)}=z_3^{(0)}+\delta$.
If the three values $s(a;\ta)-t$ for $a=z_-^{(0)}, z_3^{(0)}, z_+^{(0)}$ have not all the same sign, we stop the algorithm. This means that the crossing point $\hat{\xi}(\tau,t)$ is located in the interval between $[z_-^{(0)}, z_+^{(0)}]$ and we return $z_3^{(0)}$ as an approximation of $\hat{\xi}(\tau,t)$. 
Otherwise, if the three quantities are all positive, then the crossing point must lie to the right of $z_3^{(0)}$. We therefore replace the left endpoint by the midpoint and keep the right endpoint unchanged.
Conversely, if the three quantities are all negative, then the crossing point lies to the left of $z_3^{(0)}$, and we replace the right endpoint by the midpoint and keep the left endpoint unchanged. 
The procedure is then repeated recursively. At the $k$-th iteration, we compute the midpoint
$z_3^{(k)}=\frac{z_1^{(k)}+z_2^{(k)}}2$, and we 
evaluate $s(\cdot;\tau)-t$ at the three points
$z_3^{(k)}-\delta, z_3^{(k)}, z_3^{(k)}+\delta$
and continue updating the interval endpoints as long as these three values have the same sign. The algorithm terminates at the first iteration for which the signs are not all identical, indicating that the level crossing occurs within a neighborhood around the current midpoint. The corresponding midpoint is taken as the estimate $\hat{\xi}(\tau,t)$. Once the threshold $\hat{\xi}(\tau,t)$ is estimated, we compare each $|X_i|$ with it.
 Consequently, 
 the function from the {\tt horseshoe} package that generates a posterior sample is  called only a few times  instead of $n$ times, which is particularly advantageous for large $n$. 

For the $Cs$--value based procedure, we introduce the following algorithm. 
Let $\hat{\xi}(\tau,t)$ denote the threshold obtained from the binary search algorithm described above, and define
$\mathcal R=\{1\le i\le n:\ |X_i|>\hat{\xi}(\tau,t)\}$ and $c_{\mathcal R}$ its cardinality. 
We begin by evaluating $s(X_i;\tau)$ only for the observations in $\mathcal R$.  If $\mathcal R$ is empty, we initialize it with the observation having the largest absolute value.
The  computed $s$--values are then ordered increasingly, and for each $k \in \{1, \ldots, c_{\mathcal R}\}$ we calculate the cumulative means
$\sum_{i=1}^k s_{(i)} / k$
where $s_{(1)}\le\cdots\le s_{(c_{\mathcal R})}$ denote the ordered $s$--values.
If all these cumulative means are smaller than $t$, then the available observations in $\mathcal R$ are not sufficient to determine $\hat{k}_{C_s}$ defined in \eqref{kchap}. In that case, we enlarge the set $\mathcal R$ by adding the next observation in decreasing order of $|X_i|$, evaluate its $s$--value, and compute again the cumulative means. This procedure is repeated until at least one cumulative mean exceeds $t$.
Once this occurs, we approximate $\hat{k}_{C_s}$ by  the quantity
$$ \max\left\{1 \le k \le c_{\mathcal R}:\ \ \frac{1}{k} \sum_{i=1}^k s_{(i)} \le t \right\},$$
and the $Cs$--value procedure rejects all the $s$--values smaller than $s_{\hat{k}_{C_s}}$.

The FDR and TDR obtained using either the binary search approach or the {\tt horseshoe} package can be seen to be numerically  close, with the binary search approach offering a substantial gain in both computational time and memory requirements for storing posterior samples.
In terms of computational time, for $n=10000$,  the $s$--value procedure is between $15$ and $23$ times faster when using the binary search than the {\tt horseshoe} package, as the search of the value at which the $s$--value crosses $t$ only takes a few iterations. 
The computational time of the $Cs$--value procedure depends on the number of iterations required to expand $\mathcal {R}$ until $\hat{k}_{Cs}$ can be evaluated. It takes on average $1.3$ iterations when $s=10$ and $23.7$ iterations when $s=100$, making the binary search $18$ times faster when $s=10$ and $4$ times faster when $s=100$.\\

\subsubsection{$S$--value procedure} \label{Appendix:SvalueEquivalent}

To apply the $S$--value procedure, one possibility would be to first derive  a sampler from conditional distributions of $\te_1$ given $X_1\ge x$ (as these cannot be directly obtained from the usual conditionals $\cL(\te_1\given X_1=x)$). Since this is expected to be a delicate task for more general models (where we rather advocate the use of the $Cs$--value procedure, which is simpler since solely based on the $s$--values and hence can be done from posterior samples), here our purpose will be mostly illustrative and we propose to use a simple approximation of  the $S$--value \eqref{svalx}, whose expression we recall here
\begin{align*}  
	S(x;\ta) & = 2\left( \Pi_\ta(\te_1<0 \given X_1\ge x) \wedge  \Pi_\ta(\te_1>0 \given X_1\le x)\right).
\end{align*}
Namely,  we use the theoretical equivalents from Section \ref{sec-bSval} to approximate the posterior probability as follows, with $C^*=(2/\pi)^{3/2}$
\[
\begin{aligned}
	& \Pi_{\tau}(\theta < 0 \mid X \ge x)
	\approx \frac{\bar{\Phi}(x)}{2\bar{\Phi}(x) + \dfrac{C^* \tau}{x}},
	\qquad \text{for } x > \mu_n; 
	 \\[0.8em]
	& \Pi_{\tau}(\theta > 0 \mid X \le x)  
	\approx \frac{\Phi(x)}{2\Phi(x) - \dfrac{C^* \tau}{x}}, 
	\qquad \text{for } x < -\mu_n; 
	 \\[0.8em]
	& \Pi_{\tau}(\theta < 0 \mid X \ge x)
	\approx \Pi_{\tau}(\theta > 0 \mid X \le x)
	\approx 1/2, 
	\qquad \text{for } x \in [-\mu_n, \mu_n].
\end{aligned} 
\]

\subsubsection{Further simulations in the sequence model} \label{Appendix:sequencemodel}

Figure \ref{fig::risk} presents the multiple testing risk $\cR=\FDR+\FNR$ for the $s$--value and $S$--value procedures in the sparse sequence model with $n=10000$ and $s_n=10$, using an empirical Bayes approach, together with the minimax risk $\overline{\Phi}(b)$ as a function of  the real parameter $b$.
The $s$--value procedure achieves a risk close to the optimal one for very large signals. The $S$--value procedure suffers a loss in this regime  due to the additional threshold $t$, although it remains close to optimal when $t$ is small ($t=5\%$). However, it provides very good finite-sample performance for intermediate signal strengths. These results indicate that choosing a small target level $t$ yields a good compromise between finite-sample performance and asymptotic near-optimality of the $\mathcal{R}$-risk.

\begin{figure}[!h]
	\includegraphics[scale=0.55]{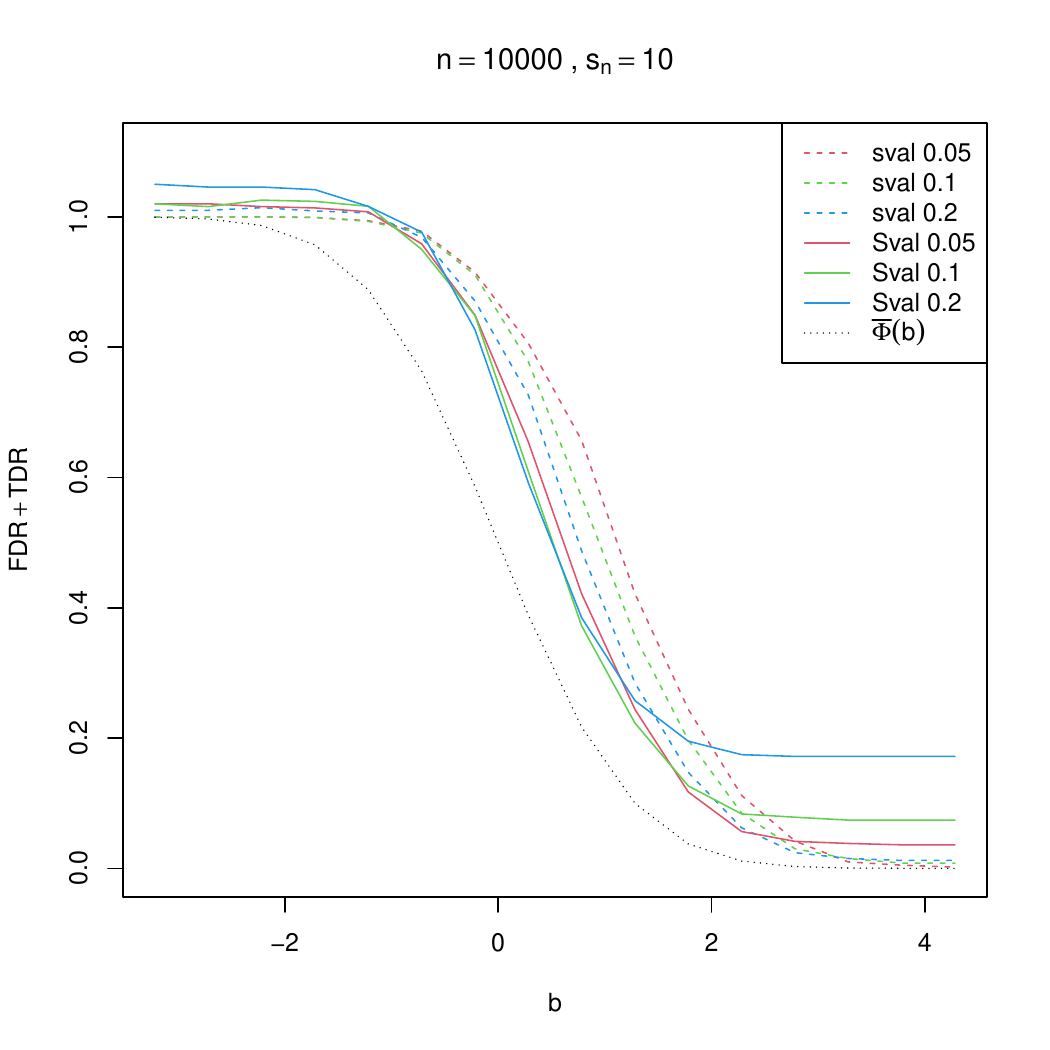}
	\centering
	\caption{ Risk for the $s$--value and $Cs$-- value procedures with threshold $t \in  \{0.05, 0.1, 0.2\}$, as well as the minimax risk $\overline{\Phi}(b)$ as a function of  the real parameter $b$; sequence model with $n = 10000; s=10$,  alternative all equal to $\mu$ and an  Empirical Bayes approach; $100$	replications.} 
	\label{fig::risk}
\end{figure}

Figure 	\ref{fig:all.Horseshoe.s100} presents the FDR and TDR in the sparse sequence model for  $n=10000, s_n=100$, using an  empirical Bayes approach. The results are simular to those obtained in Figure \ref{fig:all.Horseshoe.s10}, although we note that for the weaker sparsity $s_n /n = 0.01$, the FDR values	are slightly higher above the threshold $t$. This suggests that the asymptotical regime has not been reached yet for this setting. This is in line with what has been observed empirically for the C$\ell$--value procedure for spike--and--slab posteriors, see e.g. \cite{cr20} (Supplement, Section S--8.1 and Figure S--3) for more details.
	
\begin{figure}[!h]
	\includegraphics[scale=0.55]{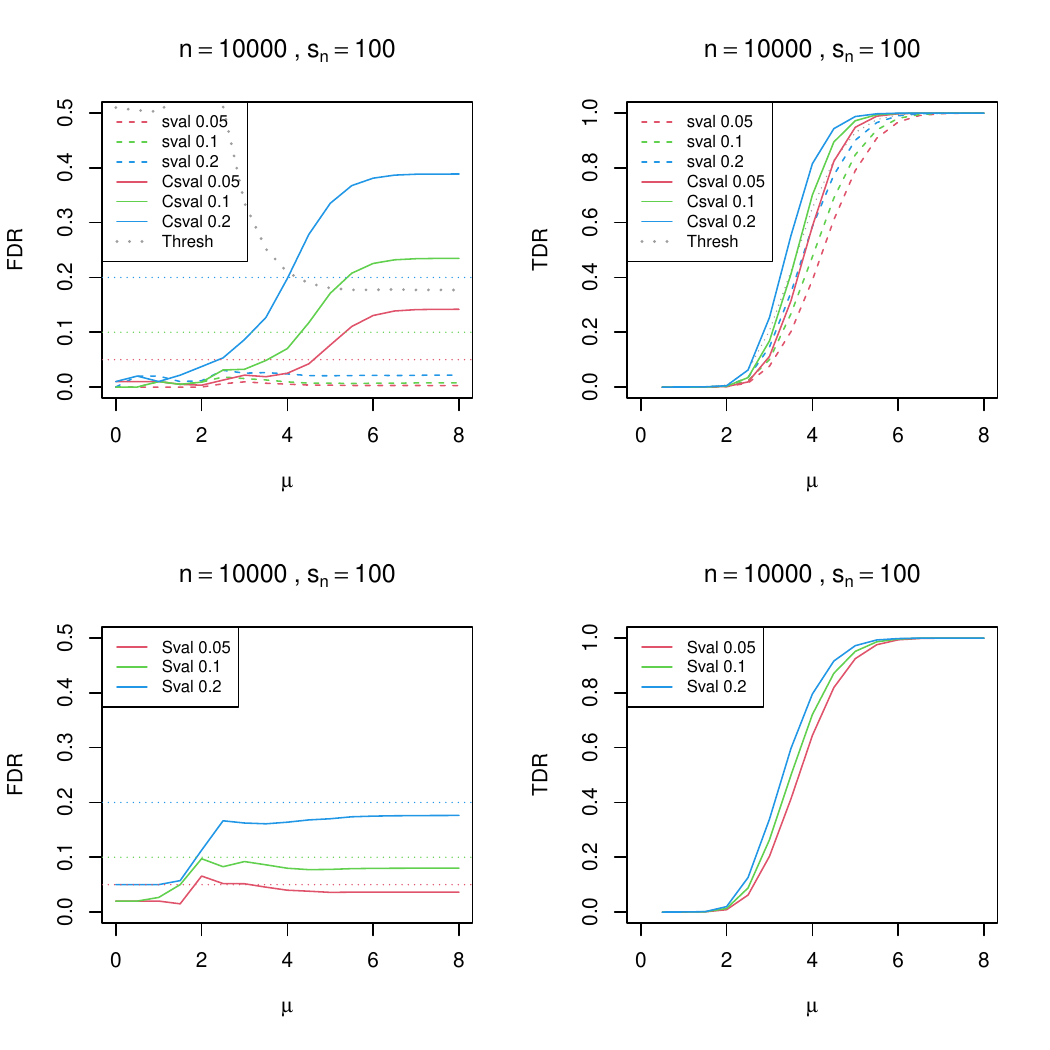}
	\centering
	\caption{Sequence model. Empirical Bayes approach. FDR and TDR of the $s$--value and $Cs$-- value procedures with threshold $t \in  \{0.05, 0.1, 0.2\}$, as well as those of the thresholding method with $\kappa=0.5$ (top graph); FDR and TDR of the $S$--value procedure with threshold $t \in  \{0.05, 0.1, 0.2\}$ (bottom graph); $n = 10000; s_n=10^2$; alternative all equal to $\mu$; $100$	replications.} 
	\label{fig:all.Horseshoe.s100}
\end{figure}

\noindent Figure \ref{SeqModel_Horseshoe_fullBayes} presents results in the sparse sequence model, using a  fully hierarchical Bayesian approach.  We use the R package {\tt Mhorseshoe}, specifying, in the particular case of the sequence model, that the design matrix $X$ is the identity matrix of size $n$.  Simulations are performed for $n=5000$ and $s=50$ with $20$ replications,  as they are more time-consuming that using fonctions specific to the sequence model. The $s$--value and $Cs$--value procedures yield results similar to those from the empirical Bayes approach.

\begin{figure}[!h]
	\includegraphics[scale=0.55]{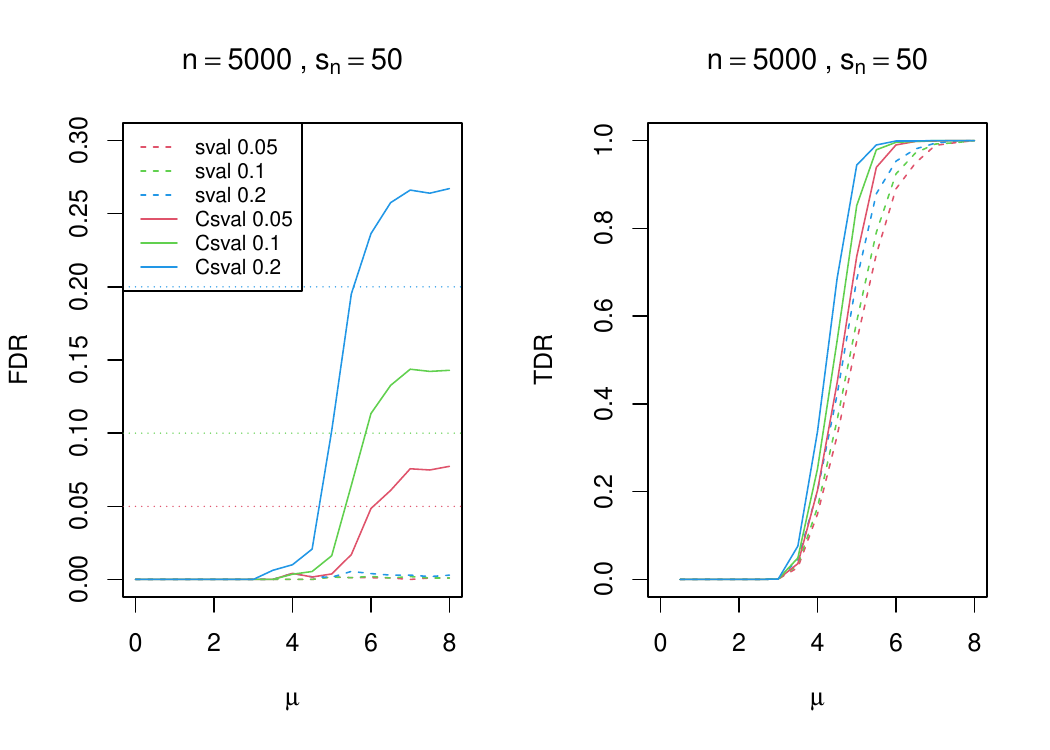}
	\centering
	\caption{Sequence model.  Hierarchical Bayes approach. FDR and TDR of the $s$--value and $Cs$-- value procedures with threshold $t \in  \{0.05, 0.1, 0.2\}$; $n = 5000; s=50$; alternative all equal to $\mu$; $20$	replications (expensive in time).} 
	\label{SeqModel_Horseshoe_fullBayes}
\end{figure}

\subsection{Further simulations in the regression model } \label{Appendix:regressionmodel}

Figures \ref{fig:regression.Horseshoe_n2000_p4000}, \ref{fig:regression.Horseshoe_n3000_p6000}, \ref{fig:regression.Horseshoe_n4000_p8000}, \ref{fig:regression.Horseshoe_n5000_p10000} present additional results for the regression model with the uncorrelated design (i) and various values of $(n,p)$ and $s$. The behavior of the $s$--value and $Cs$--value procedures is similar to that observed in Section \ref{sec:linreg}, as illustrated in Figure \ref{fig:regression.Horseshoe.uncorrelated}.

\begin{figure}[!ht]
	\includegraphics[scale=0.55]{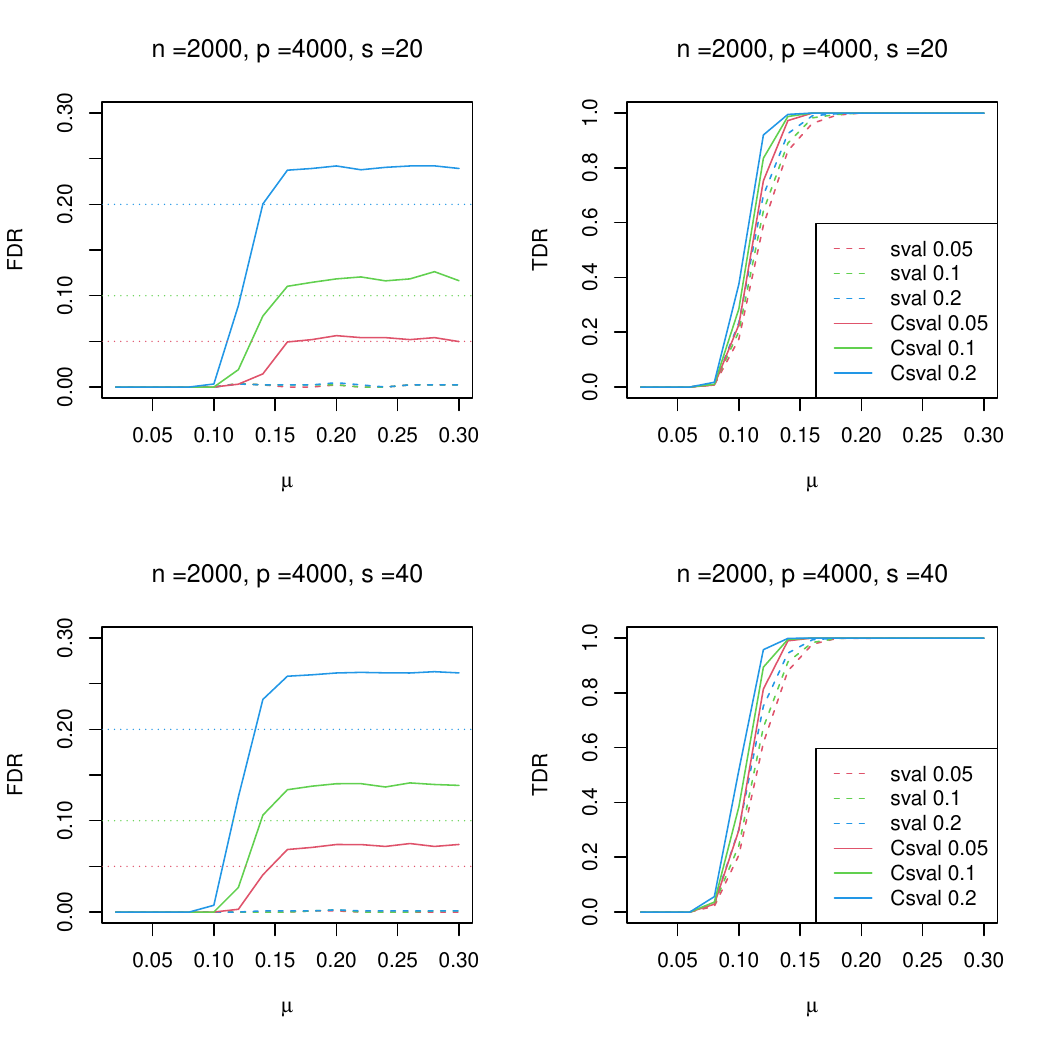}
	\centering
	\caption{Linear Regression with uncorrelated features (i), performance of $s$--value and C$s$--value procedures. }
	\label{fig:regression.Horseshoe_n2000_p4000}
\end{figure}

\begin{figure}[!ht]
	\includegraphics[scale=0.55]{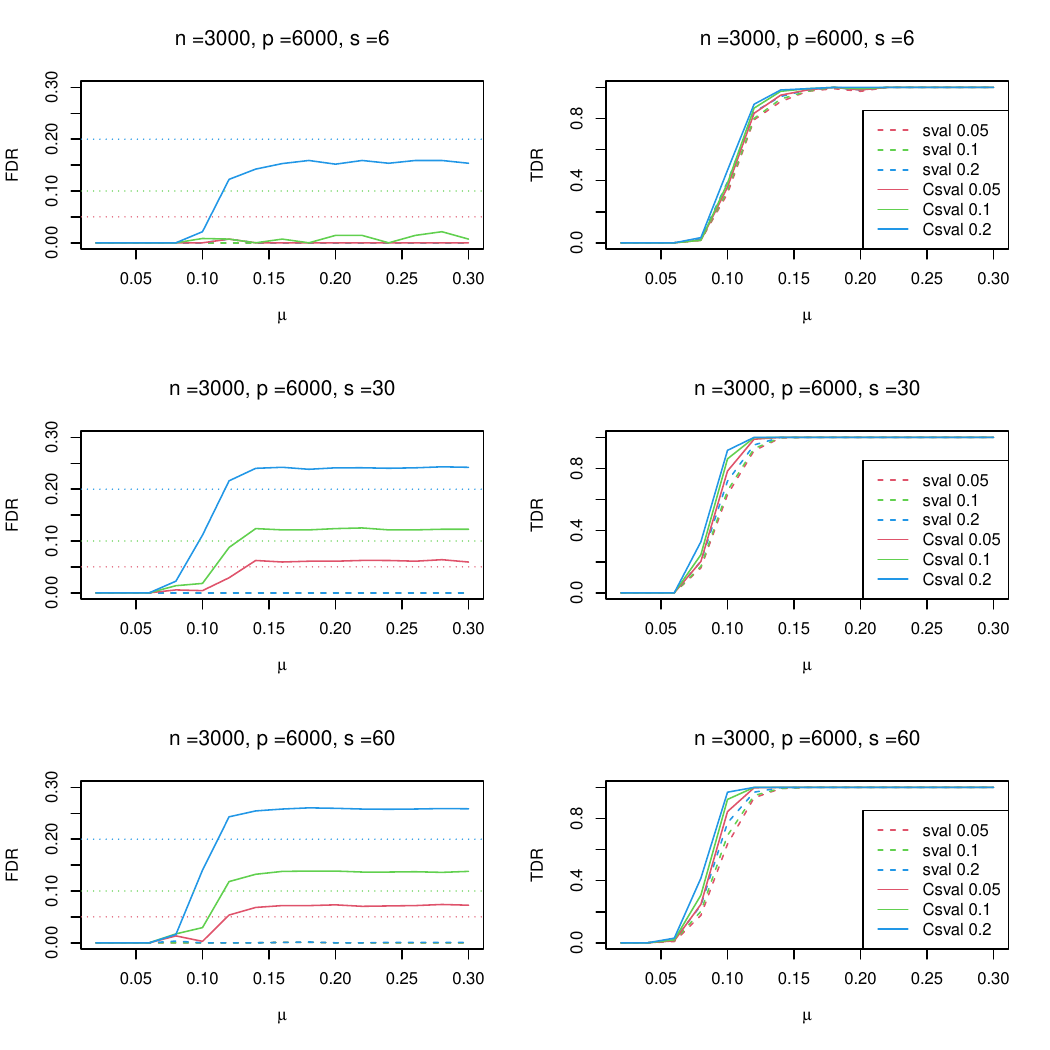}
	\centering
	\caption{Linear Regression with uncorrelated features (i), performance of $s$--value and C$s$--value procedures. }
	\label{fig:regression.Horseshoe_n3000_p6000}
\end{figure}

\begin{figure}[!ht]
	\includegraphics[scale=0.55]{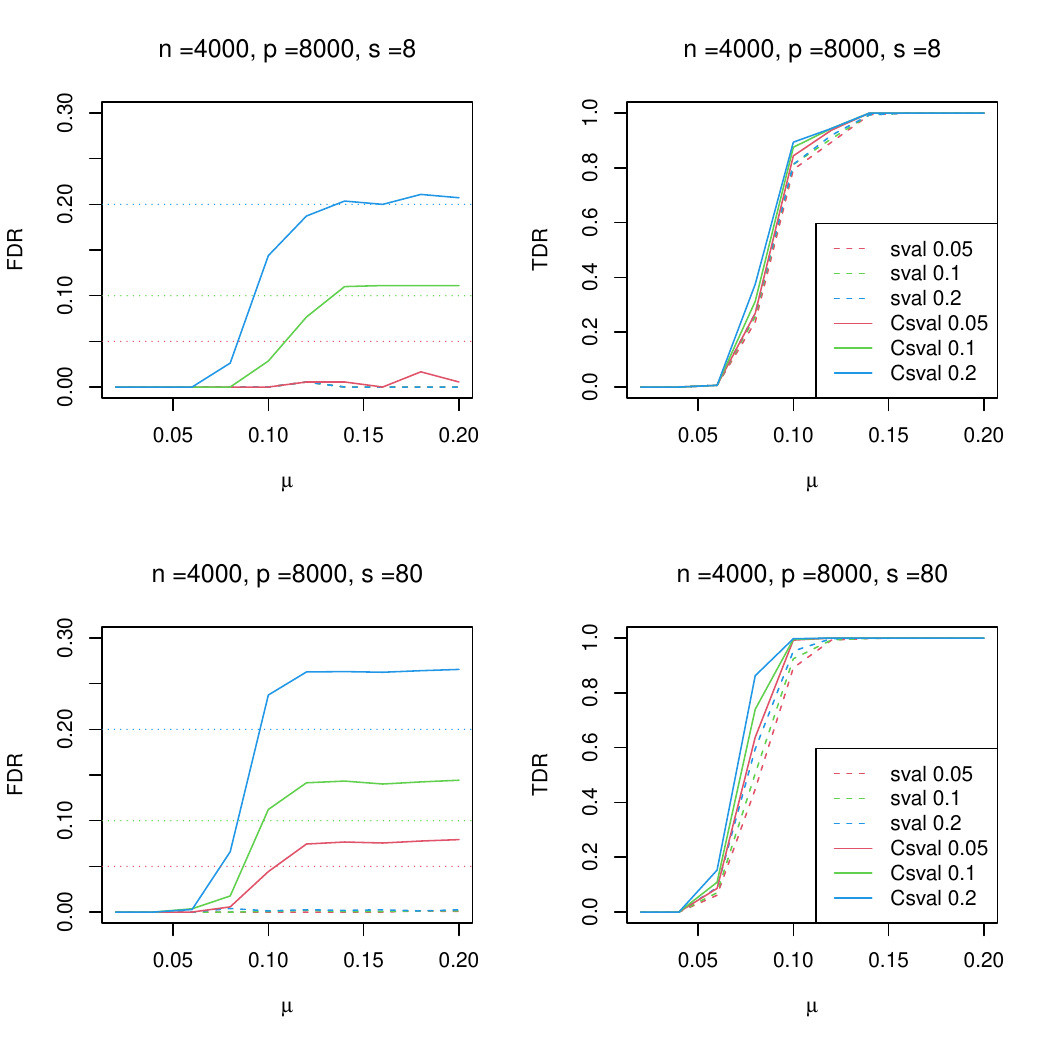}
	\centering
	\caption{Linear Regression with uncorrelated features (i), performance of $s$--value and C$s$--value procedures. }
	\label{fig:regression.Horseshoe_n4000_p8000}
\end{figure}

\begin{figure}[!ht]
	\includegraphics[scale=0.55]{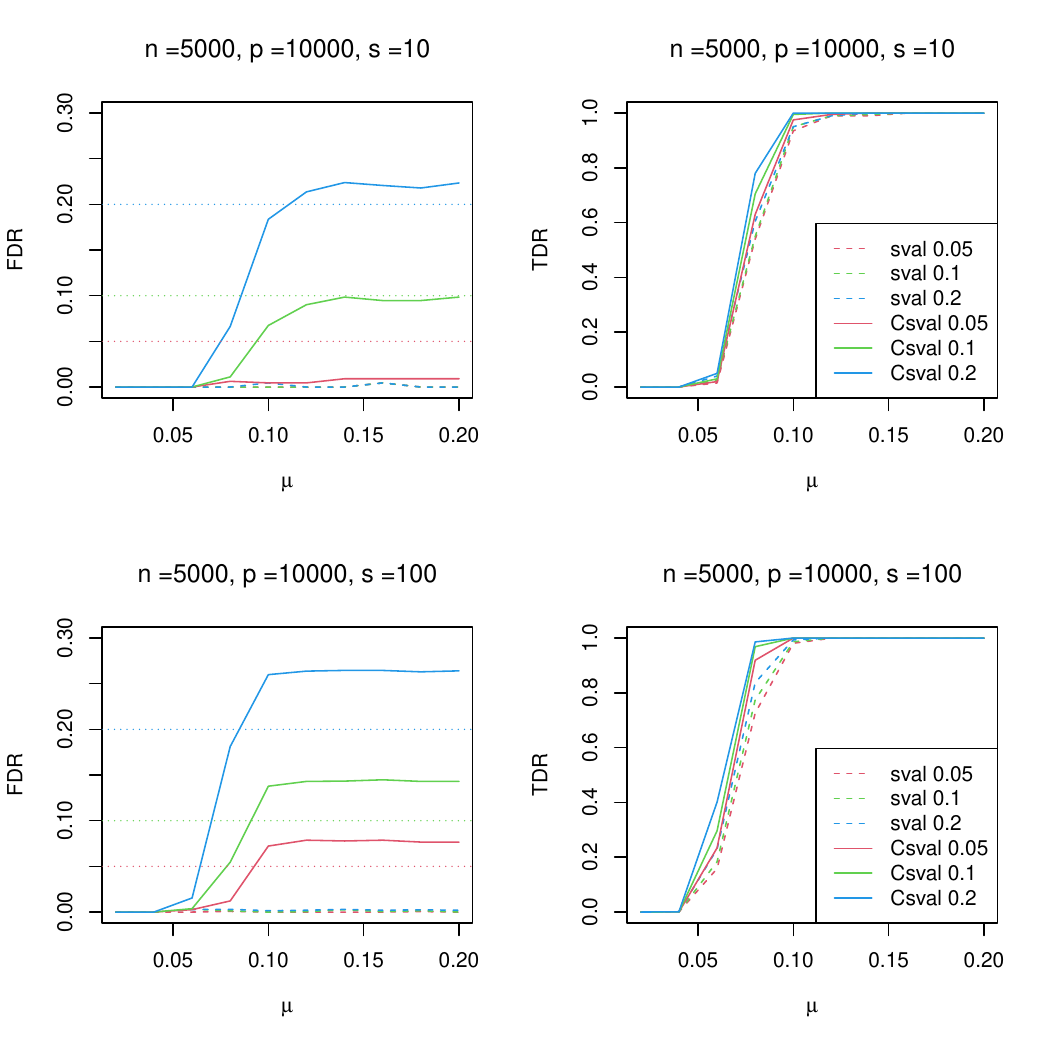}
	\centering
	\caption{Linear Regression with uncorrelated features (i), performance of $s$--value and C$s$--value procedures. }
	\label{fig:regression.Horseshoe_n5000_p10000}
\end{figure}

\noindent Moreover, to assess the impact of the generation of the design matrix, we  perfom simulations under the following two settings: (a)  the design matrix  $X$ is generated  at each replication  and (b)  the design matrix  $X$ is generated once for all replications. The results are presented in Figures \ref{fig:generatedata.independentRegression}, 	\ref{fig:generatedata.dependentRegressionAR1}, 	\ref{fig:generatedata.dependentRegressionAR0}
For the very correlated design (iii), we note that, even though the FDR curves are less smooth that for the other designs,  the FDR still appears to be globally controlled for both procedures.
 
\begin{figure}[!h]
	\includegraphics[scale=0.5]{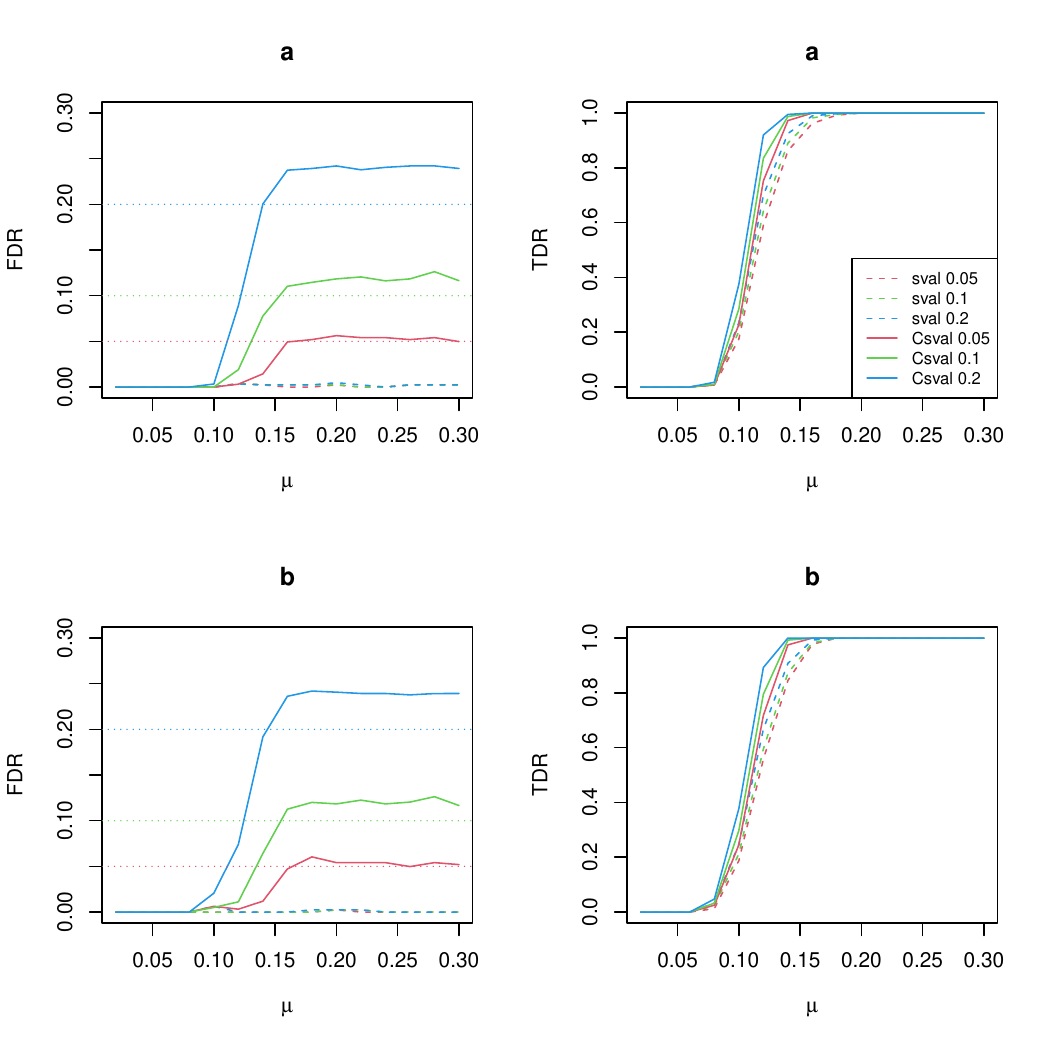}
	\centering
	\caption{Linear Regression with uncorrelated features (i). a: $X$ is generated  at each replication, b: $X$ 		is generated once for all replications. $n=2000; p=4000; s=20$; $20$ replications }
	\label{fig:generatedata.independentRegression}
\end{figure}

\begin{figure}[!h]
	\includegraphics[scale=0.5]{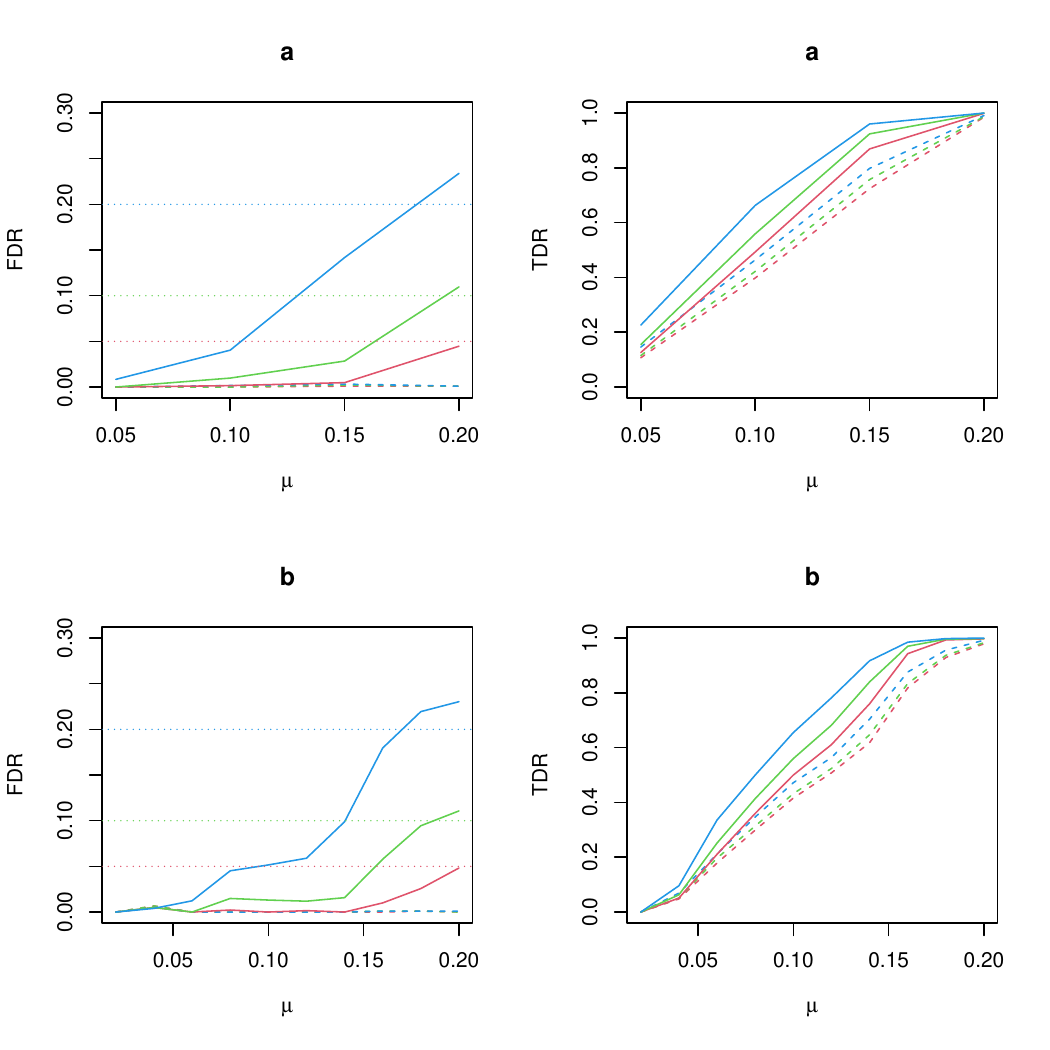}
	\centering
	\caption{Linear Regression with  design (ii). a: $X$ is generated  at each replication, b: $X$
		is generated once for all replications. $n=2000; p=4000; s=20$; 
		 $50$ replications (the scales for $\mu$ are slightly different in the two graphs)}
	\label{fig:generatedata.dependentRegressionAR1}
\end{figure}

\begin{figure}[!hb]
	\includegraphics[scale=0.5]{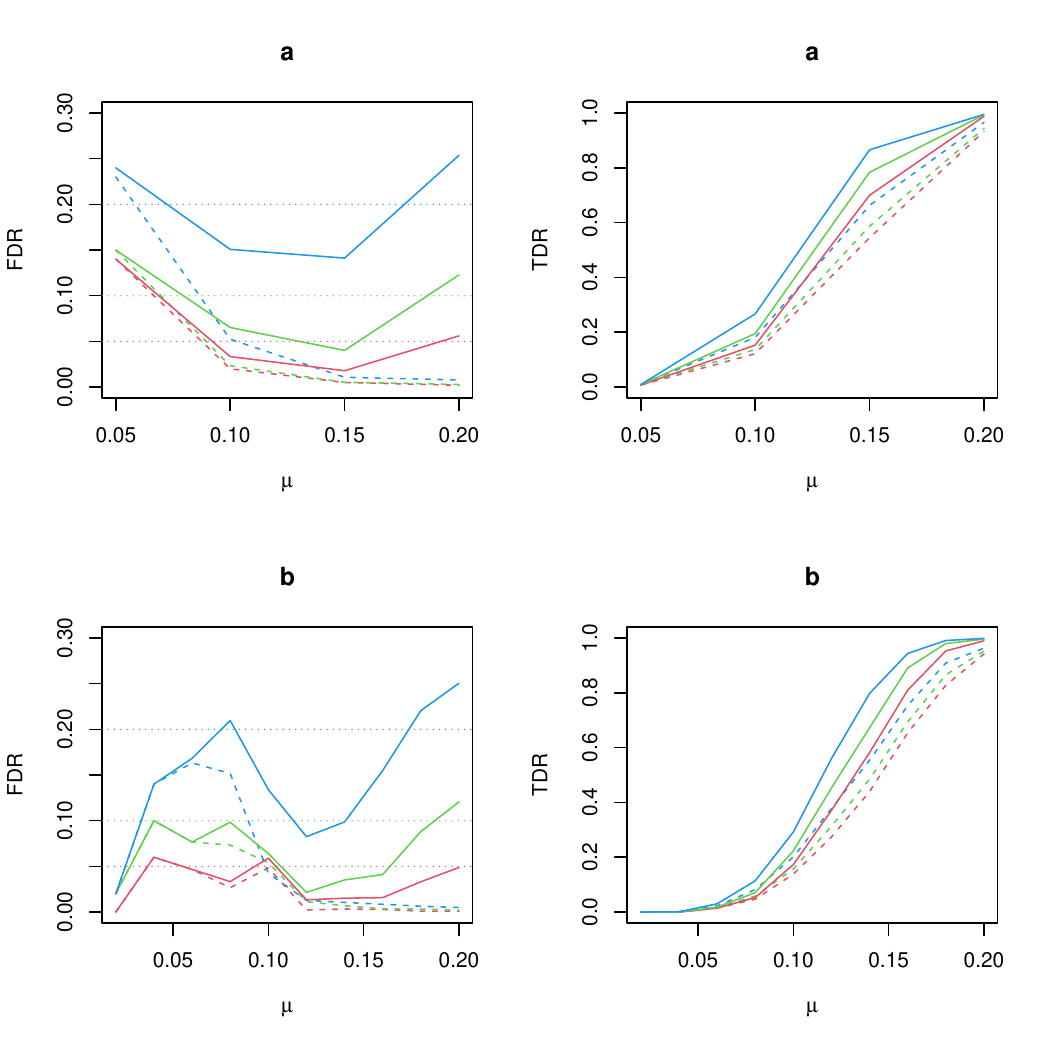}
	\centering
	\caption{Linear Regression with  design (iii). a: $X$ is generated  at each replication, b: $X$ 
		is generated once for all replications. $n=2000; p=4000; s=20$; 
			 $50$ replications. 
			 }
	\label{fig:generatedata.dependentRegressionAR0}
\end{figure}

\pagebreak

\subsection{Further simulations in Gaussian Graphical Models} \label{app_dl}

Throughout most of the paper, we considered a horseshoe prior on the parameter of interest. However, the proposed $s$--value and $Cs$--value procedures are not restricted to this specific choice and can be applied with other priors as well.

In particular, \cite{Chandra26} adopt a Dirichlet-Laplace (DL) prior in the context of  Gaussian graphical model, modeling the precision matrix $\Omega$ via a low-rank plus diagonal decomposition $\Omega= \Lambda \Lambda^T + D$ for some  matrix $\Lambda$ and some diagonal matrix $D$. The Dirichlet-Laplace prior belongs to the class of continuous shrinkage priors that encourage most entries to be close to zero  without producing exact zeros. 
To detect the edges of the graph, that is the non-zero entries of the precision matrix, we propose to apply the $s$--value and $Cs$--value procedures to the posterior samples of each entry of the precision matrix. More precisely, let  $X \sim \mathcal{N}_p(0, \Omega)$, where $\Omega=(w_{ij})_{1 \le i, j \le p}$ is  the symmetric precision matrix. For each entry $w_{ij}$ with $1 \le i<j \le p$, one defines an $s$-value from its marginal posterior distribution given the observed data $X$
\begin{align*}  \label{svalGGM}
	s_{ij}(X) & = 2\left( \Pi(w_{ij}<0 \given X) \wedge  \Pi(w_{ij}>0 \given X)\right).
\end{align*}
The $s$--value and $Cs$--value procedures are then defined analogously to those introduced in the rest of the paper.
We follow the same simulation framework as \cite{Chandra26}, namely the RSM setting (randomly
 generated arbitrarily structured sparse precision matrix) with $n = p = 100$. We apply  
the $s$--value and $Cs$--value procedures with threshold $t$ in $\{0.05,0.1, 0.2\}$. We  evaluate the FDP and TDP, defined as the proportion of errors among the detected edges, and the proportion of true edges that are successfully recovered, over $50$ replications.

\begin{figure}[!h]
	\includegraphics[scale=0.6]{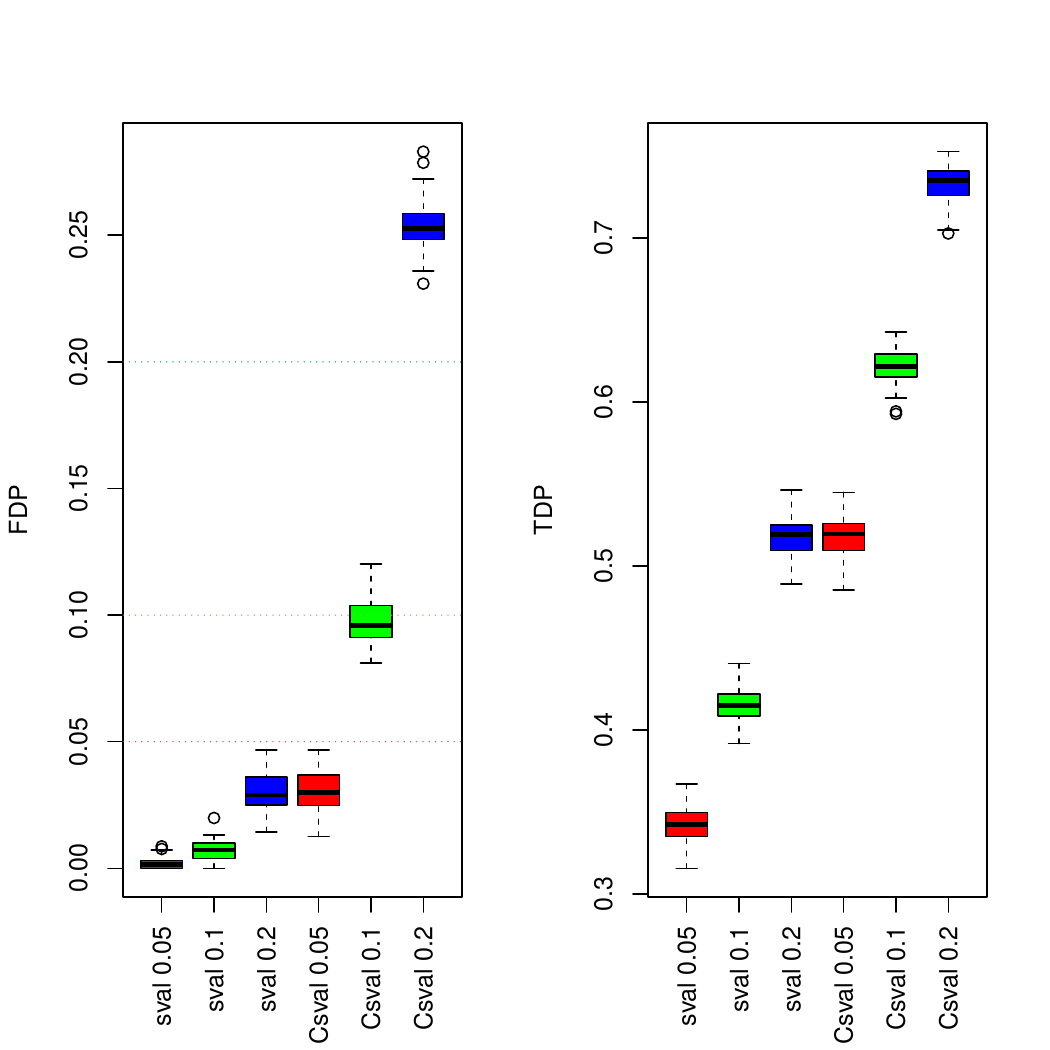}
	\centering
	\caption{RSM setting with $n=p=100$, performance of $s$--value and $Cs$--value procedures with threshold $t \in  \{0.05, 0.1, 0.2\}$; FDP and TDP are evaluated with 50 replications }
	
	\label{fig:SimuChandra26}
\end{figure}

Results are presented in Figure \ref{fig:SimuChandra26}.
We observe that the $s$--value procedure controls the FDR at the target level and is, as expected, slightly conservative. In contrast, the $Cs$--value procedure can sometimes yield an FDR that slightly exceeds the desired threshold. These findings are consistent with what we have already observed for these procedures across different models using a Horseshoe prior on the parameter of interest.
In terms of power, the $Cs$--value procedure with threshold $t=0.1$ provides a good compromise between FDR control and the ability to detect edges.
This confirms the usefulness of the $s$--value and $Cs$--value, which can be applied whenever posterior samples are available, even when deviating from the theoretical setting initially developed for the Horseshoe prior.

\section{Spike-and-Slab LASSO: a comparison and simulations} \label{app_ssl}

In this section we illustrate that $s$--values are very close to inclusion probabilities in the case of a shrinkage prior that is a mixture of two continuous distributions. We do so for the spike--and--slab LASSO prior \cite{rockovageorge17}, which is  a discrete mixture of two Laplace distributions (a continuous `spike' and a continuous `slab'), so one can define $\ell$--values as the posterior probability to belong to the `spike' part.  

\subsection{Definition} 

The spike-and-slab LASSO prior on $\te$ with  parameter $w\in[0,1]$ sets 
\begin{equation} \label{priorlasso}
 \Pi_w \sim \bigotimes_{i=1}^n\, (1-w)G_0(\cdot) + w G_1(\cdot), 
\end{equation} 
where the distributions  $G_0$ and $G_1$ are chosen as
\[ G_0= \text{Lap}(\la_0),\qquad G_1 = 
\begin{cases}
& \text{Lap}(\la_1) \\
\text{ or } \\
& \text{Cauchy}(1/\la_1).
\end{cases} \]
For $G_1$  Laplace, this  leads to the spike and slab LASSO prior of \cite{rockovageorge17}, and to the SSL prior of \cite{cm18} in the case of a $G_1$ Cauchy. 
Let $\ga_0, \ga_1$ denote the Lebesgue--densities of $G_0, G_1$. Typical choices of $\la_0$ and $\la_1$ are $\la_0=n$ and $\la_1$ equal to a constant such as $.5$ or $1$. In sparse settings, $w$ is chosen to be vanishing with $n$; a default choice is $w=1/n$, but often one rather estimates $w$ from the data and sets $w=\hat{w}$, for instance the marginal maximum likelihood estimator (which is expected to have similar properties than $\hat{\ta}$ for the horseshoe prior; see \cite{cm18} for some results in this direction). \\

By Bayes' formula the posterior distribution under \eqref{model} and \eqref{priorlasso} with fixed $w\in[0,1]$ is 
\begin{equation} \label{post2}
\Pi_w[\cdot\given X] 
\sim \bigotimes_{i=1}^n\, (1-\al(X_i))G_{0,X_i}(\cdot) + \al(X_i) G_{1,X_i}(\cdot),
\end{equation}
where $g_k(x)=\phi*G_k(x)=\int \phi(x-u)dG_k(u)$ is the convolution of $\phi$ and $G_k$ at point $x\in\RR$ for $k=0,1$,  the posterior weight $\al(X_i)$ is defined through the function $\al(\cdot)$ given by
\[ \al(x) = \al_w(x) =\frac{w g_1(x)}{(1-w)g_0(x) + w g_1(x)}, \]
and the distribution $G_{k,X_i}$ has density with respect to Lebesgue measure
\begin{align*}
\ga_{k,X_i}(\cdot) & := \frac{\phi(X_i-\cdot) \ga_k(\cdot)}{g_k(X_i)}.
\end{align*}

\subsection{$\ell$-- and $s$--values for the spike--and--slab LASSO}\label{SSLformula}

For standard spike--and--slab priors, (multiple) testing can be based on the $\ell$--value $\ell_i(X)=\Pi[\te_i=0\given X]$.
By analogy to the standard spike--and--slab prior, for the spike--and--slab LASSO it seems natural to define the $\ell$--value $\ell(x;w)$ as the posterior probability that the class `0' is selected, that is
\begin{equation} \label{lvalssl}
 \ell(x;w) = 1-\al(x) = \frac{(1-w)g_0(x)}{(1-w)g_0(x)+wg_1(x)}.
\end{equation} 
Further define, for any real $x$,
\[ g_{0,+}(x) = \int_0^{\infty}\phi(x-u)\ga_0(u)du,\qquad
g_{0,-}(x) = \int_{-\infty}^{0}\phi(x-u)\ga_0(u)du,
 \]
and similarly for $g_{1,+}, g_{1,-}$ replacing $\ga_0$ by $\ga_1$ in the last display. 
By definition $g_{0,+}+g_{0,-}=g_0$ and
\[ \Pi(\te<0\given X) = \frac{(1-w)g_{0,-}(X)+wg_{1,-}(X)}{
(1-w)g_0(X)+wg_1(X)}, \]
from which one deduces the expression of the $s$--value for the spike--and--slab LASSO
\begin{equation} \label{svssl}
 s_L(x;w) = 2 \frac{(1-w)g_{0,-}(|x|)+wg_{1,-}(|x|)}{
(1-w)g_0(x)+wg_1(x)}. 
\end{equation}

\subsection{The $s$--value is an analogue of the $\ell$--value: intuition for spike--and--slab LASSO}

To provide an intuition, we wish to compare the two formulas for the $\ell$--value \eqref{lvalssl} and $s$--value \eqref{svssl}. 
Suppose we have one observation $X=\te+\veps$ and equip $\te$ with a SSL prior with parameters $\la_0, \la_1, w$, with $\la_0=n, \la_1=1$ (say) and $w=1/n$ (say). 

By symmetry we can assume $X\ge 0$, and we distinguish two regimes 
\begin{enumerate}
\item[a)] $X$ `small' or `moderately large' say bounded or growing to infinity   slower than $\sqrt{2\log{n}}$; 
\item[b)] $X$ `very large', growing to infinity faster than $\sqrt{2\log{n}}$.
\end{enumerate}

Since $\la_0$ is very large, we have $g_0(X)\approx \phi(X)$ (see also \eqref{gophi} for a formal bound)  and similarly $g_{0,-}(X)\approx \phi(X)\int_{-\infty}^0\ga_0(u)du=\phi(X)/2$. This means that, for the terms involving the spike density $\gamma_0$ only, the numerators in \eqref{lvalssl} and \eqref{svssl} are nearly the same. 

In regime a),  the terms involving $\ga_1$ are of smaller order since $w$ is very small and $\phi_0(X)$ is not too small yet. This means that in regime a),  \eqref{lvalssl} and \eqref{svssl} are close; also, in this regime, the terms involving $\ga_0$ dominate so that both quantities stay roughly close to $1$, and the procedures do not reject the null hypothesis. In the extreme case of $X=0$, the $s$--value even equals $1$ and the procedure always rejects. 

In regime b) on the other hand, to show that  \eqref{lvalssl} and \eqref{svssl} are close, it suffices to note that $w g_{1,-}(X)\le w\phi(X)$ (bounding $\phi(x-u)$ in the integral by $\phi(x)$ for $u<0$) is negligible compared to $(1-w)g_{0,-}(X)\approx (1-w)\phi(X)\approx \phi(X)$. Thus, both expressions have a numerator of the same order and the same denominator. To show that both procedures reject, one notes that $wg_1(X)\approx w\ga_1(X)$ for large $X$, as the tails of $g_1$ are the same of those of $\ga_1$ (see \cite{js04}, Lemma 1), so $wg_1(X)\gg (1-w)g_0(X)\approx \phi(X)$ for large $X$ and the term $wg_1(X)$ dominates in the denominator and the numerators, of order $\phi(X)$, are negligible compared to it as noted above.
This intuition is formalised by the next Lemma,  
that confirms that both $s_L$-- and $\ell$--value are close to each other when the `spike' is sufficiently marked ($\la_0$ large enough, which holds if one takes e.g. $\la_0=n$) and $w$ is small (which is typical under sparsity assumptions).

\begin{lemma} \label{lemboundssl}
For any real $x$, any $w\in[0,1]$, with 
the notation as in \eqref{lvalssl}--\eqref{svssl}, 
\[ u_1(x)\cdot \ell(x;w) \le s_L(x;w) \le u_2(x)\cdot \ell(x;w),\]
where $u_1(\cdot), u_2(\cdot)$ are defined by, for any $\la_0, \la_1>0$ and $\ga_1(0)=\la_1$,
\begin{align*} 
u_1(x) & = \frac{1}{1+\la_0^{-2}\phi(x)^{-1}}
\frac{\la_0}{\la_0+|x|}(1-\la_0^{-2}), \\
u_2(x) & = \frac{\phi}{g_0}(0)
\left(\frac{\la_0}{\la_0+|x|}+\frac{w}{1-w}\ga_1(0)\sqrt{2\pi} \right).
\end{align*}
\end{lemma}
In particular, the next Corollary confirms that in the regime a) above (i.e. $X=x$ not too large), then $s_L$ and $\ell$ only differ by a multiplicative factor that goes to $1$ as $\la_0\to\infty, w\to 0$. In the regime b), $u_2(x)$ is still bounded from above by a constant close to $1$, and goes to $0$ for extremely large $x$, which just shows that the $s$--value rejects more easily (which is desirable in that case since $X$ is large). 
\begin{corollary} \label{corsv}
Suppose $\phi(x)\ge \la_0^{-1-\delta}$ for some $\delta\in(0,1/2]$. Then, as $\la_0\to\infty$, and for bounded $\la_1>0$,
\[ u_1(x) \ge 1-O(\la_0^{-1/2}),\qquad\  u_2(x) \le 1+ O\left(\frac{\sqrt{\log{\la_0}}}{\la_0}\right) + O\left(w \vee \la_0^{-2}\right).\]
For any real $x$, one has, as $\la_0\to\infty$, 
\[ u_2(x) \le 1+o(1)+O(w).\]
\end{corollary}

The proof of both statements can be found below. \\

Figure \ref{fig:illus_ssl} illustrates the closeness between $s_L$ and $\ell$ as above for $\la_0=10^4$ and $w=0.01$. Analogous comments as above can be made for the $S$--value compared to the $q$--value. While for spike--and--slab LASSO one could consider using both $\ell$--value or $s$--value for multiple testing purposes (and our scheme of proof for the horseshoe could be followed to do so), for the horseshoe or other shrinkage priors that are continuous mixture, the $\ell$--value is not defined, hence the usefulness of the $s$--value for such more complex mixture distributions.

\begin{figure}[!h]
\includegraphics[scale=0.7]{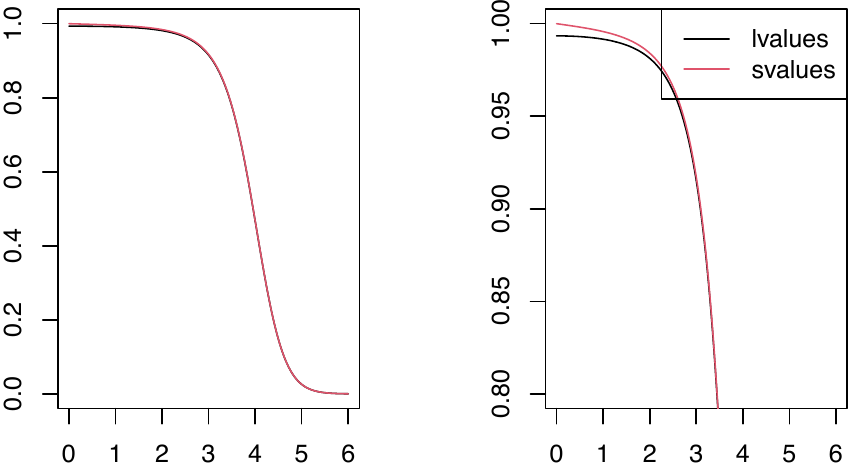}
\centering
	\caption{Comparison of $\ell$--value or $s$--value for the spike--and--slab LASSO with parameters $\la_0=10^4$ and $w=0.01$. The figure on the right is a zoom of the first.}  
	\label{fig:illus_ssl}
\end{figure}

\begin{proof}[Proof of Corollary \ref{corsv}]
Under the assumed condition on $x$, we have, as $\la_0\to\infty$, 
\[\frac{1}{1+\la_0^{-2}\phi(x)^{-1}} \ge (1+\la_0^{\delta-1})^{-1}=1-O(\la_0^{\delta-1}),
 \] 
as well as $|x|\le \sqrt{2\log(\la_0^{1+\delta}/\sqrt{2\pi})}$, so that, as $\la_0\to\infty$, 
\[  \frac{\la_0}{\la_0+|x|} = 1+O(\sqrt{\log{\la_0}}/\la_0). \]
Since $|g_0-\phi|\le \la_0^{-2}$ by Lemma \ref{lem-classic}, we know that $(\phi/g_0)(0)=1/(1+O(\la_0^{-2}))=1+O(\la_0^{-2})$. Putting the previous estimates together gives the first part of the statement.
 
For the second part of the statement, one has  $\la_0/(\la_0+|x|)\le 1$; using again that $(\phi/g_0)(0)=1+O(\la_0^{-2})$ as $\la_0\to\infty$, the result follows.
\end{proof}
\begin{proof}[Proof of Lemma \ref{lemboundssl}]
To prove the upper bound, one compares the numerator of $s_L(x;w)$ to $(1-w)g_0(x)$; we thus wish to bound from above the quantity
\[
\frac{g_{0,-}}{g_0}(|x|)
+\frac{w}{1-w}\frac{g_{1,-}(|x|)}{g_0(x)}. \]
For the last ratio, we use Lemma \ref{lem-classic}, eq. \eqref{gminus}, which implies $g_{1,-}(|x|)\le \ga_1(0)\phi(t)\sqrt{2\pi}/2$ combined with the monotonicity of $\phi/g$, which gives $\phi/g\le (\phi/g)(0)$. In order to bound $g_{0,-}/g_0$, one writes
\begin{align*}
g_{0,-}(x) & = \int_{-\infty}^0 \phi(x) e^{ux-u^2/2} \ga_0(u) du \\
& = \phi(x)  \int_{-\infty}^0 (\la_0/2) e^{-(u-(x+\la_0))^2/2} e^{(x+\la_0)^2/2} du
 =  (\la_0/2) \phi(x)   \frac{\overline{\Phi}(x+\la_0)}{\phi(x+\la_0)}.
\end{align*}
Now Lemma \ref{lem-stand} applied at point $|x|+\la_0$ gives 
\[2(g_{0,-}/g_0)(|x|)\le \frac{\phi}{g_0}(x)\frac{\la_0}{\la_0+|x|}.  \]
The upper-bound of the Lemma follows by using the monotonicity of $\phi/g$.

To prove the lower bound of the Lemma, one first ignores the term with $g_{1,-}$ which is nonnegative; then one uses the identity on $g_{0,-}$ obtained in the first display of the present proof, and combines it with the lower bound in Lemma \ref{lem-stand} to obtain
\[ 2(g_{0,-}/g_0)(|x|)\ge \frac{\phi}{g_0}(x) \la_0\left(\frac1{\la_0+|x|}-
\frac1{(\la_0+|x|)^3} \right)\ge \frac{\phi}{g_0}(x) \frac{\la_0}{\la_0+|x|}
\left(1- \la_0^{-2}\right).
\]
The lower bound of the Lemma follows by using, in the last display, that $|g_0-\phi|\le \la_0^{-2}$ implies
\[ \frac{\phi}{g_0}(x) \ge \frac{\phi(x)}{\phi(x)+\la_0^{-2}}= 
\frac{1}{1+\la_0^{-2}\phi(x)^{-1}}. \qedhere\]
\end{proof}
\vspace{.1cm}

\begin{lemma} \label{lem-classic}
For an arbitrary slab $\ga$ and $g=\phi*\ga$, 
\begin{itemize}
\item the function $g$ is symmetric i.e. $g(-x)=g(x)$ for all $x$;
\item the function $\phi/g$ is decreasing on $[0,\infty)$;
\item for any $x\ge 0$, we have 
\begin{equation} \label{gminus}
 g_-(x)=\int_{-\infty}^0 \phi(x-t)\ga(t)dt \le \ga(0)\phi(x)\sqrt{2\pi}/2;
\end{equation}
\item for any real $x$, we have 
\begin{equation} \label{gophi}
|g_0-\phi|(x)\le \la_0^{-2}.
\end{equation}
\end{itemize}
\end{lemma}
\begin{proof}
The monotonicity of $\phi/g$ (equivalently, $g/\phi$) is established in Lemma 1 in \cite{js04}. The inequality \eqref{gminus} follows by expanding $\phi(x-t)$, see Lemma S--38 in \cite{cr20} (the upper-bound therein has a $\phi(1)$ factor missing at the denominator, so we rewrite here  the inequality for completeness). As in the proof of  Lemma S--38 in \cite{cr20}, for $x\ge 0$,
\[ (g_-/\phi)(x) \le \ga(0)\overline\Phi(x)/\phi(x). \]
As $x\to \overline\Phi(x)/\phi(x)$ is decreasing on $[0,+\infty)$, the last display is bounded by $\ga(0)\overline\Phi(0)/\phi(0)=\ga(0)\sqrt{2\pi}/2$, 
which gives the result.  
 For \eqref{gophi}, see Lemma 17 in \cite{cm18}.
\end{proof}

\begin{lemma} \label{lem-stand}
The survival function $\bar\Phi$ of a $\cN(0,1)$ variable verifies, for any $x>0$,
\[ \frac{1}{x}-\frac{1}{x^3} \le  \frac{\ol{\Phi}}{\phi}(x) \le \frac{1}{x}.\]
\end{lemma}

\subsection{Simulations in the sequence model with the Spike-and-Slab LASSO prior}

We present simulations results for the sequence model  \eqref{model}  with the horseshoe prior distribution \eqref{priorlasso} on the unknown parameter $\theta$. We consider 
$n=10^4$, with sparsity levels $s_n=10$ or $s_n=100$  and constant alternatives $\theta_{0,i} = \mu$ if $1 \le i \le s_n$ and $0$ otherwise.  The signal strength $\mu$ varies over the grid $\{0, 0.5, 1, 1.5, \ldots, 8\}$. 
The spike distribution $G_0$ is chosen to be a Laplace distribution with parameter $\lambda_0=n$ while the  slab distribution $G_1$ is taken either as a Laplace distribution with parameter $\lambda_1=1/2$ or as a quasi-Cauchy distribution (as already done in \cite{cr20}) with density 
$$\gamma_1(x) = (2\pi)^{-1/2}(1-|x| \bar{\Phi}(x)/\phi(x))
$$
We adopt an empirical Bayes approach in which 
the unknown parameter $w$ is estimated by the marginal maximum likelihood method, following \cite{js05} and adapted to the slab distribution (Laplace or quasi-Cauchy). 

We compare several multiple testing procedures. First, we consider the $\ell$--value procedure
$$	\vphi^{\ell}_t(X) = \1\left\{ \ell_i(X) < t \right\},\quad 1\le i\le n.
$$
based on the $\ell$--value $\ell_i(X)=\ell(X_i;w)$ with the  expression $\ell(x;w)$  given in \eqref{lvalssl}. We also consider the corresponding $Cl$--value procedure, defined analogously to the $Cs$--value procedure but using  $\ell$--values instead of  $s$--values. Next, we consider the $s$--value procedure based on the  $s$--values $s_L(x,w)$  defined in \eqref{svssl}, together with the associated $Cs$--value procedure constructed from the  $s_L(x,w)$ values. 
As the Spike-and-Slab LASSO prior allows to evaluate the posterior probability that the class `0' is selected, we also consider the $q$--value introduced by Storey \cite{Storey2003} and defined by 
$q_i(X)=q(X_i; w)$, 
where the quantity $q(x; w)=\Pi[\te_i=0\given |X_i|\ge |x|]$ is defined as
\begin{align*}
q(x;w)&=\
\frac{(1-w)\overline{F}_{g_0}(|x|)}{
	(1-w)\overline{F}_{g_0}(|x|)+w\overline{F}_{g_1}(|x|)}, \\
\overline{F}_{g_k}(x)&= \int_x^{\infty}g_{k}(t)dt,  \qquad k \in \{0,1\}
\end{align*}
xhere $g_k(x)=\phi*G_k(x)=\int \phi(x-u)dG_k(u)$ is the convolution of $\phi$ and $G_k$ at point $x\in\RR$ for $k=0,1$. 
We finaly consider the $S$--value procedure \eqref{Svpro} with 
\[
S(X; w) = 2 \min \left (\Pi(\te<0\given X \ge x), \Pi(\te>0\given X \le x)\right ),
\]
\[
 \Pi(\te<0\given X \ge x) = \frac{(1-w)\overline{F}_{g_0,-}(x)+w\overline{F}_{g_1, -}(x)}{ 	(1-w)\overline{F}_{g_0}(x)+w\overline{F}_{g_1}(x)}, 
 \]
 \[
 \Pi(\te>0\given X \le x) =
 \frac{(1-w)\overline{F}_{g_0,-}(-x)+w\overline{F}_{g_1, -}(-x)}{ 	(1-w)\overline{F}_{g_0}(-x)+w\overline{F}_{g_1}(-x)},
\]
\[
\overline{F}_{g_k,-}(x)= \int_x^{\infty}g_{k_-}(t)dt, \qquad \overline{F}_{g_k,+}(x)= \int_x^{\infty}g_{k_+}(t)dt, \qquad k \in \{0,1\}
\]
and $g_k,-(.),g_k,+(.) $ defined in Appendix \ref{SSLformula}.

The FDR and the True Discovery Rate of each procedure are evaluated empirically with $50$ replications and are presented in Figures \ref{fig:SSLn=10} and \ref{fig:SSLn=100} for the Laplace Slab,  and in Figures \ref{fig:SSLCn=10} and \ref{fig:SSLCn=100} for the quasi-Cauchy Slab.
These simulation results confirm that, under a Spike-and-Slab prior, the $s$--value procedure behaves similarly to the $\ell$--value procedure, the $Cs$--value procedure behaves similarly to $Cl$--value procedure, and the  $S$--value procedure  behaves similarly to the $q$--value procedure.   This close correspondence in the Spike-and-Slab setting motivates the introduction of the $s$-- and $S$--values as counterparts of the $\ell$-- and $q$--values for continuous priors that do not generate exact zeros, and for which the latter quantities cannot be defined.

\begin{figure}[!h]
	\includegraphics[scale=0.9]{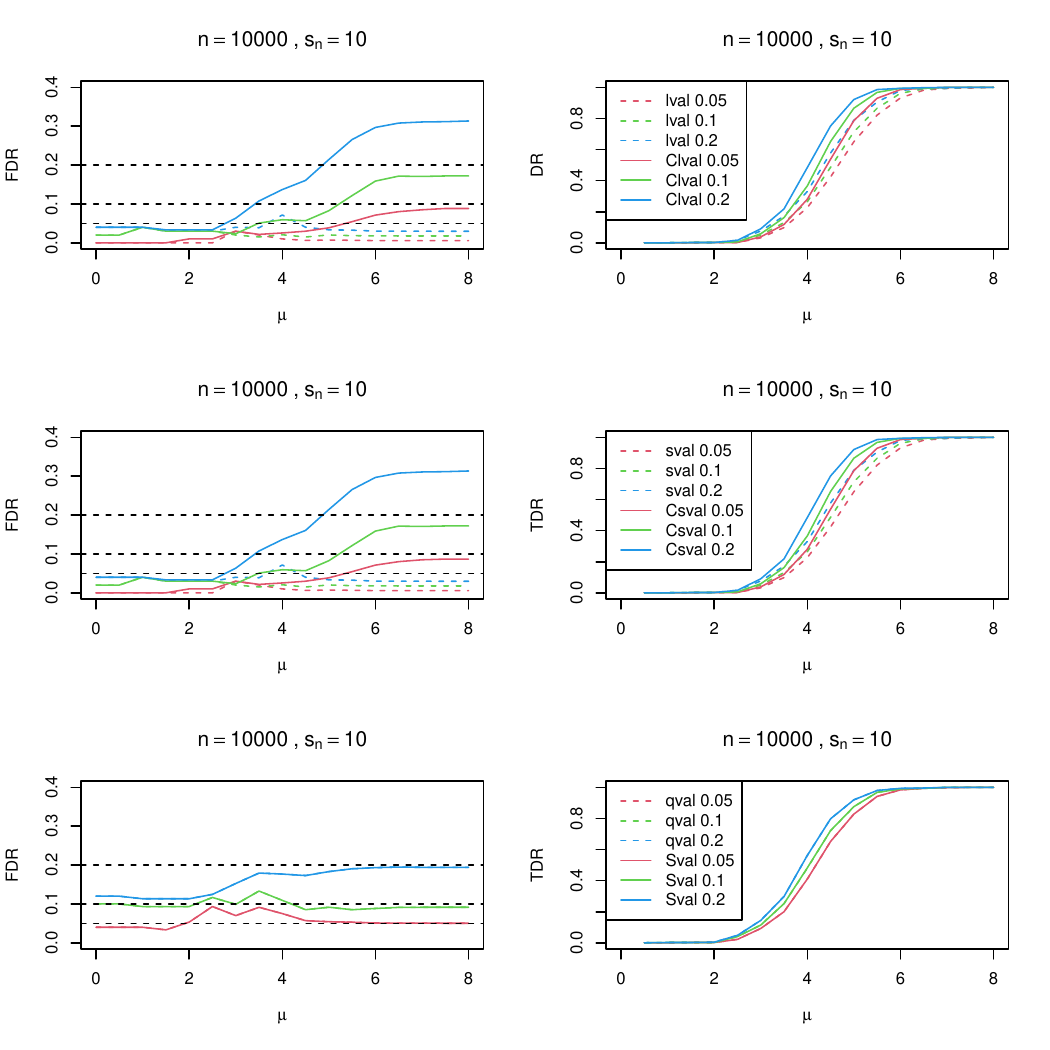}
    \centering
	\caption{Spike-and-Slab prior with a Laplace distribution for the slab. FDR of several multiple testing procedures with threshold $t \in  \{0.05, 0.1, 0.2\}$}
	\label{fig:SSLn=10}

\end{figure}

\begin{figure}[t]
	\includegraphics[scale=0.9]{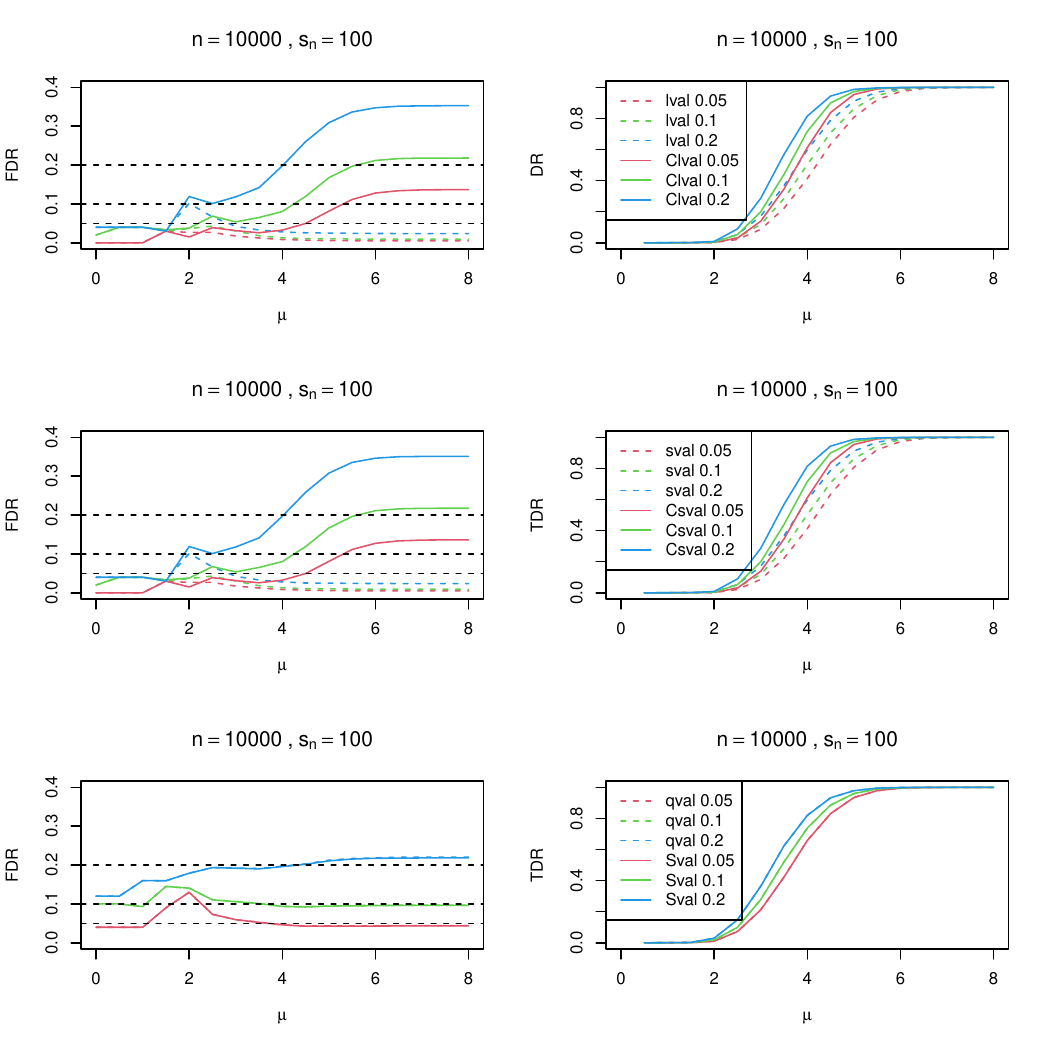}
	\centering
	\caption{Spike-and-Slab prior with a Laplace distribution for the slab. FDR of several multiple testing  procedures with threshold $t \in  \{0.05, 0.1, 0.2\}$}
	\label{fig:SSLn=100}
\end{figure}

\begin{figure}[t]
	\includegraphics[scale=0.9]{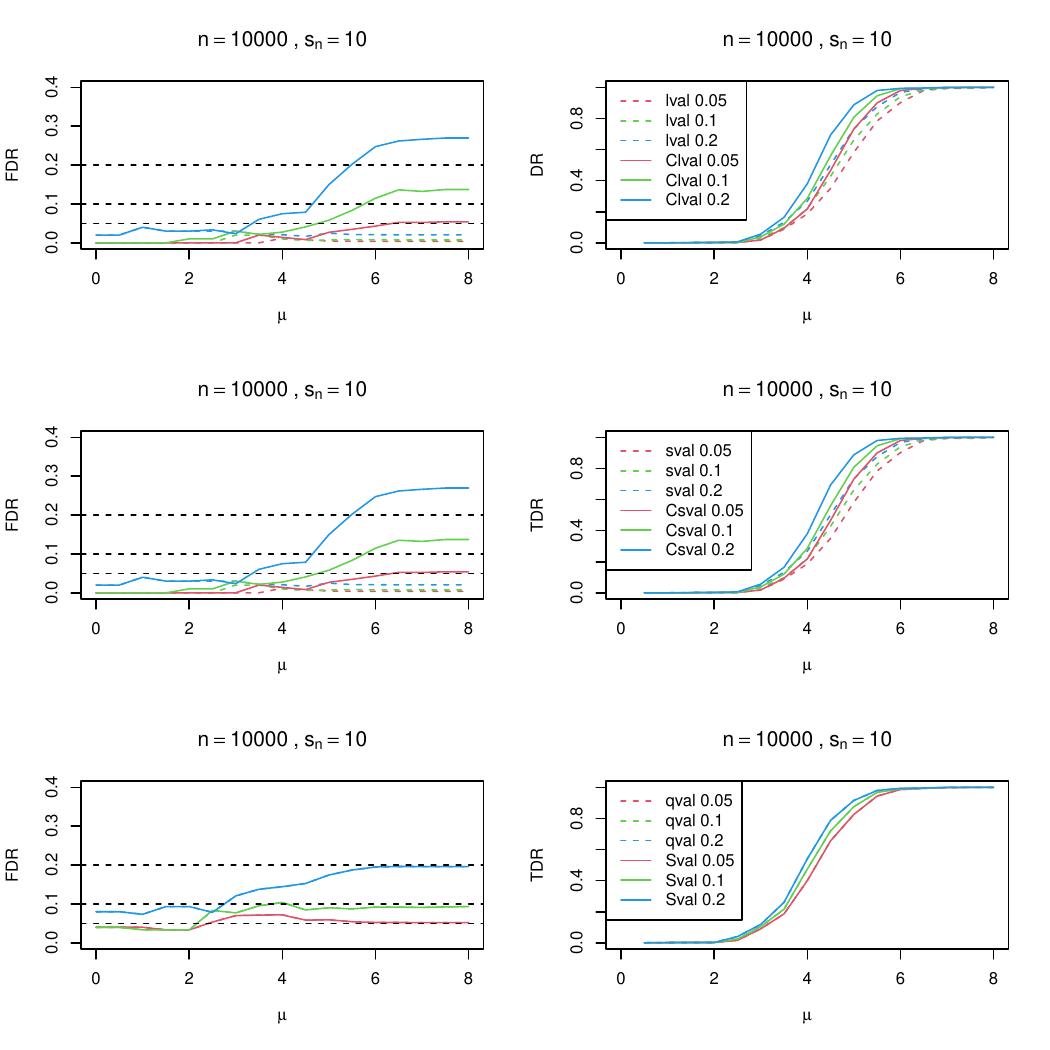}
	\centering
	\caption{Spike-and-Slab prior with a quasi-Cauchy distribution for the slab. FDR of several multiple testing procedures with threshold $t \in  \{0.05, 0.1, 0.2\}$}
	\label{fig:SSLCn=10}
\end{figure}

\begin{figure}[t]
	\includegraphics[scale=0.9]{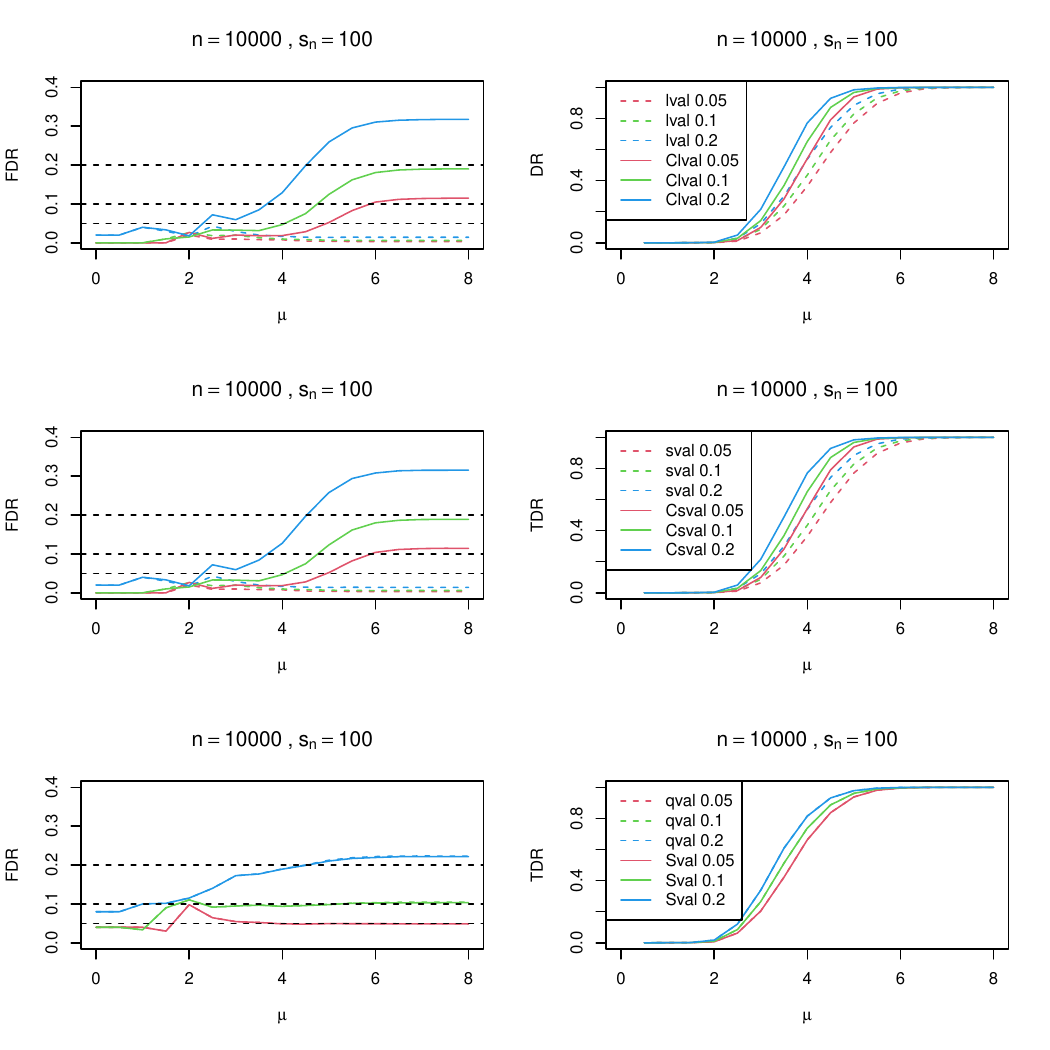}
	\centering
	\caption{Spike-and-Slab prior with a quasi-Cauchy distribution for the slab. FDR of several multiple testing procedures with threshold $t \in  \{0.05, 0.1, 0.2\}$}
	\label{fig:SSLCn=100}
\end{figure}

\end{document}